\documentclass[a4paper,11pt]{article}

\usepackage{ucs}
\usepackage[utf8x]{inputenc}
\usepackage{type1ec}
\usepackage[utf8x]{inputenc}

\usepackage{hyperref}
\usepackage[frenchb]{babel}
\usepackage{fullpage}

\usepackage{bbm}
\usepackage[bbgreekl]{mathbbol}
\usepackage{nccrules}
\usepackage{tocbibind}
\usepackage[intlimits]{amsmath}
\usepackage{amsthm}
\usepackage{amssymb}
\usepackage{mathrsfs}
\usepackage{stmaryrd}
\usepackage[all]{xy}
\newdir{ >}{{}*!/-10pt/\dir{>}}

\newcommand{\Z}{\ensuremath{\mathbb{Z}}}
\newcommand{\Q}{\ensuremath{\mathbb{Q}}}
\newcommand{\R}{\ensuremath{\mathbb{R}}}
\newcommand{\C}{\ensuremath{\mathbb{C}}}

\newcommand{\Tr}{\ensuremath{\mathrm{tr}\,}}

\newcommand{\dd}{\ensuremath{\,\mathrm{d}}}

\newcommand{\angles}[1]{\ensuremath{\langle #1 \rangle}}
\newcommand{\mes}{\ensuremath{\mathrm{mes}}}

\newcommand{\identity}{\ensuremath{\mathrm{id}}}

\newcommand{\Hom}{\ensuremath{\mathrm{Hom}}}

\newcommand{\rightiso}{\ensuremath{\stackrel{\sim}{\rightarrow}}}

\newcommand{\Ker}{\ensuremath{\mathrm{Ker}\,}}

\newcommand{\Ad}{\ensuremath{\mathrm{Ad}\,}}
\newcommand{\ad}{\ensuremath{\mathrm{ad}\,}}

\newcommand{\Gm}{\ensuremath{\mathbb{G}_\mathrm{m}}}

\newcommand{\GL}{\ensuremath{\mathrm{GL}}}
\newcommand{\SO}{\ensuremath{\mathrm{SO}}}
\newcommand{\Spin}{\ensuremath{\mathrm{Spin}}}
\newcommand{\U}{\ensuremath{\mathrm{U}}}

\newcommand{\gl}{\ensuremath{\mathfrak{gl}}}
\newcommand{\so}{\ensuremath{\mathfrak{so}}}

\newcommand{\syp}{\ensuremath{\mathfrak{sp}}}

\newcommand{\Sp}{\ensuremath{\mathrm{Sp}}}

\newcommand{\Mp}{\ensuremath{\widetilde{\mathrm{Sp}}}}

\newcommand{\Lgrp}[1]{\ensuremath{{}^{\mathrm{L}} #1}}

\theoremstyle{plain}
\newtheorem{proposition}{Proposition}[subsection]
\newtheorem{lemma}[proposition]{Lemme}
\newtheorem{theorem}[proposition]{Théorème}
\newtheorem{corollary}[proposition]{Corollaire}

\theoremstyle{definition}
\newtheorem{definition}[proposition]{Définition}
\newtheorem{definition-theorem}[proposition]{Définition-Théorème}
\newtheorem{definition-proposition}[proposition]{Définition-Proposition}

\newtheorem{example}[proposition]{Exemple}

\newtheorem{remark}[proposition]{Remarque}

\newtheorem{conjecture}[proposition]{Conjecture}

\newcommand{\bmu}{\ensuremath{\bbmu}}

\newcommand{\noyau}{\ensuremath{\boldsymbol{\varepsilon}}} 
\newcommand{\Css}{\ensuremath{\mathscr{C}_\text{ss}}} 
\newcommand{\Cssgeo}{\ensuremath{\mathscr{C}_\text{ss}^\text{géo}}} 
\newcommand{\rev}{\ensuremath{\mathbf{p}}} 
\newcommand{\asp}{\ensuremath{\dashrule[.7ex]{2 2 2 2}{.4}}} 
\newcommand{\dualmeta}[1]{\ensuremath{\widehat{\widetilde{#1}}}} 

\newcommand{\Gred}{\ensuremath{0}}  
\newcommand{\Endo}{\ensuremath{!}}  

\title{Le lemme fondamental pondéré pour le groupe métaplectique}
\author{Wen-Wei Li}
\date{}

\begin{document}

\maketitle

\begin{abstract}
  Dans cet article, on énonce une variante du lemme fondamental pondéré d'Arthur pour le groupe métaplectique de Weil, qui sera un ingrédient indispensable de la stabilisation de la formule des traces. Pour un corps de caractéristique résiduelle suffisamment grande, on en donne une démonstration à l'aide de la méthode de descente, qui est conditionnelle: on admet le lemme fondamental pondéré non standard sur les algèbres de Lie. Vu les travaux de Chaudouard et Laumon, on s'attend à ce que cette condition soit ultérieurement vérifiée.
\end{abstract}

\tableofcontents

\section{Introduction}
Cet article s'inscrit dans un programme consistant à stabiliser la formule des traces d'Arthur-Selberg pour le groupe métaplectique de Weil, qui est un revêtement non algébrique $\rev: \Mp(W) \to \Sp(W)$ du groupe symplectique d'un espace symplectique $(W,\angles{|})$. Le formalisme de l'endoscopie elliptique est déjà adapté à ce revêtement dans \cite{Li09}, et les conjectures à la Langlands-Shelstad, à savoir le transfert et le lemme fondamental pour les unités, sont aussi prouvées. Néanmoins, pour stabiliser toute la formule des traces, Arthur \cite{Ar02} a aussi besoin d'une généralisation sophistiquée du lemme fondamental en presque toute place non archimédienne, dite le lemme fondamental pondéré. L'objectif de cet article est de formuler puis prouver une variante du lemme fondamental pondéré pour les groupes métaplectiques. La preuve est conditionnelle lorsque $\dim W > 2$: on admet le lemme fondamental pondéré non standard sur les algèbres de Lie \ref{prop:LF-nonstandard}.

Fixons un corps local $F$ de caractéristique nulle et un caractère unitaire non trivial $\psi: F \to \C^\times$. Soient $G=\Sp(W)$ et $\rev: \tilde{G} = \Mp(W) \to G(F)$ le revêtement métaplectique déterminé par $\psi$. Tout d'abord, il faut étudier les sous-groupes de Lévi de $\tilde{G}$, ou plus précisément les fibres de $\rev$ au-dessus des sous-groupes de Lévi de $G$. Dans \cite{Li09}, on a choisi de travailler avec les revêtements métaplectiques tels que $\Ker(\rev) = \bmu_8 := \{\noyau \in \C^\times : \noyau^8=1\}$. Grâce à ce choix, les Lévi ont une forme très simple comme suit: $\tilde{M} = \prod_{i \in I} \GL(n_i) \times \Mp(W^\flat)$, et cela introduit une structure de récurrence pour l'étude des groupes métaplectiques. De tels revêtements $\rev: \tilde{M} \to M(F)$ s'appellent les groupes de type métaplectique. On définit les données endoscopiques elliptiques de $\tilde{M}$ par composantes, de la façon évidente: en la composante $\Mp(W^\flat)$, on l'a déjà défini en \cite{Li09}; en les composantes $\GL(n_i)$ elles sont tautologiques. Ensuite, on peut définir les données endoscopiques pour $\tilde{G}$ comme les données endoscopiques elliptiques d'un Lévi, pour l'essentiel (voir \ref{prop:endo-Levi}), comme dans le cas de groupes réductifs.

En fait, on définira les données endoscopiques de $\tilde{G}$ en termes du groupe dual $\dualmeta{G} := \Sp(2n,\C)$ muni de l'action galoisienne triviale, où $2n=\dim_F W$, à l'instar de Langlands-Shelstad \cite{LS1}. Or le rôle du centre $Z_{\dualmeta{G}}$, qui fournit des symétries de données endoscopiques dans le cas de groupes réductifs, est remplacé par $Z_{\dualmeta{G}}^\Gred := \{1\}$. On conserve ce symbole non trivial pour $\{1\}$ afin de signaler l'analogie et d'indiquer la généralisation aux groupes de type métaplectique: si $\tilde{M} = \prod_{i \in I} \GL(n_i) \times \Mp(W^\flat)$, $2m = \dim_F W^\flat$, alors on pose
\begin{align*}
  \dualmeta{M} & := \prod_{i \in I} \GL(n_i,\C) \times \Sp(2m,\C), \\
  Z_{\dualmeta{M}}^\Gred & = \prod_{i \in I} \C^\times \times \{1\}.
\end{align*}

Maintenant, supposons que $F$ est non archimédien de caractéristique résiduelle $p$ suffisamment grande par rapport à $G$, et que $\psi$ est de conducteur $\mathfrak{o}_F$. Pour $\tilde{G}$, $\tilde{M}$ comme ci-dessus, fixons une donnée endoscopique elliptique $s_0$ pour $\tilde{M}$, qui donne un groupe endoscopique $M^\Endo$ et une correspondance de classes de conjugaison géométriques semi-simples. On sait aussi définir le facteur de transfert $\Delta$ dans ce cadre. Soit $K$ un sous-groupe hyperspécial de $G(F)$ associé à un réseau autodual dans $(W,\angles{|})$ en bonne position relativement à $M$, alors $\Delta$ et $K$ sont adaptés au sens de \cite{Li09} 5.15. On sait définir les intégrales orbitales pondérées non ramifiées $r^{\tilde{G}}_{\tilde{M},K}(\cdot)$ en les éléments réguliers de $\tilde{G}$. L'intégrale orbitale pondérée endoscopique est définie de façon habituelle

$$ r^{\tilde{G}}_{M^\Endo,K}(\gamma) := \sum_{\delta \in M(F)/\text{conj}} \Delta(\gamma, \tilde{\delta}) r^{\tilde{G}}_{\tilde{M},K}(\tilde{\delta}), \qquad \gamma \in M^\Endo_{G-\text{reg}}(F). $$

Suivant Arthur, on définit un ensemble fini $\mathcal{E}_{M^\Endo}(\tilde{G})$ qui indexe des données endoscopiques elliptiques pour $\tilde{G}$ ``couvrant'' $s_0$, avec des multiplicités. Soit $s \in \mathcal{E}_{M^\Endo}(\tilde{G})$, on note le groupe endoscopique par $G[s]$. On récapitule la situation par le diagramme
$$\xymatrix{
  G[s] \ar@{--}[rr]^{\text{endo.ell.}}_{s} & & \tilde{G} \\
  M^\Endo \ar@{--}[rr]^{\text{endo.ell.}}_{s_0} \ar@{^{(}->}[u]^{\text{Lévi}} & & \tilde{M} \ar@{^{(}->}[u]_{\text{Lévi}}
}$$
où les plongements de sous-groupes de Lévi sont uniques à conjugaison près.

C'est une partie du lemme fondamental pondéré pour les groupes réductifs connexes, prouvé par Chaudouard et Laumon \cite{CL09-1,CL09-2} en étendant la méthode de Ngô, que l'on peut définir les ``fonctions stabilisées'' $s^{G[s]}_{M^\Endo}: M^\Endo_{G[s]-\text{reg}}(F) \to \C^\times$, qui généralisent les intégrales orbitales stables de la fonction caractéristique d'un hyperspécial. En reprenant le formalisme d'Arthur \cite{Ar02}, le lemme fondamental pondéré métaplectique \ref{prop:LF-pondere} est l'égalité

$$ r^{\tilde{G}}_{M^\Endo,K}(\gamma) = \sum_{s \in \mathcal{E}_{M^\Endo}(\tilde{G})} i_{M^\Endo}(\tilde{G}, G[s]) s^{G[s]}_{M^\Endo}(\gamma[s]), $$
où
\begin{itemize}
  \item $i_{M^\Endo}(\tilde{G}, G[s])$ sont des coefficients définis dans \eqref{eqn:coef-meta};
  \item $\gamma[s] := \gamma \cdot z[s]$, où $z[s]$ est un élément d'ordre deux et central dans $M^\Endo(F)$ défini dans \ref{def:torsion}.
\end{itemize}

L'apparition de la ``torsion'' $\gamma \mapsto \gamma[s]$ est plus curieuse. On peut penser qu'elle reflète la différence entre la correspondance de classes par ``Lévi d'un groupe endoscopique'' et celle par ``groupe endoscopique d'un Lévi'' (voir \ref{prop:torsion-correspondance}). D'ailleurs, la démonstration ne marche pas sans cette torsion car elle rend des commutants corrects dans la procédure de descente.

\paragraph{Méthodologie}
L'idée de base est la méthode de descente de Harish-Chandra: on prouve l'égalité cherchée pour $\gamma$ au voisinage d'un élément semi-simple $\epsilon \in M^\Endo(F)$ tel que $M^\Endo_\epsilon$ est quasi-déployé. La méthode est modelée sur la démonstration du lemme fondamental pondéré tordu par Waldspurger \cite{Wa09}. Ainsi, on transforme $r^{\tilde{G}}_{M^\Endo,K}(\gamma)$ en une combinaison linéaire des intégrales orbitales pondérées endoscopiques sur les algèbres de Lie. L'autre côté est transformé en une combinaison linéaire des fonctions stabilisées sur les algèbres de Lie. On compare les deux à l'aide de
\begin{itemize}
  \item le lemme fondamental pondéré sur les algèbres de Lie, ce qui est prouvé par Chaudouard et Laumon \cite{CL09-1,CL09-2} dans le cas de caractéristique positive, auquel le cas de caractéristique nulle se réduit d'après \cite{Wa09-IF}.
  \item le lemme fondamental pondéré non standard, qui reste conjectural à l'heure où cet article est écrit; or on s'attend à ce que la méthode de Chaudouard et Laumon s'y applique également.
\end{itemize}

Deux autres ingrédients sont aussi cruciaux: l'identification des commutants et la descente du facteur de transfert. Heureusement les résultats dans \cite{Li09} sont encore applicables. Enfin, on se ramène à une comparaison des coefficients. Cela fait l'objet du yoga du \S\ref{sec:yoga-centres}, dont la preuve s'inspire de celle d'Arthur \cite{Ar99}.

Remarquons en passant que notre résultat n'est pas conditionnel si $\dim W=2$: le lemme fondamental pondéré non standard qui y intervient peut être vérifié à la main.

\paragraph{Organisation de cet article}
Après un court rappel de conventions et notations dans le \S\ref{sec:notations}, nous définirons les données endoscopiques dans le \S\ref{sec:donnees-meta}. Les cas qui nous occupent sont (1) $\tilde{G}$ un groupe métaplectique et $s$ une donnée endoscopique quelconque pour $\tilde{G}$; (2) $\tilde{M}$ de type métaplectique et $s_0$ une donnée endoscopique elliptique pour $\tilde{M}$. On peut aussi considérer le cas le plus général; on peut même déduire le transfert et le lemme fondamental non pondéré pour les données endoscopiques non elliptiques. Comme ceux-là ne sont pas les propos de cet article, tous sont laissés au lecteur.

Signalons aussi que le corps $F$ et le caractère $\psi$ n'interviennent pas dans les définitions de données endoscopiques et d'ellipticité. Ces notions ont donc un sens pour tout $F$ local de caractéristique nulle ou un corps de nombres. Les facteurs de transfert peuvent être définis pour tout $F$ local de caractéristique nulle.

Dans le \S\ref{sec:int-orb}, nous commençons à supposer que $F$ est non archimédien de caractéristique résiduelle suffisamment grande par rapport à $G$ et $\psi$ est de conducteur $\mathfrak{o}_F$. Le lemme fondamental pondéré y est énoncé. Là encore, c'est possible de le généraliser au cas de groupes de type métaplectique. Dans le \S\ref{sec:endoscopie}, nous rappellerons les définitions de l'endoscopie et le lemme fondamental pondéré (standard et non standard) sur les algèbres de Lie. Il y aura aussi des calculs du coefficient qui apparaissent dans le lemme fondamental pondéré non standard.

Dans le \S\ref{sec:descente-donnees}, nous étudierons la situation après la descente. Il faut identifier des commutants connexes à certains groupes endoscopiques, à l'opération \ref{def:bar} près qui remplace les facteurs $\SO$ impairs par $\Sp$. On l'a traité dans \cite{Li09}, mais ici la situation se complique à cause d'une construction d'Arthur.

L'argument dans \cite{Wa09} est repris dans le \S\ref{sec:descente-int}. Grâce à la descente et aux divers lemmes fondamentaux pondérés sur les algèbres de Lie, les deux côtés de l'égalité du lemme fondamental pondéré sont transformés en des combinaisons linéraires des fonctions stabilisées sur les algèbres de Lie. On en introduira un ensemble d'indices $E^\natural$. Dans le \S\ref{sec:comparaison}, il sera démontré que les coefficients pour les deux côtés coïncident comme fonctions définies sur $E^\natural$, d'où le lemme fondamental pondéré.

\paragraph{Remerciements}
J'ai l'agréable devoir de remercier Jean-Loup Waldspurger pour ses remarques très pertinentes et pour sa lecture attentive du manuscrit.

\section{Notations et conventions}\label{sec:notations}
\paragraph{Les groupes métaplectiques}
Soient $F$ un corps local de caractéristique nulle et $\psi: F \to \C^\times$ un caractère unitaire non trivial. Nous conservons les notations de \cite{Li09}. En particulier, soit $(W,\angles{|})$ un $F$-espace symplectique de dimension $2n$, notons $G := \Sp(W)$ le groupe symplectique et $\rev: \tilde{G} = \Mp(W) \to G(F)$ le revêtement métaplectique à huit feuillets, i.e. $\Ker(\rev) = \bmu_8 = \{\noyau \in \C^\times : \noyau^8=1 \}$. Si $M \subset G$ est un sous-groupe de Lévi, nous noterons $\tilde{M} := \rev^{-1}(M(F))$ et $\rev: \tilde{M} \to M(F)$ le revêtement ainsi induit. Toute construction des objets dits métaplectiques dans cet article dépendra du choix $\psi$.

Soit $n \in \Z_{\geq 0}$, le symbole $\SO(2n+1)$ désigne toujours le groupe orthogonal spécial impair déployé.

\paragraph{Groupes réductifs}
Soient $S$ un schéma et $M$ un $S$-schéma en groupes raisonnable (voir \cite{SGA3} Exposé $\mathrm{VI}_{\mathrm{B}}$ \S 3), on note $M^0$ sa composante neutre. Soient $F$ un corps et $M$ un $F$-groupe réductif connexe. Les normalisateurs (resp. commutants, commutants connexes) dans $M$ sont notés $N_M(\cdot)$ (resp. $Z_M(\cdot)$, $Z_M(\cdot)^0$). Le centre (resp. centre connexe) de $M$ est noté $Z_M$ (resp. $Z_M^0$). Si $m \in M(F)$, on écrit aussi $M^m := Z_M(m)$ et $M_m := Z_M(m)^0$. La classe de conjugaison de $m$ dans $M(F)$ est notée $\mathcal{O}^M(m)$. L'ensemble des classes de conjugaison géométriques semi-simples dans $M$ rencontrant $M(F)$ est noté $\Cssgeo(M(F))$.

Soit $T$ un $F$-tore, notons $X_*(T) := \Hom(\Gm, T)$; il est en dualité avec $X^*(T)$. Lorsqu'il y en a besoin d'indiquer le corps de base, on les notera $X_*(T)_F$ et $X^*(T)_F$.. Notons $X^*(M) := \Hom(M, \Gm)$ et $\mathfrak{a}_M := \Hom(X^*(M),\R)$. Il y en a une autre interprétation: notons $A_M$ le plus grand $F$-tore déployé dans $Z_M$, alors la restriction $X^*(M) \to X^*(A_M)$ induit un isomorphisme $X_*(A_M) \otimes_\Z \R \rightiso \mathfrak{a}_M$.

Soient $G$ un $F$-groupe réductif connexe et $M$ un sous-groupe de Lévi. L'ensemble des sous-groupes de Lévi de $G$ contenant $M$ est désigné par $\mathcal{L}^G(M)$. On désigne par $\mathcal{P}^G(M)$ l'ensemble des sous-groupes paraboliques de $G$ ayant $M$ comme composante de Lévi. On pose $W^G(M) := N_G(M)(F)/M(F)$.

On a la restriction $X^*(G) \to X^*(M)$ ainsi que l'inclusion $A_G \hookrightarrow A_M$. Ces deux applications induisent une suite exacte courte $W^G(M)$-équivariante, scindée canoniquement
$$ 0 \to \mathfrak{a}_G \to \mathfrak{a}_M \leftrightarrows \mathfrak{a}^G_M \to 0. $$

Le revêtement simplement connexe du groupe dérivé de $G$ est noté $\pi: G_\text{SC} \to G$; on note $M_\text{sc} := \pi^{-1}(M)$.

Soit $E$ une $F$-algèbre commutative de dimension finie. Soit $G$ un $E$-groupe. Par abus de notation, on omet la restriction des scalaires relativement à $E/F$ et on regarde $G$ comme un $F$-groupe.

\paragraph{Corps locaux, la décomposition de Jordan topologique}
Soit $F$ un corps local, le groupe de Galois absolu est noté $\Gamma_F$ et le groupe de Weil absolu est noté $W_F$. Si $F$ est non archimédien, le sous-groupe d'inertie de $W_F$ est noté $I_F$; on note $\mathfrak{o}_F$ l'anneau des entiers de $F$ et $\mathfrak{p}_F$ son idéal maximal.

Supposons $F$ non archimédien de caractéristique résiduelle $p$. Soit $G$ un $F$-groupe réductif connexe. On dira que $p$ est suffisamment grand par rapport à $G$ si la minoration \cite{Wa08} 4.4 (H1) est satisfaite et $p>2$. Dans ce cas-là, on peut définir les éléments topologiquement unipotents (resp. nilpotents) dans $G(F)$ (resp. dans $\mathfrak{g}(F)$). L'exponentielle définit un homéomorphisme de l'espace des éléments topologiquement nilpotents sur celui des éléments topologiquement unipotents (voir \cite{Wa08} 4.3 et appendice B).

Un élément $x \in G(F)$ est dit compact si le sous-groupe engendré par $x$ est d'adhérence compacte. Un tel élément $x \in G(F)$ admet une unique décomposition de Jordan $x = x_\text{tu} x_{p'} = x_{p'} x_\text{tu}$, où $x_{p'}$ est d'ordre fini premier à $p$ et $x_\text{tu}$ est topologiquement unipotent. Il existe un unique $X \in \mathfrak{g}(F)$, qui est topologiquement nilpotent, tel que $\exp(X)=x_\text{tu}$. De plus, $x_\text{tu}$ et $x_{p'}$ appartiennent à l'adhérence du sous-groupe engendré par $x$.

\paragraph{L-groupes}
Pour les groupes algébriques complexes, on confond systématiquement le schéma en groupe et la variété formée de ses $\C$-points. Une donnée de L-groupe pour un $F$-groupe réductif connexe $M$ signifie les données suivantes
\begin{itemize}
  \item un torseur intérieur $\phi: M \times_F \bar{F} \rightiso M^* \times_F \bar{F}$, où $M^*$ est un $F$-groupe réductif quasi-déployé;
  \item une paire de Borel $(T^*, B^*)$ de $M^*$ définie sur $F$;
  \item un $\C$-groupe réductif $\hat{M}$ muni d'une paire de Borel $(\hat{T},\hat{B})$;
  \item une action $\rho$ de $\Gamma_F$ sur $\hat{M}$ qui laisse $(\hat{T},\hat{B})$ invariante.
  \item un isomorphisme $\Gamma_F$-équivariant entre les données radicielles basées $\Psi(M^*,T^*,B^*)^\vee \rightiso \Psi(\hat{M},\hat{T},\hat{B})$, où $\Psi(\cdots)^\vee$ désigne le dual.
\end{itemize}
C'est possible de rigidifier certains de ces choix en fixant des $F$-épinglages; nous ne le faisons pas dans cet article. Il existe toujours une donnée de L-groupe pour $M$, ce que l'on fixe. On introduit ainsi le L-groupe $\Lgrp{M} := \hat{M} \rtimes W_F$.

Supposons maintenant que $M$ est un sous-groupe de Lévi de $G$. Fixons $P_0 \in \mathcal{P}^G(M)$. On dira que les données de L-groupes pour $G$ et $M$ sont compatibles si elles vérifient les conditions suivantes:
\begin{itemize}
  \item le torseur intérieur $\phi: G \to G^*$ se restreint en celui pour $M$, disons $\phi|_M: M \to M^*$, tel que $P_0^* := \phi(P_0)$ est défini sur $F$;
  \item notons $(T^*,(B^M)^*)$ la paire de Borel pour $M^*$, la paire de Borel pour $G^*$ est $(T^*, B^*)$ où $B^*$ est l'unique sous-groupe de Borel tel que $(B^M)^* \subset B^* \subset P_0^*$;
  \item $\hat{M} \subset \hat{G}$, les actions galoisiennes étant compatibles;
  \item notons $(\hat{T}, \hat{B}^M)$ la paire de Borel pour $\hat{M}$, la paire de Borel pour $\hat{G}$ est de la forme $(\hat{T},\hat{B})$;
  \item à $P_0^*$ est associé l'ensemble de ses racines simples $\Delta_{P_0^*}$ qui correspond par dualité à l'ensemble $\Delta_{\hat{P}_0}$, où $\hat{P}_0 \in \mathcal{P}^{\hat{G}}(\hat{M})$, tel que $\hat{B}^M \subset \hat{B} \subset \hat{P}_0$;
  \item les isomorphismes $\Psi(M^*,T^*,(B^M)^*)^\vee \rightiso \Psi(\hat{M},\hat{T},\hat{B}^M)$ et $\Psi(G^*, T^*, B^*)^\vee \rightiso \Psi(\hat{G}, \hat{T}, \hat{B})$ sont compatibles.
\end{itemize}

De tels choix sont possibles. Ces choix induisent une inclusion canonique $\Lgrp{M} \hookrightarrow \Lgrp{G}$. Cela permet aussi de définir une application injective $\mathcal{L}^G(M) \to \mathcal{L}^{\hat{G}}(\hat{M})$. Son image consiste des $\hat{L} \in \mathcal{L}^{\hat{G}}(\hat{M})$ tels qu'il existe $\hat{P} \in \mathcal{P}^{\hat{G}}(\hat{L})$ tels que $\hat{L}$ et $\hat{P}$ sont tous $\Gamma_F$-stables. Cf. \cite{Ar99} \S 1.

\section{Endoscopie métaplectique}\label{sec:donnees-meta}
Le corps local $F$ et le caractère $\psi: F \to \C^\times$ sont fixés dans cette section.

\subsection{Données endoscopiques}
Soit $(W,\angles{|})$ un $F$-espace symplectique de dimension $2n$, $n \in \Z_{\geq 0}$. Posons $G := \Sp(W)$. Un sous-groupe de Lévi $M$ correspond à des sous-espaces de $W$
$$ (\ell^i, \ell_i)_{i \in I}, W^\flat $$
où
\begin{itemize}
  \item $I$ est un ensemble fini;
  \item pour tout $i$, $(\ell^i \oplus \ell_i, \angles{|})$ est un $F$-espace symplectique dont $\ell^i$ et $\ell_i$ sont des lagrangiens;
  \item $(W^\flat,\angles{|})$ est un $F$-espace symplectique;
  \item on a une somme directe orthogonale $W = \bigoplus_{i \in I} (\ell^i \oplus \ell_i) \oplus W^\flat$.
\end{itemize}

Posons $n_i := \dim \ell_i$, alors
$$ M = \prod_{i \in I} \GL(n_i) \times \Sp(W^\flat). $$

Les sous-groupes de Lévi de $G$ sont ainsi paramétrés, à conjugaison par $G(F)$ près, par les données $(I, (n_i)_{i \in I})$ où
\begin{itemize}
  \item $I$ est un ensemble fini,
  \item $(n_i)_{i \in I} \in \Z_{\geq 1}^I$ à permutation près,
\end{itemize}
telles que $m := n - \sum_{i \in I} n_i \geq 0$.

Fixons un sous-groupe de Lévi $M$ associé à la suite de sous-espaces comme précédemment. Soit $\rev: \tilde{G} \to G(F)$ le revêtement métaplectique, alors (voir \cite{Li09} \S 5.4) $\rev: \tilde{M} \to M(F)$ est canoniquement isomorphe au revêtement
\begin{gather}\label{eqn:type-metaplectique}
  \tilde{M} = \prod_{i \in I} \GL(n_i) \times \Mp(W^\flat) \xrightarrow{(\identity, \rev)} \prod_{i \in I} \GL(n_i) \times \Sp(W^\flat),
\end{gather}
où la restriction de $\rev$ à la composante $\Mp(W^\flat)$ est encore notée par $\rev$. Remarquons que, tandis que le choix des espaces symplectiques $(W^\flat,\angles{|})$, $(W,\angles{|})$ n'affecte pas les groupes à isomorphisme près pourvu qu'ils aient les bonnes dimensions, il affecte les plongements $M \hookrightarrow G$ et $\tilde{M} \hookrightarrow \tilde{G}$. Par ailleurs, selon \cite{Li09}, les candidats de sous-groupes hyperspéciaux et de facteurs de transfert dépendent aussi de la forme symplectique. S'il n'y a pas de telles dépendances à craindre, on écrira $G=\Sp(2n)$ et $\tilde{G}=\Mp(2n)$, idem pour $M, \tilde{M}$.

\begin{definition}
  Les revêtements $\rev: \tilde{M} \to M(F)$ de la forme \eqref{eqn:type-metaplectique} sont dits de type métaplectique. Par abus de notation, on dit aussi que $\tilde{M}$ est un groupe de type métaplectique. Ici c'est sous-entendu que l'on a choisi $(W^\flat, \angles{|})$.
\end{definition}

Si $\rev: \tilde{L} \to L(F)$ est de type métaplectique et $M \subset L$ est un sous-groupe de Lévi, alors la restriction $\rev: \widetilde{M} \to M(F)$ est aussi de type métaplectique de façon évidente. On dit aussi que $\tilde{M}$ est un sous-groupe de Lévi de $\tilde{L}$. Les notions de distributions spécifiques et de fonctions anti-spécifiques (cf. \cite{Li09} \S 2.1) s'adaptent à ce cadre sans difficulté.

Rappelons que le dual de Langlands de $\GL(n)$ est le groupe complexe $\GL(n,\C)$ et celui de $\SO(2n+1)$ est $\Sp(2n,\C)$, tous munis de l'action galoisienne triviale. La définition ci-dessous reflète l'analogie entre $\Mp(W)$ et $\SO(2n+1)$.

\begin{definition}
  Soit $\tilde{M} = \prod_{i \in I} \GL(n_i) \times \Mp(W^\flat)$ un groupe de type métaplectique avec $\dim W^\flat = 2m$. Posons
  \begin{align*}
    \dualmeta{M} & := \prod_{i \in I} \GL(n_i,\C) \times \Sp(2m,\C), \\
    Z_{\dualmeta{M}}^\Gred & = \prod_{i \in I} \C^\times \times \{1\},
  \end{align*}
  où on identifie $\C^\times$ au centre de $\GL(n_i,\C)$ pour chaque $i \in I$. Il y a une bijection naturelle entre les classes de conjugaison des sous-groupes de Lévi de $\dualmeta{M}$ et celles de $M$. Munissons $\dualmeta{M}$ de l'action triviale de $\Gamma_F$.
\end{definition}

Donnons une définition \textit{ad hoc} des données endoscopiques de $\tilde{M}$; une interprétation plus naturelle sera donnée dans \ref{prop:endo-Levi}.
\begin{definition}
  Avec les notations précédentes, une donnée endoscopique de $\tilde{M}$ est une classe dans
  $$ \mathcal{E}(\tilde{M}) := Z_{\dualmeta{M}}^\Gred \backslash \{s \in \dualmeta{M} : s \text{ est semi-simple} \} / \text{conj}. $$

  Écrivons $s = ((s_i)_{i \in I}, s^\flat)$. Chaque $s_i$ détermine une donnée endoscopique de $\GL(n_i)$, à laquelle est associé le groupe endoscopique $M_i^\Endo$. Supposons que les valeurs propres de $s^\flat$ sont
  $$ \underbrace{a_1^{\pm 1}}_{k_1 \text{ fois}}, \ldots, \underbrace{a_r^{\pm 1}}_{k_r \text{ fois}}, \underbrace{+1}_{2m' \text{ fois}}, \underbrace{-1}_{2m'' \text{ fois}} $$
  où $a_1, \ldots, a_r \neq \pm 1$ et $a_i \neq a_j^{\pm 1}$ si $i \neq j$. Définissons le groupe endoscopique associé comme
  $$ M^\Endo := \prod_{i \in I} M_i^\Endo \times \prod_{j=1}^r \GL(k_j) \times \SO(2m'+1) \times \SO(2m''+1). $$
\end{definition}

\begin{remark}
  Si $\tilde{M} = \prod_{i \in I} \GL(n_i) \times \bmu_8$, i.e. s'il n'y a pas de revêtement, alors la définition ci-dessus se réduit à l'endoscopie pour le groupe réductif connexe $M = \prod_{i \in I} \GL(n_i)$. En général, on se ramène aussitôt à l'étude de l'endoscopie pour $\GL$ et de l'endoscopie pour $\Mp$.
\end{remark}

\begin{definition}
  On dit que $s \in \mathcal{E}(\tilde{M})$ est elliptique si $Z_{\dualmeta{M}}(s)$ (bien défini à conjugaison près) n'appartient à aucun sous-groupe de Lévi propre de $\dualmeta{M}$. L'ensemble des données endoscopiques elliptiques pour $\tilde{M}$ est noté $\mathcal{E}_\text{ell}(\tilde{M})$.
\end{definition}

Le résultat suivant est immédiat.
\begin{proposition}\label{prop:endo-ell}
  Soient $\tilde{M} = \prod_{i \in I} \GL(n_i) \times \Mp(W)$ et $s = ((s_i)_{i \in I}, s^\flat) \in \mathcal{E}(\tilde{M})$, alors $s$ est elliptique si et seulement si $s_i$ est central dans $\GL(n_i,\C)$ pour chaque $i \in I$ et $(s^\flat)^2 = 1$. Dans ce cas-là, on a
  \begin{gather}\label{eqn:nabla-desc-ell}
    M^\Endo = \prod_{i \in I} \GL(n_i) \times \SO(2m'+1) \times \SO(2m''+1), \quad m'+m''=m .
  \end{gather}

  Par conséquent, $\mathcal{E}_\text{ell}(\tilde{M})$ est en bijection avec $\{(m',m'') \in \Z_{\geq 0}^2 : m' + m'' = m \}$: la multiplicité de $1$ (resp. $-1$) dans les valeurs propres de $s^\flat$ est égale à $2m'$ (resp. $2m''$).
\end{proposition}

\begin{remark}
  Pour le cas $\tilde{M} = \Mp(W)$, on se ramène au formalisme posé dans \cite{Li09}.
\end{remark}

\begin{proposition}\label{prop:endo-Levi}
  Soient $\tilde{L}$ un groupe de type métaplectique et $s \in \mathcal{E}(\tilde{L})$. Alors il existe un sous-groupe de Lévi $\tilde{M}$ et $s_M \in \dualmeta{M}$ tels que
  \begin{itemize}
    \item $s_M$ détermine une donnée endoscopique elliptique pour $\tilde{M}$ dont le groupe endoscopique $M^\Endo$ est isomorphe à $L^\Endo$;
    \item via l'inclusion $\dualmeta{M} \hookrightarrow \dualmeta{L}$, $s_M$ détermine la donnée endoscopique $s$ pour $\tilde{L}$.
  \end{itemize}
  Le Lévi $\tilde{M}$ est unique à conjugaison près. La donnée endoscopique elliptique pour $\tilde{M}$ déterminée par $s_M$ est unique. On en déduit une application surjective
  \begin{gather}\label{eqn:endo-Levi}
    \mathcal{E}(\tilde{L}) \to \bigsqcup_{M / \mathrm{conj}} \mathcal{E}_\mathrm{ell}(\tilde{M}).
  \end{gather}
\end{proposition}
\begin{proof}
  On se ramène aussitôt aux cas $L = \GL(n)$ ou $L = \Sp(2n)$. Il suffit de traiter le deuxième cas. Soit $s \in \mathcal{E}(\tilde{L})$; on en prend un représentant dans $\dualmeta{L}$, noté encore par $s$. Il existe un Lévi de $\dualmeta{L}$, noté $\dualmeta{M}$, tel que l'on peut écrire
  \begin{align*}
    \dualmeta{M} & = \prod_{i \in I} \GL(n_i,\C) \times \Sp(2m,\C), \\
    s & = ((s_i)_{i \in I}, s^\flat) \in \dualmeta{M},
  \end{align*}
  où
  \begin{itemize}
    \item les valeurs propres de $s^\flat \in \Sp(2m,\C)$ sont $\pm 1$;
    \item $s_i \in \C^\times = Z_{\GL(n_i,\C)}$ et $s_i \neq \pm1$ pour tout $i$;
    \item $s_i \neq s_j^{\pm 1}$ si $i \neq j$.
  \end{itemize}
  Alors $\dualmeta{M}$ est unique à conjugaison près. On prend $\tilde{M}$ un Lévi de $\tilde{L}$ dual de $\dualmeta{M}$ et on prend $s_M := s$. La donnée endoscopique pour $\tilde{M}$ déterminée par $s_M$ est elliptique et $M^\Endo = L^\Endo$ d'après \ref{prop:endo-ell}. Un tel élément dans $\mathcal{E}_\text{ell}(\tilde{M})$ est déterminé par la multiplicité de $+1$ (resp. $-1$) dans les valeurs propres de la composante en $\Sp(2m,\C)$ de $s_M$. Or c'est égal à la multiplicité de $1$ (resp. $-1$) dans les valeurs propres de $s$, d'où l'unicité.

  Montrons la surjectivité de \eqref{eqn:endo-Levi}. Fixons une donnée endoscopique dans $\mathcal{E}_\text{ell}(\tilde{M})$ déterminée par un élément $s_M \in \dualmeta{M}$. On peut prendre $s=s_M t$ où $t \in Z_{\dualmeta{M}}^\Gred$ est en position générale de sorte que $s$ vérifie les conditions précédentes relativement à $\tilde{M}$. La surjectivité s'ensuit.
\end{proof}

Remarquons que $\dualmeta{M}$ admet la description comme le commutant dans $\dualmeta{L}$ du centre connexe de $Z_{\dualmeta{L}}(s)$. La théorie que nous élaborerons ne dépend que de l'image de $s$ sous \eqref{eqn:endo-Levi}, pour l'essentiel.

Soient $\tilde{L},\tilde{M},M^\Endo$ comme ci-dessus. Comme dans l'endoscopie de groupes réductifs connexes, on a
$$ \mathfrak{a}_L \hookrightarrow \mathfrak{a}_M \rightiso \mathfrak{a}_{M^\Endo} = \mathfrak{a}_{L^\Endo}. $$
Une donnée endoscopique pour $\tilde{L}$ est elliptique si et seulement si $\mathfrak{a}_{L^\Endo} \rightiso \mathfrak{a}_L$ via ces applications. Parallèlement, on a des inclusions
$$ Z_{\dualmeta{L}}^\Gred \hookrightarrow Z_{\dualmeta{M}}^\Gred \hookrightarrow Z_{\widehat{M^\Endo}}^{\Gamma_F} = Z_{\widehat{L^\Endo}}^{\Gamma_F} . $$

\subsection{Correspondance des classes géométriques semi-simples}
Fixons $\tilde{L}$ un groupe de type métaplectique. Soient $s \in \mathcal{E}(\tilde{L})$ et $L^\Endo$ le groupe endoscopique associé. Notre but est de définir une application
$$ \mu: \Cssgeo(L^\Endo(F)) \to \Cssgeo(L(F)). $$

D'après \ref{prop:endo-Levi}, il existe un sous-groupe de Lévi $\tilde{M}$ tel que $s$ induit une donnée endoscopique elliptique de $\tilde{M}$ et $M^\Endo = L^\Endo$. On sait définir une application $\Cssgeo(M^\Endo(F)) \to \Cssgeo(M(F))$; en effet, selon \eqref{eqn:nabla-desc-ell}, l'endoscopie est tautologique en les composantes $\GL$ et c'est la situation considérée dans \cite{Li09} en la composante $\Sp$, pour laquelle une application $\mu$ des classes géométriques semi-simples est déjà définie. Rappelons-la brièvement.

Supposons momentanément que $\tilde{M} := \Mp(W^\flat)$, le groupe endoscopique elliptique de $\tilde{M}$ associé à la paire $(m',m'')$ est $M^\Endo = \SO(2m'+1) \times \SO(2m''+1)$. Soit $\gamma=(\gamma',\gamma'') \in M^\Endo(F)$ semi-simple ayant valeurs propres
$$ \underbrace{a'_1, \ldots, a'_{n'}, 1, (a'_{n'})^{-1}, \ldots, (a'_{1})^{-1}}_{\text{provenant de } \gamma'}, \underbrace{a''_1, \ldots, a''_{n''}, 1, (a''_{n''})^{-1}, \ldots, (a''_{1})^{-1}}_{\text{provenant de } \gamma''}. $$
On dit que $\delta \in M(F)$ correspond à $\gamma$ s'il est semi-simple avec valeurs propres
$$ a'_1, \ldots, a'_{n'}, (a'_{n'})^{-1}, \ldots, (a'_{1})^{-1}, -a''_1, \ldots, -a''_{n''}, -(a''_{n''})^{-1}, \ldots, -(a''_{1})^{-1}. $$
On en déduit une application $\Cssgeo(M^\Endo(F)) \to \Cssgeo(M(F))$. Le cas où $\tilde{M}$ est de type métaplectique s'ensuit.

Composons $\Cssgeo(M^\Endo(F)) \to \Cssgeo(M(F))$ avec l'application canonique $\Cssgeo(M(F)) \to \Cssgeo(L(F))$, on obtient $\mu$. On récapitule la situation par le diagramme suivant.
$$\xymatrix{
  & & \tilde{L} \\
  M^\Endo \ar@{--}[rr]^{\text{endo.ell.}} & & \tilde{M} \ar@{^{(}->}[u]_{\text{Lévi}}
}$$

On dit que $\gamma \in L^\Endo(F)_\text{ss}$ et $\delta \in L(F)_\text{ss}$ se correspondent si leurs classes géométriques se correspondent via $\mu$.

\begin{proposition}
  Supposons que $\delta \in L(F)_\text{ss}$ et $\gamma \in L^\Endo(F)_\text{ss}$ se correspondent.
  Si $\delta$ est régulier, alors $\delta$ et $\gamma$ sont tous fortement réguliers et on a
  $$ L^\Endo_\gamma \simeq L_\delta . $$
\end{proposition}
\begin{proof}
  Dans $L$, régulier implique fortement régulier et on se ramène au cas où $L^\Endo$ est un groupe endoscopique elliptique pour $\tilde{L}$. On se ramène ensuite au cas $L$ métaplectique qui est traité dans \cite{Li09}.
\end{proof}

\begin{remark}
  Jusqu'à maintenant, nos définitions ont peu à faire avec le corps $F$ et le revêtement n'y intervient pas. Donc on peut aussi définir les données endoscopiques, la notion d'ellipticité et la correspondance ci-dessus dans le cas où $F$ est un corps global. Nous ne l'utiliserons pas dans cet article.
\end{remark}

Supposons maintenant $\tilde{M}$ de type métaplectique, $s \in \mathcal{E}_\text{ell}(\tilde{M})$ et $M^\Endo$ est le groupe endoscopique elliptique associé. Nous allons définir le facteur de transfert. Écrivons
\begin{align*}
  \tilde{M} & = \prod_{i \in I} \GL(n_i) \times \Mp(W^\flat), \\
  M^\Endo & = \prod_{i \in I} \GL(n_i) \times \SO(2m'+1) \times \SO(2m''+1).
\end{align*}
Soient $\gamma \in M^\Endo(F)_\text{M-\text{reg}}$, $\delta \in M(F)_\text{reg}$ qui se correspondent. On isole les composantes dans $\GL(n_i)$ en les écrivant comme $\gamma = ((\gamma_i)_{i \in I}, \gamma', \gamma'')$ et $\delta = ((\delta_i)_{i \in I}, \delta^\flat)$. Alors $\gamma^\flat := (\gamma',\gamma'')$ et $\delta^\flat$ sont réguliers et ils se correspondent pour l'endoscopie elliptique associée à la paire $(m',m'')$. On applique la théorie de \cite{Li09}.

\begin{definition}\label{def:facteur-transfert}
  Soient $\delta$, $\gamma$ comme ci-dessus. Soit $\tilde{\delta} = ((\delta_i)_{i \in I}, \tilde{\delta}^\flat) \in \rev^{-1}(\delta)$, on définit le facteur de transfert par
  $$ \Delta(\gamma, \tilde{\delta}) = \Delta_{M^\Endo, \tilde{M}}(\gamma, \tilde{\delta}) := \Delta(\gamma^\flat, \tilde{\delta}^\flat). $$

  Si $\delta \in M(F)$ et $\gamma \in M^\Endo(F)$ ne se correspondent pas, on pose $\Delta(\gamma,\tilde{\delta})=0$ pour tout $\tilde{\delta}^\flat) \in \rev^{-1}(\delta)$.
\end{definition}

Ce facteur vérifie toutes les propriétés énumérées dans \cite{Li09} \S 1. En particulier, il ne dépend que de la classe de conjugaison géométrique de $\gamma$ et la classe de conjugaison de $\tilde{\delta}$; on a aussi $\Delta(\gamma, \noyau\tilde{\delta}) = \noyau\Delta(\gamma,\tilde{\delta})$ pour tout $\noyau \in \bmu_8$.

On dit que $\gamma \in M^\Endo(F)_\text{ss}$ est $L$-régulier s'il correspond à un élément $\delta \in M(F)$ qui est régulier dans $L(F)$. On note la sous-variété ouverte des éléments $L$-réguliers par $M^\Endo_{L-\text{reg}}$, c'est inclus dans $M^\Endo_{M-\text{reg}}$.

\begin{remark}\label{rem:facteur-transfert}
  On peut étendre \ref{def:facteur-transfert} au cas des données endoscopiques non elliptiques d'un groupe métaplectique. Soit $\gamma \in M^\Endo_{L-\text{reg}}(F)$. Notons $\Xi^M[\gamma]$ l'ensemble des classes de conjugaison dans $M(F)$ qui correspondent à $\gamma$, et $\Xi^L[\gamma]$ la variante pour $L$ au lieu de $M$. C'est bien connu (eg. \cite{Wa09} 5.4 (4)) que l'application naturelle $\Xi^M[\gamma] \to \Xi^L[\gamma]$ est bijective. Pour tout $\tilde{\delta} \in \tilde{G}$ tel que $\delta$ correspondant à $\gamma$, on peut choisir un conjugué $\delta_M \in M(F)$; alors $\tilde{\delta}$ est conjugué à un élément $\tilde{\delta}_M \in \tilde{M}$. Posons
  $$ \Delta(\gamma,\tilde{\delta}) := \Delta(\gamma,\tilde{\delta}_M) $$                                                                                                                                                                                                                                                                                                                                                                                                                                                                                                                                         
  en utilisant le cas elliptique \ref{def:facteur-transfert}. Montrons qu'il est bien défini. Soient $\tilde{\delta}_M^1$, $\tilde{\delta}_M^2$ deux choix comme ci-dessus. Il existe alors $x,y \in L(F)$ tels que $x \tilde{\delta} x^{-1} = \tilde{\delta}_M^1$, $y\tilde{\delta}y^{-1} = \tilde{\delta}_M^2$. On sait aussi qu'il existe $m \in M(F)$ tel que $m \delta_M^1 m^{-1} = \delta_M^2$. Donc l'action adjointe par $m^{-1} yx^{-1}$ préserve $\delta_M^1$. Rappelons que deux éléments dans un groupe de type métaplectique commutent si et seulement si leurs images par le revêtement commutent. Il en résulte que $m^{-1}yx^{-1}$ centralise $\tilde{\delta}_M^1$, d'où $m \tilde{\delta}_M^1 m^{-1} = \tilde{\delta}_M^2$ et
  $$ \Delta(\gamma,\tilde{\delta}_M^1) = \Delta(\gamma,\tilde{\delta}_M^2). $$
\end{remark}

\subsection{L'ensemble $\mathcal{E}_{M^\Endo}(\tilde{G})$}
Prenons désormais $\tilde{G} = \Mp(W)$ et $\tilde{M}$ un sous-groupe de Lévi de la forme $\tilde{M} = \prod_{i \in I} \GL(n_i) \times \Mp(W^\flat)$. Supposons choisis $P_0 \in \mathcal{P}(M)$, des paires de Borel  $(T,B^M)$ et $(T,B)$ définies sur $F$ pour $M$ et $G$, respectivement, telles que $B^M \subset B \subset P_0$ (cf. \S\ref{sec:notations}). En particulier, $M$ est un Lévi standard de $G$ pour ces choix.


Fixons toujours $s_0 \in \mathcal{E}_\text{ell}(\tilde{M})$ et le groupe endoscopique elliptique associé $M^\Endo$. On prend un représentant dans $\dualmeta{M}$ de la classe $s_0$ et on le note abusivement par le même symbole $s_0$. Écrivons $s_0 = ((s_{0,i})_{i \in I}, s_0^\flat)$ selon la décomposition $\dualmeta{M} = \prod_{i \in I} \GL(n_i,\C) \times \Sp(2m,\C)$.

\begin{definition}
  Posons
  $$ \mathcal{E}_{M^\Endo}(\tilde{G}) := \{s \in s_0 Z_{\dualmeta{M}}^\Gred / Z_{\dualmeta{G}}^\Gred  : (\text{la classe de } s) \in \mathcal{E}_\text{ell}(\tilde{G}) \} . $$
\end{definition}

C'est sous-entendu que cet ensemble dépend de $s_0$, non seulement du groupe $M^\Endo$. 

\begin{lemma}\label{prop:finitude-s}
  On a $|\mathcal{E}_{M^\Endo}(\tilde{G})| = 2^{|I|}$.
\end{lemma}
\begin{proof}
  La donnée endoscopique étant elliptique, on a $s_{0,i} \in \C^\times$ pour tout $i \in I$. On peut prendre un représentant de $s_0$ tel que $s_{0,i}=1$ pour tout $i \in I$. Alors
  $$ \mathcal{E}_{M^\Endo}(\tilde{G}) = \{ ((s_i)_{i \in I}, s_0^\flat) \in \dualmeta{M} : s_i = \pm 1 \} $$
  d'après \ref{prop:endo-ell}, d'où l'assertion.
\end{proof}

Soit $s \in \mathcal{E}_{M^\Endo}(\tilde{G})$, il fournit un groupe endoscopique elliptique pour $\tilde{G}$, noté $G[s]$ dans ce contexte. Écrivons $s = ((s_i)_i, s_0^\flat)$ avec $s_i = \pm 1$ comme dans la preuve de \ref{prop:finitude-s} et posons
\begin{align*}
  I' & := \{ i \in I : s_i = +1 \}, \\
  I'' & := \{ i \in I : s_i = -1 \}, \\
  n' & := m' + \sum_{i \in I'} n_i, \\
  n'' & := m'' + \sum_{i \in I''} n_i .
\end{align*}
Alors $n'+n''=n$ et $s$ est la donnée endoscopique elliptique de $\tilde{G}$ associée à la paire $(n',n'')$. Donc on a $G[s] = \SO(2n'+1) \times \SO(2n''+1)$. D'autre part, on peut plonger $M^\Endo$ dans $G[s]$ de la façon suivante: $\prod_{i \in I'} \GL(n_i) \times \SO(2m'+1)$ (resp. $\prod_{i \in I''} \GL(n_i) \times \SO(2m''+1)$) se plonge dans $\SO(2n'+1)$ (resp. $\SO(2n''+1)$) comme un sous-groupe de Lévi. Ce plongement est unique à conjugaison près par $G[s](F)$. On récapitule la situation par le diagramme suivant.

$$\xymatrix{
  G[s] \ar@{--}[rr]^{\text{endo.ell.}} & & \tilde{G} \\
  M^\Endo \ar@{^{(}->}[u]^{\text{Lévi}} \ar@{--}[rr]^{\text{endo.ell.}} & & \tilde{M} \ar@{^{(}->}[u]_{\text{Lévi}}
}$$

Dans cette situation, le Lévi $M$ de $G$ correspond à la paire $(G[s],M^\Endo)$ au sens de \cite{Li09} \S 5.4. On peut regarder $M^\Endo$ de deux manières: un groupe endoscopique elliptique  du Lévi $\tilde{M}$ de $\tilde{G}$, ou un Lévi du groupe endoscopique elliptique $G[s]$ de $\tilde{G}$. Au contraire du cas des groupes réductifs, il y une différence comme suit.

\begin{definition}\label{def:torsion}
  Soit $s \in \mathcal{E}_{M^\Endo}(\tilde{G})$. Posons $z[s] := ((z_i)_{i \in I}, 1) \in Z_{M^\Endo}(F)$ où $z_i = +1$ (resp. $-1$) si $i \in I'$ (resp. si $i \in I''$). C'est contenu dans tout sous-groupe hyperspécial de $M^\Endo(F)$. Soit $\gamma \in M^\Endo(F)$ , posons $\gamma[s] := z[s] \gamma$. Signalons aussi que l'on peut translater une classe de conjugaison dans $M^\Endo(F)$ par l'élément central $z[s]$.

  Cette définition se généralise immédiatement au cas $\tilde{G}$ de type métaplectique et $s \in \mathcal{E}_{M^\Endo}(\tilde{G})$: les facteurs $\GL$ supplémentaires de $\tilde{G}$ n'y interviennent pas.

  Plus généralement, soit $s \in s_0 Z_{\dualmeta{M}}^\Gred/Z_{\dualmeta{G}}^\Gred$. Alors il existe un sous-groupe de Lévi $\tilde{L}$ de $\tilde{G}$ contenant $\tilde{M}$, tel que $s \in \mathcal{E}_{M^\Endo}(\tilde{L})$. On définit $z[s] \in Z_{M^\Endo}(F)$ et l'application $\gamma \mapsto \gamma[s]= z[s] \gamma$ sur $M^\Endo(F)$ par référence à $\tilde{L}$.
\end{definition}

\begin{proposition}\label{prop:torsion-correspondance}
  Notons $\mu_1 : \Cssgeo(M^\Endo(F)) \to \Cssgeo(G(F))$ le composé de $\Css(M^\Endo(F)) \to \Css(G[s](F))$ (induit par l'inclusion) avec $\mu: \Css(G[s](F)) \to \Css(G(F))$ (induit par l'endoscopie déterminée par $s$). Alors
  $$ \mu_1(\mathcal{O}) = \mu(z[s] \cdot \mathcal{O})$$
  pour tout $\mathcal{O} \in \Cssgeo(M^\Endo(F))$.
\end{proposition}
\begin{proof}
  C'est clair d'après la définition de $\mu$.
\end{proof}

Notons qu'il y a une inclusion canonique $W^{G[s]}(M^\Endo) \hookrightarrow W^G(M)$. En effet, notons $\mathfrak{S}^\pm(I)$ le produit semi-direct $\mathfrak{S}(I) \ltimes (\Z/2\Z)^I$, où $\mathfrak{S}(I)$ est le groupe symétrique opérant sur $I$. Alors $W^G(M)$ s'identifie à $\mathfrak{S}^\pm(I)$, tandis que $W^{G[s]}(M^\Endo)$ s'identifie à $\mathfrak{S}^\pm(I') \times \mathfrak{S}^\pm(I'')$.

Le résultat suivant est alors immédiat.

\begin{proposition}
  L'isomorphisme $\mathfrak{a}_{M^\Endo} \rightiso \mathfrak{a}_M$ est équivariante par rapport à $W^{G[s]}(M^\Endo) \hookrightarrow W^G(M)$.
\end{proposition}

\begin{proposition}
  Désignons par $\Delta_{M^\Endo, \tilde{M}}$ et $\Delta_{G[s],\tilde{G}}$ les facteurs de transfert associés aux données $s_0$ et $s$, respectivement. Soient $\gamma \in M^\Endo(F)_{G-\mathrm{reg}}$ et $\delta \in M(F)_{G-\mathrm{reg}}$ qui se correspondent. Pour tout $\tilde{\delta} \in \rev^{-1}(\delta)$, on a
  $$ \Delta_{M^\Endo, \tilde{M}}(\gamma, \tilde{\delta}) = \Delta_{G[s],\tilde{G}}(\gamma[s], \tilde{\delta}). $$
\end{proposition}
\begin{proof}
  Adoptons la notation dans \ref{def:facteur-transfert}. D'après la descente parabolique de $\Delta_{G[s],\tilde{G}}$ \cite{Li09} 5.18, appliquée en $(\gamma[s],\tilde{\delta}$), on a
  $$ \Delta_{G[s],\tilde{G}}(\gamma[s],\tilde{\delta}) = \Delta_{\SO(2m'+1)\times\SO(2m''+1),\Mp(W^\flat)}(\gamma^\flat, \tilde{\delta}^\flat). $$

  Or c'est exactement la définition de $\Delta_{M^\Endo,\tilde{M}}(\gamma, \tilde{\delta})$.
\end{proof}

\begin{remark}
  Cette propriété et \ref{prop:torsion-correspondance} permettent d'étendre le transfert (vrai pour tout corps local $F$ de caractéristique nulle) et le lemme fondamental non pondéré aux données endoscopiques non elliptiques de $\tilde{G}$; le facteur de transfert étant celui défini dans \ref{rem:facteur-transfert}. En effet, on le réduit de façon usuelle au transfert (resp. au lemme fondamental) elliptique suivi par une descente parabolique des intégrales orbitales. Les détails sont laissés au lecteur.
\end{remark}

\section{Intégrales orbitales pondérées endoscopiques et les fonctions stabilisées}\label{sec:int-orb}
Dans cette section, $F$ est une extension finie de $\Q_p$. Posons toujours $\tilde{G} := \Mp(W)$. Supposons que
\begin{itemize}
  \item $\psi|_{\mathfrak{o}_F}=1$, $\psi|_{\mathfrak{p}_F^{-1}}$ non trivial;
  \item $p$ est suffisamment grand par rapport à $G$.
\end{itemize}
Fixons aussi un Lévi $\tilde{M}$ de $\tilde{G}$ associé à la donnée de sous-espaces $((\ell_i, \ell^i)_{i \in I}, W^\flat)$, posons  $n_i := \dim \ell_i$ pour tout $i \in I$, alors $M$ est de la forme
$$ \tilde{M} = \prod_{i \in I} \GL(n_i) \times \Mp(W^\flat). $$
Posons $2m =  \dim_F W^\flat$.

\subsection{Intégrales orbitales pondérées non ramifiées anti-spécifiques}
Fixons un réseau autodual $\mathfrak{L} \subset W$ par rapport à $\angles{|}$ et posons $K := \mathrm{Stab}_{G(F)}(\mathfrak{L}) \subset G(F)$, c'est un sous-groupe hyperspécial de $G(F)$. Supposons que $\mathfrak{L}$ est en bonne position relativement à $((\ell_i, \ell^i)_{i \in I},W^\flat)$, c'est-à-dire
$$ \mathfrak{L} = \bigoplus_{i \in I} \left((\ell_i \cap \mathfrak{L}) \oplus (\ell^i \cap \mathfrak{L})\right) \oplus (W^\flat \cap \mathfrak{L}). $$

Cela entraîne que $K^M := K \cap M(F)$ est aussi hyperspécial dans $M(F)$. Écrivons $K^M = \prod_{i \in I} K_i^M \times K^\flat$. Le modèle latticiel de la représentation de Weil induit un scindage $K^\flat \hookrightarrow \Mp(W^\flat)$, d'où le scindage $K^M \hookrightarrow \tilde{M}$. Idem, $K \hookrightarrow \Mp(W)$. Bien entendu, il faut une compatibilité.

\begin{proposition}
  Le diagramme suivant est commutatif.
  $$\xymatrix{
    K \ar[r] & \tilde{G} \\
    K^M \ar[u] \ar[r] & \tilde{M} \ar[u]
  }$$
\end{proposition}
\begin{proof}
  On se ramène aussitôt au cas $M = \GL(n)$. D'après \cite{Li09} 2.13, on a la compatibilité cherchée sur un $F$-tore maximal déployé dans $M$. D'autre part la compatibilité sur des sous-groupes unipotents est automatique. On en déduit la compatibilité sur la grosse cellule de $M$ par la décomposition de Bruhat, et le cas général en résulte par densité. 
\end{proof}

Prenons la mesure de Haar sur $G(F)$ de sorte que $\mes(K)=1$; cette mesure ne dépend pas du choix de $K$. Le scindage du revêtement $\rev$ au-dessus de $K$ permet de définir la fonction $f_K \in C_{c,\asp}^\infty(\tilde{G})$ telle que
$$ f_K(\tilde{x}) = \begin{cases} \noyau^{-1}, & \text{si } \tilde{x} \in \noyau K, \noyau \in \bmu_8 ; \\ 0, & \text{sinon}. \end{cases} $$

Fixons $s_0 \in \mathcal{E}_\text{ell}(\tilde{M})$, auquel est associé le groupe endoscopique
$$ M^\Endo = \prod_{i \in I} \GL(n_i) \times \SO(2m'+1) \times \SO(2m''+1). $$
tel que $m'+m'' = m$.

On sait définir la fonction poids d'Arthur $v_M: M(F) \backslash G(F)/K \to \R_{\geq 0}$ (voir \cite{Ar88LB}). Elle dépend du choix de $K$ et de la mesure de Haar sur $\mathfrak{a}_M$ induite du choix d'une forme quadratique définie positive sur $\mathfrak{a}_M$, invariante par $W^G(M)$, ce que l'on fixe dorénavant.

Soit $\delta \in G_\text{reg}(F)$. On normalise la mesure de Haar pour le tore $T(F) := G_\delta(F)$ comme suit. D'après \cite{BT84} 4.4, il existe un $\mathfrak{o}_F$-schéma en groupes canonique $\mathfrak{T}$, qui est lisse et de fibre générique $T$. Alors $\mathfrak{T}^0(\mathfrak{o}_F)$ s'identifie à un sous-groupe ouvert de $T(F)$. Prenons la mesure de Haar sur $T(F)$ telle que $\mes(\mathfrak{T}^0(\mathfrak{o}_F))=1$. La même convention s'applique aux commutants des éléments fortement réguliers dans n'importe quel $F$-groupe réductif connexe. Puisque la construction ne dépend que du tore $T$, c'est compatible avec la correspondance des mesures utilisée pour l'endoscopie \cite{Wa09} 2.8, ainsi que sa variante métaplectique \cite{Li09} \S 5.5.

\begin{definition}
  Soit $\tilde{\delta} \in \tilde{G}_\text{reg}$. Fixons la mesure de Haar sur $G_\delta(F)$ comme ci-dessus et posons
  \begin{gather*}
    r^{\tilde{G}}_{\tilde{M},K}(\tilde{\delta}) := |D^G(\delta)|^{\frac{1}{2}} \int_{G_\delta(F) \backslash G(F)} f_K(x^{-1}\tilde{\delta}x) v_M(x) \dd x ,
  \end{gather*}
  où $D^G(\delta) := \det(1-\Ad(\delta)|\mathfrak{g}/\mathfrak{g}_\delta)$.
\end{definition}

\begin{definition}
  Soit $\gamma \in M^\Endo_{G-\text{reg}}(F)$. L'intégrale pondérée endoscopique non ramifiée en $\gamma$ est définie comme
  \begin{gather*}
    r^{\tilde{G}}_{M^\Endo,K}(\gamma) := \sum_{\delta \in M(F)/\text{conj}} \Delta(\gamma, \tilde{\delta}) r^{\tilde{G}}_{\tilde{M},K}(\tilde{\delta}),
  \end{gather*}
  où $\tilde{\delta} \in \rev^{-1}(\delta)$ est quelconque; le produit $\Delta(\gamma, \tilde{\delta}) r^{\tilde{G}}_{\tilde{M},K}(\tilde{\delta})$ ne dépend que de $\gamma$ et $\delta$. C'est une somme finie, en fait la somme porte sur les classes de conjugaison dans $\mu(\mathcal{O}^\text{st}(\gamma))$.
\end{definition}

Rappelons brièvement les fonctions stabilisées définies par Arthur. Pour l'instant, soient $L$ un $F$-groupe réductif connexe non ramifié et $R$ un sous-groupe de Lévi de $L$. Supposons aussi fixée une forme quadratique définie positive sur $\mathfrak{a}_R$, invariante par $W^L(R)$. Imposons les mêmes choix de mesures de Haar sur $L(F)$ et sur les $F$-tores que précédemment.

On dit que $\gamma \in R(F)_\text{ss}$ est $L$-régulier s'il est régulier comme un élément dans $L(F)_\text{ss}$. La sous-variété ouverte des éléments $L$-réguliers est notée $R_{L-\text{reg}}$. On définit (voir \cite{Ar02} 5.1) la fonction stabilisée
$$ s^L_R: R_{L-\text{reg}}(F) \to \C . $$

Si $\gamma_1, \gamma_2 \in R_{L-\text{reg}}(F)$ sont stablement conjugués, on a $s^L_R(\gamma_1)=s^L_R(\gamma_2)$. Ces fonctions ne dépendent que de la mesure sur $\mathfrak{a}_R$ induite par le choix de la forme quadratique.

\subsection{Énoncé du lemme fondamental pondéré}\label{sec:enonce-LF}
Conservons les mêmes notations. On a l'inclusion $Z_{\dualmeta{M}}^\Gred \hookrightarrow Z_{\widehat{M^\Endo}}$. Soit $s \in \mathcal{E}_{M^\Endo}(\tilde{G})$. On a aussi $Z_{\dualmeta{G}}^\Gred \hookrightarrow Z_{\widehat{G[s]}}$. Ces inclusions ont des conoyaux finis, donc c'est loisible de poser
\begin{gather}\label{eqn:coef-meta}
  i_{M^\Endo}(\tilde{G},G[s]) := \dfrac{[Z_{\widehat{M^\Endo}} : Z_{\dualmeta{M}}^\Gred]}{[Z_{\widehat{G[s]}} : Z_{\dualmeta{G}}^\Gred]}.
\end{gather}

Prenons les formes quadratiques positives définies invariantes sur $\mathfrak{a}_{M^\Endo}$ et sur $\mathfrak{a}_M$ qui se correspondent par l'identification $\mathfrak{a}_{M^\Endo} \simeq \mathfrak{a}_M$, qui est équivariante pour $W^{G[s]}(M^\Endo) \hookrightarrow W^{G}(M)$. Soit $\gamma \in M^\Endo_{G-\text{reg}}(F)$, alors $\gamma[s]$ est aussi $G[s]$-régulier. Ainsi, $\gamma \mapsto s^{G[s]}_{M^\Endo}(\gamma[s])$ est bien défini et il ne dépend que de la classe de conjugaison stable de $\gamma$. Le principal résultat de cet article s'énonce comme suit.

\begin{theorem}[Cf. \cite{Ar02} 5.1]\label{prop:LF-pondere}
  Supposons vérifié le lemme fondamental pondéré non standard sur les algèbres de Lie \ref{prop:LF-nonstandard}. Pour tout $\gamma \in M^\Endo_{G-\text{reg}}(F)$, on a
  $$ r^{\tilde{G}}_{M^\Endo,K}(\gamma) = \sum_{s \in \mathcal{E}_{M^\Endo}(\tilde{G})} i_{M^\Endo}(\tilde{G},G[s]) \cdot s^{G[s]}_{M^\Endo}(\gamma[s]). $$
\end{theorem}

Le lemme fondamental pondéré non standard sera rappelé dans le \S\ref{sec:non-standard}. La démonstration du théorème occupera le reste de cet article.

\begin{remark}
  Lorsque $M=G$, on a $\mathcal{E}_{G^\Endo}(\tilde{G})=\{s_0\}$, $\gamma[s_0]=\gamma$ et $G[s_0]=G^\Endo$. On a aussi $i_{G^\Endo}(\tilde{G},G^\Endo)=1$. Dans ce cas-là $r^{\tilde{G}}_{G^\Endo,K}(\gamma)$ est l'intégrale orbitale endoscopique de la fonction $f_K$, et $s^{G^\Endo}_{G^\Endo}(\gamma)$ est l'intégrale orbitale stable de la fonction caractéristique d'un hyperspécial quelconque de $G^\Endo(F)$. Donc \ref{prop:LF-pondere} se réduit au lemme fondamental pour l'unité \cite{Li09} 5.23.
\end{remark}

Les définitions des intégrales orbitales pondérées endoscopiques et les fonctions stabilisées correspondantes se généralisent sans peine au cas $\tilde{G}$ de type métaplectique: il suffit de combiner la théorie pour les groupes $\GL(\cdot)$ avec le résultat pour $\Mp(W)$.

\begin{corollary}
  L'assertion \ref{prop:LF-pondere} demeure valable dans le cas plus général où $\tilde{G}$ est de type métaplectique.
\end{corollary}

\section{Endoscopie: standard et non standard}\label{sec:endoscopie}
\subsection{Endoscopie standard}
\paragraph{Données endoscopiques}
Soit $R$ un $F$-groupe réductif muni d'une donnée de L-groupe $(\hat{T}, \hat{B}, \ldots)$. On appelle donnée endoscopique pour $R$ un quadruplet $(R^\Endo,\mathcal{R}^\Endo,s,\hat{\xi})$ tel que
\begin{itemize}
  \item $R^\Endo$ est un $F$-groupe réductif connexe quasi-déployé, muni d'une donnée de L-groupe;
  \item $\mathcal{R}^\Endo$ s'inscrit dans une extension topologique scindée
    $$ 1 \to \widehat{R^\Endo} \to \mathcal{R}^\Endo \to W_F \to 1 $$
    telle que l'action de $W_F$ sur $\widehat{R^\Endo}$ coïncide avec $\rho^\Endo$;
  \item $s \in \hat{R}$ est semi-simple;
  \item $\hat{\xi}: \mathcal{R}^\Endo \to \Lgrp{R}$ est un L-homomorphisme, i.e. il commute aux projections sur $W_F$, tel que:
    \begin{itemize}
      \item il existe un $1$-cocycle $a: W_F \to Z_{\hat{R}}$, dont la classe cohomologique est triviale, tel que $\Ad(s) \circ \hat{\xi}(x) = a(w(x)) \hat{\xi}(x)$ pour tout $x \in \mathcal{R}^\Endo$, où $w(x)$ désigne sa projection dans $W_F$;
      \item $\hat{\xi}$ induit un isomorphisme $\widehat{R^\Endo} \rightiso Z_{\hat{R}}(s)^0$.
    \end{itemize}
\end{itemize}

On dit aussi que $R^\Endo$ est un groupe endoscopique pour $R$ associés à $s$. Un isomorphisme entre deux données endoscopiques $(R^\Endo,\mathcal{R}^\Endo,s,\hat{\xi})$ et $(R'^\Endo,\mathcal{R}'^\Endo,s',\hat{\xi}')$ de $R$ est un élément $\hat{r} \in \hat{R}$ tel que $\Ad(\hat{r})\hat{\xi}(\mathcal{R}^\Endo) = \hat{\xi}'(\mathcal{R}'^\Endo)$ et $\Ad(\hat{r})(s) = s' \mod Z_{\hat{R}}$. On dit qu'une donnée endoscopique $(R^\Endo,\mathcal{R}^\Endo,s,\hat{\xi})$ est elliptique si $\hat{\xi}(Z_{\widehat{R^\Endo}}^{\Gamma_F})^0 \subset Z_{\hat{R}}$. On dit que cette donnée est non ramifiée si
\begin{itemize}
  \item la caractéristique résiduelle $p$ de $F$ est suffisamment grande par rapport à $R$;
  \item $R$ est non ramifié et le torseur intérieur fixé est l'identité;
  \item $R^\Endo$ est non ramifié et $\hat{\xi}(\mathcal{R}^\Endo) \supset Z_{\hat{R}}(s)^0 \rtimes I_F$.
\end{itemize}

À isomorphisme près, on peut supposer que $s \in \hat{T}$ et $(\hat{\xi}^{-1}(\hat{T}), \hat{\xi}^{-1}(\hat{B}))$ est la paire de Borel $\Gamma_F$-invariante de $\widehat{R^\Endo}$. L'endoscopie fournit une inclusion $\Gamma_F$-équivariante $Z_{\hat{R}} \hookrightarrow Z_{\widehat{R^\Endo}}$. D'autre part, en dualisant $\hat{\xi}|_{\hat{\xi}^{-1}(\hat{T})}$, on obtient une inclusion $\mathfrak{a}_R \hookrightarrow \mathfrak{a}_{R^\Endo}$. La donnée endoscopique est elliptique si et seulement si $\mathfrak{a}_R \rightiso \mathfrak{a}_{R^\Endo}$.

\paragraph{Une construction d'Arthur}
Supposons que $R$ est un sous-groupe de Lévi d'un $F$-groupe réductif non ramifié $L$. Choisissons $P_0 \in \mathcal{L}^L(R)$, qui fait partie du choix des données de L-groupes compatibles pour $L$ et $R$, ce que l'on fixe, dont les torseurs intérieurs sont $\identity$.

La construction suivante est due à Arthur \cite{Ar99}. Soit $(R^\Endo, \mathcal{R}^\Endo, s_0, \hat{\xi})$ une donnée endoscopique elliptique non ramifiée pour $R$. Soit $s \in s_0 Z_{\hat{R}}^{\Gamma_F}/Z_{\hat{L}}^{\Gamma_F}$. Posons
\begin{align*}
  \widehat{L[s]} &:= Z_{\hat{L}}(s)^0, \\
  \mathcal{L}[s] &:= \widehat{L[s]} \hat{\xi}(\mathcal{R}^\Endo), \\
  \hat{\xi}[s] &:\mathcal{L}[s] \to \Lgrp{L}  \text{ est l'inclusion.}
\end{align*}

Cela fournit une donnée endoscopique $(L[s], \mathcal{L}[s], s, \hat{\xi}[s])$ pour $L$ qui est encore non ramifiée. Posons
$$ \mathcal{E}_{R^\Endo}(L) := \{ s \in s_0 Z_{\hat{R}}^{\Gamma_F}/Z_{\hat{L}}^{\Gamma_F} : (L[s],\mathcal{L}[s], s, \hat{\xi}[s]) \text{ est elliptique} \}. $$

C'est un ensemble fini d'après \cite{Ar98} \S 4. Soit $s \in \mathcal{E}_{R^\Endo}(L)$, alors on peut regarder $R^\Endo$ comme un sous-groupe de Lévi de $L[s]$. L'endoscopie elliptique fournit l'homomorphisme $W^{L[s]}(R^\Endo) \hookrightarrow W^L(R)$, pour lequel l'isomorphisme $\mathfrak{a}_R \rightiso \mathfrak{a}_{R^\Endo}$ est équivariant. On sait aussi définir le coefficient
$$ i_{R^\Endo}(L,L[s]) := \frac{[Z_{\widehat{R^\Endo}}^{\Gamma_F} : Z_{\hat{R}}^{\Gamma_F}]}{[Z_{\widehat{L[s]}}^{\Gamma_F} : Z_{\hat{L}}^{\Gamma_F}]}. $$

\paragraph{Le lemme fondamental pondéré sur les algèbres de Lie}
Plaçons-nous dans la construction d'Arthur. Fixons une donnée endoscopique elliptique non ramifiée $(R^\Endo, \mathcal{R}^\Endo, s_0, \hat{\xi})$ pour $R$. À cette donnée est associée une correspondance de classes de conjugaison géométriques semi-simples entre les algèbres de Lie $\mathfrak{r}(F)$ et $\mathfrak{r}^\Endo(F)$. On définit ainsi la sous-variété ouverte des éléments $L$-régulières $\mathfrak{r}^\Endo_{L-\text{reg}}$ de façon usuelle.

Fixons un sous-groupe hyperspécial $K$ de $L(F)$ en bonne position relativement à $R$, ce qui détermine un réseau hyperspécial $\mathfrak{k} \subset \mathfrak{l}(F)$. Fixons aussi une forme quadratique définie positive $W^L(R)$-invariante sur $\mathfrak{a}_R$; on obtient ainsi une forme quadratique définie positive sur $\mathfrak{a}_{R^\Endo}$ qui est $W^{L[s]}(R^\Endo)$-invariante pour tout $s \in \mathcal{E}_{R^\Endo}(L)$. Notons $\mathbbm{1}_{\mathfrak{k}}$ la fonction caractéristique de $\mathfrak{k}$. Ces choix permettent de définir les intégrales orbitales pondérées de $\mathbbm{1}_{\mathfrak{k}}$
$$ r^L_{R,K}(X) := |D^L(X)|^{\frac{1}{2}} \int_{L_X(F) \backslash L(F)} \mathbbm{1}_{\mathfrak{k}}(x^{-1} X x) v_R(x) \dd x, \qquad X \in \mathfrak{r}_\text{reg}(F), $$
où $D^L(X) := \det(\ad(X)|\mathfrak{l}/\mathfrak{l}_X)$. Les mesures de Haar sont choisies comme dans le \S\ref{sec:enonce-LF}.

On définit le facteur de transfert
$$ \Delta: \mathfrak{r}^\Endo_{R-\text{reg}}(F) \times \mathfrak{r}_\text{reg}(F) \to \C $$
qui est adapté à $K$, cf. \cite{Wa08} 4.7. On définit les intégrales orbitales pondérées endoscopiques
$$ r^L_{R^\Endo, K}(Y) = \sum_{X \in \mathfrak{r}(F)/\text{conj}} \Delta(Y,X) r^L_{R,K}(X), $$
ainsi que les fonctions stabilisées $s^{L[s]}_{R^\Endo}(Y)$, pour tout $Y \in \mathfrak{r}^\Endo_{L-\text{reg}}(F)$ et tout $s \in \mathcal{E}_{R^\Endo}(L)$ (voir \cite{Wa09} ou \S\ref{sec:int-orb}). Énonçons le lemme fondamental pondéré comme suit.

\begin{theorem}[Chaudouard, Laumon \cite{CL09-1,CL09-2}, Waldspurger \cite{Wa09-IF}]\label{prop:LF-pondere-alg}
  Pour tout $Y \in \mathfrak{r}^\Endo_{R-\text{reg}}(F)$, on a
  $$ r^L_{R^\Endo, K}(Y) = \sum_{s \in \mathcal{E}_{R^\Endo}(L)} i_{R^\Endo}(L,L[s]) s^{L[s]}_{R^\Endo}(Y). $$
\end{theorem}

\subsection{Exemples}\label{sec:ex-endo}
Nous allons considérer l'endoscopie elliptique des groupes linéaires généraux, symplectiques et unitaires. Nous en donnerons les conditions pour que la donnée endoscopique soit non ramifiée et la construction d'Arthur sera explicitement décrite. Le cas non elliptique sera traité dans \ref{rem:non-ell}.

\begin{example}
  Soient $m \in \Z_{\geq 0}$, $R := \GL(m)$. La seule donnée endoscopique elliptique pour $\GL(m)$ à isomorphisme près est la donnée tautologique $(R, \Lgrp{R}, 1, \identity)$.
\end{example}

\begin{example}\label{ex:symplectique}
  Soit $R := \Sp(2m)$. Alors $\hat{R} = \SO(2m+1,\C)$. Soit $(R^\Endo,\mathcal{R}^\Endo,s,\hat{\xi})$ une donnée endoscopique elliptique pour $R$. Les valeurs propres de $s$ sont forcément $\pm 1$ d'après l'ellipticité. Notons $2m'+1$ la multiplicité de $+1$ et $2m''$ celle de $-1$. On a
  $$\widehat{R^\Endo} = \SO(2m'+1,\C) \times \SO(2m'',\C).$$
  Donc
  $$ R^\Endo = \Sp(2m') \times \SO(V'',q'')$$
  où $(V'',q'')$ est un $F$-espace quadratique de dimension $2m''$ ayant noyau anisotrope de dimension $\leq 2$ car $R^\Endo$ est quasi-déployé. On exclut le cas $\dim V'' = 2$ et $(V'',q'')$ hyperbolique, pour lequel la donnée endoscopique n'est plus elliptique. De plus, $\SO(V'',q'')$ détermine $\mathcal{H}^\Endo$.

  Inversement, toutes données $m'$ et $(V'',q'')$ comme ci-dessus proviennent d'une donnée endoscopique elliptique, unique à isomorphisme près. La donnée endoscopique est non ramifiée si et seulement si $(V'',q'')$ admet un réseau autodual à homothétie près.

  Supposons $R^\Endo$ de la forme ci-dessus. Soient $X \in \mathfrak{r}_\text{reg}(F)$, $Y=(Y',Y'') \in \mathfrak{r}^\Endo_\text{reg}(F)$, alors $X$ et $Y$ se correspondent si et seulement s'ils se correspondent par valeurs propres, c'est-à-dire
  \begin{itemize}
    \item les valeurs propres de $Y'$ sont $\pm a'_1, \ldots, \pm a'_{m'}$;
    \item celles de $Y''$ sont $\pm a''_1, \ldots, \pm a'_{m''}, 0$;
    \item celles de $X$ sont $\pm a'_1, \ldots, \pm a'_{m'}, \pm a''_1, \ldots, \pm a''_{m''}$.
  \end{itemize}

  Prenons maintenant un ensemble fini $I$, $(n_i)_{i \in I} \in \Z_{\geq 1}^I$, et
  \begin{align*}
    L & = \Sp(2n), \\
    R & = \prod_{i \in I} \GL(n_i) \times \Sp(2m), \\
    R^\Endo & = \prod_{i \in I} \GL(n_i) \times \Sp(2m') \times \SO(V'',q''),
  \end{align*}
  tels que $m + \sum_{i \in I} n_i = n$, où $m$, $m'$ et $(V'',q'')$ vérifient les conditions précédentes, qui font de $R^\Endo$ un groupe endoscopique elliptique non ramifiée de $R$ (tautologique en les composantes $\GL(n_i)$). À isomorphisme près, l'élément $s_0$ correspondant dans la donnée endoscopique est de la forme
  $$ s_0 = ((1)_{i \in I}, s_0^\flat) \in \prod_{i \in I} \C^\times \times \SO(2m+1,\C) \subset \hat{R}.$$

  Les éléments $s \in \mathcal{E}_{R^\Endo}(L)$ sont en bijection avec les décompositions $I=I' \sqcup I''$: à une telle décomposition est associée $s=((s_i)_{i \in I},s_0^\flat)$ avec $s_i=1$  (resp. $s_i=-1$) si $i \in I'$ (resp. si $i \in I''$). Soit $s \in \mathcal{E}_{R^\Endo}(L)$, on introduit le groupe endoscopique elliptique $L[s]$ de $L$; écrivons-le comme
  $$ L[s] = \Sp(2n') \times \SO(U'',r'') $$
  d'après ce qui précède. On a des inclusions uniques à conjugaison près:
  \begin{align*}
    \prod_{i \in I'} \GL(n_i) \times \Sp(2m') & \hookrightarrow \Sp(2n'), \\
    \prod_{i \in I''} \GL(n_i) \times \SO(V'',q'') & \hookrightarrow \SO(U'',r'').
  \end{align*}
  Ces flèches font de $R^\Endo$ un sous-groupe de Lévi de $L[s]$. Ainsi, $\SO(U'',r'')$ est aussi déterminé: à homothétie près, $(U'',r'')$ et $(V'',q'')$ ont le même noyau anisotrope.

  Inversement, tout $(U'',r'')$ ayant le même noyau anisotrope que $(V'',q'')$, pris à homothétie près, provient d'un unique $s \in \mathcal{E}_{R^\Endo}(L)$.
\end{example}

\begin{example}\label{ex:unitaire}
  Soient $E/F$ une extension quadratique et $R := \U_{E/F}(m)$, le groupe unitaire quasi-déployé qui est isomorphe à $\GL(m)$ sur $E$. On a $\hat{R} = \GL(m,\C)$. Il est muni de l'action du groupe $\Gamma_{E/F} := \Gamma_F/\Gamma_E$: l'élément non trivial dans $\Gamma_{E/F}$ opère, modulo automorphismes intérieurs, en renversant le diagramme de Dynkin $\mathbf{A}_{m-1}$. On introduit ainsi le L-groupe $\Lgrp{R}$.

  Décrivons les données endoscopiques elliptiques $(R^\Endo,\mathcal{R}^\Endo,s,\hat{\xi})$ pour $R$. Les valeurs propres de $s$ sont $\pm 1$; notons $m'$ la multiplicité de $1$ et $m''$ celle de $-1$. Alors
  $$ \widehat{R^\Endo} = \GL(m',\C) \times \GL(m'',\C). $$

  Étant donné $(m',m'')$, il n'y a qu'une seule possibilité de $R^\Endo$ et de $\mathcal{R}^\Endo$:
  $$ R^\Endo = \U_{E/F}(m') \times \U_{E/F}(m''). $$

  Inversement, toute donnée endoscopique elliptique de $R$ s'obtient d'une paire $(m',m'')$ telle que $m'+m''=m$, unique à la symétrie $(m',m'') \mapsto (m'',m')$ près. La donnée endoscopique est non ramifiée si et seulement si $E/F$ l'est.

  La correspondance des classes de conjugaison est similaire au cas du groupe symplectique, c'est-à-dire la correspondance par valeurs propres des applications $E$-linéaires. Nous ne la répétons pas.

  Supposons maintenant que $E/F$ est non ramifiée. Prenons $I,(n_i)_{i \in I}$ comme dans \ref{ex:symplectique} et posons
  \begin{align*}
    L & = \U_{E/F}(n), \\
    R & = \prod_{i \in I} \GL_E(n_i) \times \U_{E/F}(m), \\
    R^\Endo & = \prod_{i \in I} \GL_E(n_i) \times \U_{E/F}(m') \times \U_{E/F}(m''),
  \end{align*}
  qui font de $R^\Endo$ un groupe endoscopique elliptique non ramifiée de $R$. On peut supposer que l'élément $s_0$ de cette donnée endoscopique s'écrit
  $$ s_0 = ((1)_{i \in I}, s_0^\flat) \in \prod_{i \in I} \C^\times \times \GL(m,\C).$$

  Comme dans \ref{ex:symplectique}, les éléments $s \in \mathcal{E}_{R^\Endo}(L)$ sont en bijection avec les décompositions $I=I' \sqcup I''$; le groupe endoscopique elliptique associé est de la forme
  $$ L[s] = \U_{E/F}(n') \times \U_{E/F}(n''). $$
  On a des inclusions bien déterminées à conjugaison près:
  \begin{align*}
    \prod_{i \in I'} \GL_E(n_i) \times \U_{E/F}(m') & \hookrightarrow \U_{E/F}(n'), \\
    \prod_{i \in I''} \GL_E(n_i) \times \U_{E/F}(m'') & \hookrightarrow \U_{E/F}(n'').
  \end{align*}
  Cela détermine aussi $(n',n'')$. Remarquons que des différents choix de $s$ peuvent induire la même donnée endoscopique pour $L$.
\end{example}

\begin{remark}\label{rem:non-ell}
  En général, une donnée endoscopique $(L^\Endo,\mathcal{L}^\Endo,s,\hat{\xi})$ s'interprète comme une donnée endoscopique elliptique d'un sous-groupe de Lévi $R$ de $L$. Le sous-groupe $R$ est unique à conjugaison près par $L(F)$, tandis que la donnée endoscopique elliptique pour $R$ est unique à l'action de $W^L(R)$ près.
\end{remark}

\subsection{Endoscopie non standard}\label{sec:non-standard}
\paragraph{Définitions}
Rappelons la définition dans \cite{Wa09} 3.7. Soient $G_1$, $G_2$ deux $F$-groupes réductifs connexes et simplement connexes. Supposons qu'ils sont non ramifiés et $p$ est suffisamment grand par rapport à $G_1$ et $G_2$. Pour $i=1,2$, fixons une paire de Borel $(T_i,B_i)$ définie sur $F$ pour $G_i$; notons $\Sigma_i$, $\check{\Sigma}_i$ les ensembles de racines et coracines, respectivement. Une donnée endoscopique non standard non ramifiée est un triplet $(G_1,G_2,j_*)$ où
\begin{itemize}
  \item $G_1, G_2$ sont comme ci-dessus, munis de paires de Borel définies sur $F$;
  \item $j_*: X_*(T_1)_{\bar{F}} \otimes \Q \rightiso X_*(T_2)_{\bar{F}} \otimes \Q$, son transposé est noté $j^*$;
  \item il existe des bijections $\check{\tau}: \check{\Sigma}_1 \to \check{\Sigma}_2$, $\tau: \Sigma_2 \to \Sigma_1$ et des fonctions $\check{b}: \check{\Sigma}_1 \to \Q_{>0} \cap \Z_p^\times$, $b: \Sigma_2 \to \Q_{>0} \cap \Z_p^\times$, telles que
  \begin{itemize}
    \item $\tau$ s'identifie à $\check{\tau}^{-1}$ via les bijections naturelles $\Sigma_i \simeq \check{\Sigma}_i$, $i=1,2$;
    \item $j^*$ et $j_*$ sont $\Gamma_F$-équivariants;
    \item on a $j_*(\check{\alpha}_1) = \check{b}(\check{\alpha}_1) \check{\tau}(\check{\alpha}_1)$ et $j^*(\alpha_2) = b(\alpha_2) \tau(\alpha_2)$ pour tout $\check{\alpha}_1 \in \check{\Sigma}_1$ et tout $\alpha_2 \in \Sigma_2$.
  \end{itemize}
\end{itemize}

À l'instar du cas standard, ces données définissent
\begin{itemize}
  \item une correspondance de classes de conjugaison géométriques entre $\mathfrak{g}_{1,\text{reg}}(F)$ et $\mathfrak{g}_{2,\text{reg}}(F)$;
  \item les bijections de racines induisent une correspondance de sous-groupes de Lévi semi-standards de $G_1$ et $G_2$.
\end{itemize}

Soient $M_1 \subset G_1$, $M_2 \subset G_2$ des sous-groupes de Lévi semi-standards qui se correspondent, alors $W^{G_1}(M_1) = W^{G_2}(M_2)$ et on a un isomorphisme équivariant $\mathfrak{a}_{M_1} \rightiso \mathfrak{a}_{M_2}$.

Notons $R_1$ (resp. $R_2$) le sous-groupe engendré par $\check{\Sigma}_1$ (resp. $\check{\Sigma}_2$). La même définition s'applique à $M_1$, $M_2$ et leurs coracines. On définit ainsi les groupes $R^{M_1}$, $R^{M_2}$.

Soient $H_1, H_2$ deux sous-groupes commensurables dans un groupe fixé, adoptons la convention $[H_1 : H_2] := [H_1 : H_1 \cap H_2] [H_2 : H_1 \cap H_2]^{-1}$. D'après \cite{Wa09} 3.7 et l'erratum \cite{WaErr}, on définit le coefficient

\begin{gather}\label{eqn:coef-nonstandard-general}
  c^{G_1,G_2}_{M_1,M_2} := \frac{[R_2^{\Gamma_F} : j_*(R_1^{\Gamma_F})]}{[R^{M_2, \Gamma_F} : j_*(R^{M_1,\Gamma_F})]}.
\end{gather}

Signalons que ce coefficient dépend de $j_*$. Énonçons le lemme fondamental pondéré non standard comme une conjecture.
\begin{conjecture}[\cite{Wa09}]\label{prop:LF-nonstandard}
  Soient $Y_1 \in \mathfrak{m}_{1,G_1-\mathrm{reg}}(F)$ et $Y_2 \in \mathfrak{m}_{2,G_2-\mathrm{reg}}(F)$ qui se correspondent. Alors
  $$ s^{G_1}_{M_1}(Y_1) = c^{G_1,G_2}_{M_1,M_2} s^{G_2}_{M_2}(Y_2) $$
  où les fonctions stabilisées $s^{G_i}_{M_i}$ ($i=1,2$) sont définies à l'aide des formes quadratiques invariantes sur $\mathfrak{a}_{M_1}$ et $\mathfrak{a}_{M_2}$ qui se correspondent via $\mathfrak{a}_{M_1} \rightiso \mathfrak{a}_{M_2}$.
\end{conjecture}

\paragraph{Passage au quotient}
Pour $i=1,2$, soit $\pi_i: G_i \to \underline{G}_i$ une isogénie centrale; notons $\underline{M}_i := \pi_i(M_i)$. Cela n'affecte pas les algèbres de Lie, les espaces $\mathfrak{a}_{M_i}$ et les groupes de Weyl. La correspondance de classes marche pour $\underline{G}_i, \underline{M}_i$ au lieu de $G_i, M_i$.

\begin{corollary}\label{prop:LF-nonstandard-SO}
  Supposons vérifiée \ref{prop:LF-nonstandard}. Soient $\underline{Y}_1 \in \mathfrak{m}_{1,\underline{G}_1-\mathrm{reg}}(F)$ et $\underline{Y}_2 \in \underline{\mathfrak{m}}_{2,\underline{G}_2-\mathrm{reg}}(F)$ qui se correspondent. Alors
  $$ s^{\underline{G}_1}_{\underline{M}_1}(\underline{Y}_1) = c^{G_1,G_2}_{M_1,M_2} s^{\underline{G}_2}_{\underline{M}_2}(\underline{Y}_2). $$
\end{corollary}
\begin{proof}
  Pour $i=1,2$, soit $Y_i \in \mathfrak{m}_i(F)$ l'élément qui s'envoie sur $\underline{Y}_i$, alors il suffit de noter que $s^{G_i}_{M_i}(Y_i) = s^{\underline{G}_i}_{\underline{M}_i}(\underline{Y}_i)$, ce qui résulte de \cite{Wa09} 5.7.
\end{proof}

La convention suivante sera commode. Soient $G_i, M_i$ ($i=1,2$) tels que $M_i$ est un Lévi de $G_i$ et que les groupes $G_{i,\text{SC}}, M_{i, \text{sc}}$ sont comme dans le lemme fondamental pondéré non standard \ref{prop:LF-nonstandard}. Nous noterons
\begin{gather}\label{eqn:c-SO}
  c^{G_1, G_2}_{M_1, M_2} = c^{G_{1,\text{SC}},G_{2,\text{SC}}}_{M_{1,\text{sc}}, M_{2,\text{sc}}}.
\end{gather}

\paragraph{Un exemple}
Supposons $n \geq 1$. Fixons un ensemble fini $I$, $(n_i)_{i \in I} \in \Z_{\geq 1}^I$, $m \in \Z_{\geq 0}$ et posons $n := m + \sum_{i \in I} n_i$. Posons
\begin{align*}
  G_1 & := \Sp(2n), &
  \underline{G}_2 & := \SO(2n+1),\\
  M_1 & := \prod_{i \in I} \GL(n_i) \times \Sp(2m),&
  \underline{M}_2 & := \prod_{i \in I} \GL(n_i) \times \SO(2m+1),
\end{align*}
Regardons $M_1$ comme un sous-groupe de Lévi de $G_1$ en choisissant un plongement, qui est unique à conjugaison près par $G_1(F)$. Idem pour $\underline{G}_2$ et $\underline{M}_2$. Notons $\pi: G_2 \to \underline{G}_2$ le revêtement simplement connexe, c'est-à-dire $G_2 = \Spin(2n+1)$, et posons $M_2 := \pi^{-1}(\underline{M}_2)$, alors $M_2$ et un sous-groupe de Lévi de $G_2$.

Comme indiqué dans \cite{Li09} \S 8.2, on peut choisir des $F$-tores maximaux déployés $T_1$, $T_2$ dans $M_1$, $M_2$ respectivement, avec $\underline{T}_2 := \pi(T_2)$, de sorte qu'il existe une base $\{e_1, \ldots, e_n\}$ de $X_*(T_1)$ telle que l'on a des identifications
\begin{align*}
  X_*(T_1) & = \bigoplus_{k=1}^n \Z e_k, \\
  X_*(T_2) & = \left\{ \sum_{k=1}^n r_k e_k : \sum_{k=1}^n r_k \equiv 0 \mod 2 \right\},\\
  X_*(\underline{T}_2) & = \bigoplus_{k=1}^n \Z e_k .
\end{align*}
En choisissant des sous-groupes de Borel convenables, les coracines se décrivent comme suit.
\begin{align*}
  \check{\Sigma}_1 & = \{ \pm e_i \pm e_j : 1 \leq i \neq j \leq n \} \sqcup \{\pm e_i : 1 \leq i \leq n \}, \\
  \check{\Sigma}_2 & = \{ \pm e_i \pm e_j : 1 \leq i \neq j \leq n \} \sqcup \{\pm 2e_i : 1 \leq i \leq n \}.
\end{align*}

Alors l'inclusion $X_*(T_2) \hookrightarrow X_*(\underline{T}_2)$ correspond à l'isogénie $\pi: G_2 \to \underline{G}_2$.

On prend $j_* := \identity: X_*(T_1) \otimes \Q \rightiso X_*(T_2) \otimes \Q$. On prend les fonctions $\tau$, $\check{\tau}$, $b$, $\check{b}$ qui fournissent un triplet non standard non ramifié $(G_1, G_2, j_*)$, pour laquelle $M_1$ et $M_2$ sont des sous-groupes de Lévi semi-standards qui se correspondent.

Remarquons que le choix des divers identifications peut affecter $j_*$, mais la correspondance de classes et la correspondance de mesures de Haar sur $\mathfrak{a}_{M_1}$ et $\mathfrak{a}_{M_2}$ ne changent pas.

C'est plus pratique de travailler avec $\underline{G}_2$ et $\underline{M}_2$. La correspondance des classes géométriques entre $\mathfrak{g}_{1,\text{reg}}(F)$ et $\underline{\mathfrak{g}}_{2,\text{reg}}(F)$ est la correspondance par valeurs propres. Pour $\mathfrak{m}_1$ et $\underline{\mathfrak{m}}_2$, c'est la même correspondance pour les facteurs $\syp(2m)$ et $\so(2m+1)$; pour les facteurs $\gl(\cdot)$, c'est la correspondance tautologique.

Le coefficient de cette donnée endoscopique non standard se calculent comme suit.
\begin{proposition}\label{prop:coef-nonstandard}
  Pour cette donnée,
  $$
    c^{G_1,\underline{G}_2}_{M_1, \underline{M}_2} = c^{G_1,G_2}_{M_1,M_2} := \begin{cases}
      1, & \text{si } m \neq 0, \\
      \frac{1}{2}, & \text{si } m=0.
    \end{cases}
  $$
\end{proposition}
\begin{proof}
  Les coracines dans $\check{\Sigma}_1$ et $\check{\Sigma}_2$ engendrent respectivement les réseaux $X_*(T_1)$ et $X_*(T_2)$. Vu les identifications ci-dessus et le choix $j_* = \identity$, on voit que $[R_2^{\Gamma_F} : j_*(R_1^{\Gamma_F})]=\frac{1}{2}$. Soient $M_1, M_2$ des sous-groupes de Lévi semi-standards qui se correspondent. Alors il en est de même pour $[R^{M_2, \Gamma_F} : j_*(R^{M_1,\Gamma_F})]$ sauf si le facteur $\Sp$ (resp. $\SO$) n'apparaît plus dans $M_1$ (resp. $\underline{M}_2$) i.e. sauf si $m=0$, auquel cas on a trivialement $[R^{M_2, \Gamma_F} : j_*(R^{M_1,\Gamma_F})]=1$. Cela permet de conclure.
\end{proof}

\begin{remark}\label{rem:ns-a}
  Explicitons l'isomorphisme $\mathfrak{a}_{M_1} \rightiso \mathfrak{a}_{M_2} = \mathfrak{a}_{\underline{M}_2}$ induit par $j_*$. On peut se limiter au cas $M_1=T_1$ et $M_2=T_2$. On a identifié $T_1$ et $\underline{T}_2$ dans la description de $X_*(T_i)$ ($i=1,2$). Comme $j_* = \identity$, il induit l'application $\identity: \mathfrak{a}_{T_1} \rightiso \mathfrak{a}_{\underline{T}_2}$. Le cas général en découle: si l'on identifie les composantes $\prod_{i \in I} \GL(n_i)$ dans $M_1$ et $\underline{M}_2$, alors $\mathfrak{a}_{M_1} \rightiso \mathfrak{a}_{\underline{M}_2}$ est encore l'identité.
\end{remark}

\section{Descente des données endoscopiques}\label{sec:descente-donnees}
Le formalisme est celui du \S\ref{sec:int-orb}, mais ici les revêtements métaplectiques ne nous concernent pas. Désormais, fixons des éléments $\eta \in M(F)_\text{ss}$, $\epsilon \in M^\Endo(F)_\text{ss}$ qui se correspondent. Supposons de plus que $\eta$ et $\epsilon$ sont d'ordre finis premiers à $p$. Nous nous placerons sous l'une des deux hypothèses suivantes.

\begin{itemize}
  \item[\textbf{(A)}] $M^\Endo_\epsilon$ est quasi-déployé.
  \item[\textbf{(B)}] $M_\eta$ et $M^\Endo_\epsilon$ sont non ramifiés.
\end{itemize}

L'hypothèse (B) est plus forte que (A). En tout cas, écrivons $\eta = ((\eta_i)_{i \in I}, \eta^\flat)$ et $\epsilon = ((\epsilon_i)_{i \in I}, \epsilon^\flat)$. Quitte à conjuguer $\eta$ et $\epsilon$, on peut supposer que $\eta_i=\epsilon_i$ pour tout $i \in I$.

\subsection{Paramétrage}\label{sec:parametrage}
Plaçons-nous sous l'hypothèse (A). Rappelons la paramétrisation des classes de conjugaison semi-simples dans \cite{Li09} \S 3 et \S 7.1. À la classe $\mathcal{O}^M(\eta)$ sont associées $(\mathcal{O}^{\GL(n_i)}(\eta_i))_{i \in I}$ et les données
\begin{itemize}
  \item $K$: une $F$-algèbre étale de dimension finie;
  \item $\tau:$ une involution non triviale de $K$, qui est déterminée par la sous-algèbre fixée $K^\# := K^{\tau=\identity}$;
  \item $a \in K^\times$ est tel que $N_{K/K^\#}(a) := a\tau(a)=1$;
  \item $(W_K,h_K)$: une $(K,K^\#)$-forme anti-hermitienne, où on suppose que $W_K$ est un $K$-module fidèle;
  \item $(W_\pm,\angles{|}_\pm)$: deux $F$-espaces symplectiques.
\end{itemize}
Ces données sont soumises à la condition $\dim_F W_K + \dim_F W_+ + \dim_F W_- = 2m$. Elles sont uniques modulo la notion d'équivalence évidente. L'élément $\eta$ se réalise comme l'opérateur $(w \mapsto aw, +\identity,-\identity)$ dans l'espace $W_K \oplus W_+ \oplus W_-$.

Écrivons $\SO(2m'+1)=\SO(V',q')$ et $\SO(2m''+1)=\SO(V'', q'')$ en choisissant des $F$-espaces quadratiques convenables. À $\mathcal{O}^{M^\Endo}(\epsilon)$ sont associées $(\mathcal{O}^{\GL(n_i)}(\epsilon_i))_{i \in I}$ et les données
\begin{itemize}
  \item $K',{K'}^\#$: $F$-algèbres étales comme précédemment;
  \item $a' \in {K'}^\times$ est tel que $N_{K'/K'^\#}(a)=1$;
  \item $(V'_K,h'_K)$: une $(K',K'^\#)$-forme hermitienne, où on suppose que $V'_K$ est un $K'$-module fidèle;
  \item $(V'_\pm, q'_\pm)$: deux $F$-espaces quadratiques;
  \item des objets similaires $(K'',K''^\#)$, $a''$, $(V''_K,h''_K)$, $(V''_\pm, q''_\pm)$.
\end{itemize}
Ces données sont soumises aux conditions
\begin{gather*}
  (\Tr_{K'/F})_*(V'_K,h'_K) \oplus (V'_+,q'_+) \oplus (V'_-,q'_-) \simeq (V',q'), \\
  (\Tr_{K''/F})_*(V''_K,h''_K) \oplus (V''_+,q''_+) \oplus (V''_-,q''_-) \simeq (V'', q''),\\
  \dim V'_+ \equiv \dim V''_+  \equiv 1 \mod 2, \\
  \dim V'_- \equiv \dim V''_-  \equiv 0 \mod 2.
\end{gather*}
Elles sont uniques modulo la notion d'équivalence évidente. Si l'on note $\eta=(\eta',\eta'')$, alors $\eta'$ se réalise comme l'opérateur $(v' \mapsto a'v', +\identity,-\identity)$ dans l'espace $V'_K \oplus V'_+ \oplus V'_-$, et $\eta''$ comme $(v'' \mapsto a''v'', +\identity, -\identity)$ dans $V''_K \oplus V''_+ \oplus V''_-$. La correspondance entre $\eta$ et $\epsilon$ entraîne que
\begin{gather*}
  (K',K'^\#,a') \times (K'',K''^\#, -a'') \simeq (K,K^\#,a),\\
  W_K \simeq V'_K \oplus V''_K \quad \text{comme } K-\text{modules},\\
  \dim_F W_+ + 1 = \dim_F V'_+ + \dim_F V''_- , \\
  \dim_F W_- + 1 = \dim_F V'_- + \dim_F V''_+ ,
\end{gather*}
où le produit de triplets dans la première condition est pris au sens évident.

On sait aussi décrire les commutants connexes.
\begin{align*}
  M_\eta & = \prod_{i \in I} \GL(n_i)_{\eta_i} \times \U_{K/K^\#}(W_K,h_K) \times \Sp(W_+) \times \Sp(W_-), \\
  M^\Endo_\epsilon & = \prod_{i \in I} \GL(n_i)_{\epsilon_i} \\
  & \times \U_{K'/K'^\#}(V'_K,h'_K) \times \U_{K''/K''^\#}(V''_K, h''_K) \\
  & \times \SO(V'_+, q'_+) \times \SO(V''_-,q''_-) \\
  & \times \SO(V''_+, q''_+) \times \SO(V'_-,q'_-).
\end{align*}
D'après nos hypothèses, $\SO(V'_+, q'_+)$ et $\SO(V''_+, q''_+)$ sont déployés.

\begin{definition}\label{def:bar}
  Soit $L$ un $F$-groupe réductif admettant une décomposition
  $$L = L_0 \times \prod_a \SO(2a+1),$$
  où $a$ parcourt des entiers positifs et $L_0$ ne contient aucun facteur direct de type $\SO$ impair déployé. On notera
  $$ \bar{L} := L_0 \times \prod_a \Sp(2a). $$

  Les groupes que l'on rencontrera dans cette section sont tous de ce type. Remarquons aussi que $Z_{\widehat{\overline{L}}}^{\Gamma_F} \hookrightarrow Z_{\hat{L}}^{\Gamma_F}$, $Z_{\widehat{\overline{L}}}^{\Gamma_F,0} \rightiso Z_{\hat{L}}^{\Gamma_F,0}$, et $\mathfrak{a}_L = \mathfrak{a}_{\bar{L}}$.
\end{definition}

Avec cette convention, on a
\begin{align*}
  \overline{M^\Endo_\epsilon} & = \prod_{i \in I} \GL(n_i)_{\epsilon_i} \\
  & \times \U_{K'/K'^\#}(V'_K, h'_K) \times \U_{K''/K''^\#}(V''_K, h''_K) \\
  & \times \Sp(\overline{V'_+}) \times \SO(V''_-,q''_-) \\
  & \times \Sp(\overline{V''_+}) \times \SO(V'_-,q'_-),
\end{align*}
où $\overline{V'_+}$, $\overline{V''_+}$ ont les bonnes dimensions et sont munis de formes symplectiques.

\subsection{Des nouvelles données endoscopiques}\label{sec:nouvelles-donnees}
Conservons toujours l'hypothèse (A) et les notations précédentes. On munit $M_\eta$ d'une donnée de L-groupe de sorte que le tore maximal $\hat{T}$ dans $\widehat{M_\eta}$ est ce que l'on a fixé dans la donnée de L-groupe pour $M$. Les actions galoisiennes sur $\hat{T}$ héritées de $\widehat{M_\eta}$ et de $\hat{M}$ peuvent différer par un $1$-cocycle $\Gamma_F \to N_{\hat{M}}(\hat{T})$. Toutefois l'inclusion $Z_{\hat{M}} \hookrightarrow Z_{\widehat{M_\eta}}$ est $\Gamma_F$-équivariante et est indépendante de tout choix. On a aussi $\mathfrak{a}_M \hookrightarrow \mathfrak{a}_{M_\eta}$.

Vu \ref{ex:symplectique}, \ref{ex:unitaire} et \ref{rem:non-ell}, $\overline{M^\Endo_\epsilon}$ est un sous-groupe endoscopique de $M_\eta$, ce qui détermine la donnée endoscopique $(\overline{M^\Endo_\epsilon}, \overline{\mathcal{M}}^\Endo_\epsilon, \bar{s}_0,\hat{\xi})$ à isomorphisme près. La donnée endoscopique $(\overline{M^\Endo_\epsilon}, \overline{\mathcal{M}}^\Endo_\epsilon, \bar{s}_0,\hat{\xi})$ est non ramifiée sous l'hypothèse (B).

\begin{lemma}
  Soient $X \in \mathfrak{m}_{\eta,\mathrm{reg}}(F)$, $Y \in \mathfrak{m}^\Endo_{\epsilon,\mathrm{reg}}(F)$. Soit $\bar{Y} \in \overline{\mathfrak{m}^\Endo_{\epsilon,\mathrm{reg}}}(F)$ tel que $Y$ et $\bar{Y}$ se correspondent par valeurs propres.

  Si $\delta = \exp(X)\eta \in M(F)$ et $\gamma = \exp(Y)\epsilon \in M^\Endo(F)$ sont des décompositions de Jordan topologiques, et si $\delta$ correspond à $\gamma$, alors $X$ et $\bar{Y}$ se correspondent via l'endoscopie décrite ci-dessus. 
\end{lemma}
\begin{proof}
  Lorsque $I=\emptyset$, cela est démontré dans \cite{Li09}. Le cas général en découle car l'endoscopie est tautologique en la composante $\GL(n_i)_{\epsilon_i} = \GL(n_i)_{\eta_i}$, pour tout $i \in I$.
\end{proof}

D'après \ref{rem:non-ell}, il existe un Lévi $R$ de $M_\eta$, unique à conjugaison près par $M_\eta(F)$, tel que $\overline{M^\Endo_\epsilon}$ est un groupe endoscopique elliptique de $R$. Rappelons que $\hat{R}$ et $\widehat{M_\eta}$ partagent le même tore maximal $\hat{T}$. À isomorphisme près on peut supposer $\bar{s}_0 \in \hat{T}$, et $\bar{s}_0$ fait encore partie de cette donnée endoscopique elliptique pour $R$.

On a $R \neq M_\eta$ si et seulement si $\U_{K'/K'^\#}(V'_K,h'_K)$, $\U_{K''/K''^\#}(V''_K, h''_K)$, $\SO(V''_-,q''_-)$ ou $\SO(V'_-,q'_-)$ contiennent des facteurs $\GL_E(\cdot)$, où $E$ est une certaine extension finie de $F$. On absorbe ces $\GL_E(\cdot)$ supplémentaires en introduisant les objets $\bar{I} \supset I$, $\bar{K} \subset K$, $\bar{K}' \subset K'$, etc., et on écrit
\begin{align*}
  R & = \prod_{i \in \bar{I}} \GL_{F_i}(n_i)_{\eta_i} \times \U_{\bar{K}/\bar{K}^\#}(\bar{W}_K,\bar{h}_K) \times \Sp(\bar{W}_+) \times \Sp(\bar{W}_-), \\
  \overline{M^\Endo_\epsilon} & = \prod_{i \in \bar{I}} \GL(n_i)_{\epsilon_i} \\
  & \times \U_{\bar{K}'/\bar{K}'^\#}(V'_{\bar{K}}, h'_{\bar{K}}) \times \U_{\bar{K}''/\bar{K}''^\#}(V''_{\bar{K}}, h''_{\bar{K}}) \\
  & \times \Sp(\overline{V'_+}) \times \SO(\bar{V}''_-, \bar{q}''_-) \\
  & \times \Sp(\overline{V''_+}) \times \SO(\bar{V}'_-,\bar{q}'_-),
\end{align*}

Ici $F_i$ est une extension finie de $F$ pour tout $i \in \bar{I}$, $(\bar{K},\bar{K}^\#)$ est une $F$-algèbre étale à involution qui est un facteur direct de $(K,K^\#)$, et $(\bar{W}_K,\bar{h}_K)$ est le facteur direct correspondant de $(W_K,h_K)$; idem pour $(\bar{K}',\bar{K}'^\#)$, $(\bar{K}'',\bar{K}''^\#)$ etc. D'autre part $(\bar{V}''_-, \bar{q}''_-)=(\{0\},0)$ si $(V''_-,q''_-)$ est de dimension $2$ hyperbolique; sinon $(\bar{V}''_-, \bar{q}''_-)=(V''_-,q''_-)$. Idem pour $(\bar{V}'_-,\bar{q}'_-)$. Supposons comme d'habitude que $\eta_i = \epsilon_i$ pour tout $i \in \bar{I}$.

Rappelons que $\bar{s}_0 \in \hat{T}$ est l'élément faisant partie de la donnée endoscopique $(\overline{M^\Endo_\epsilon}, \overline{\mathcal{M}}^\Endo_\epsilon, \bar{s}_0,\hat{\xi})$ pour $M_\eta$. Décrivons-le. Soit $i \in \bar{I}$, écrivons
$$ \GL_{F_i}(n_i)_{\eta_i} = \prod_{j \in J_i} \GL_{F_{ij}}(n_{ij}) $$
où $J_i$ est un ensemble fini; pour tout $j \in J_i$, $F_{ij}$ est une extension finie de $F_i$. Écrivons
$$ J_i = J_i^U \sqcup J_i^+ \sqcup J_i^- $$
tel que $j \in J_i^\pm$ si et seulement si la composante dans $\GL_{F_{ij}}(n_{ij})$ de $\eta_i$ est $\pm 1$, sinon $j \in J_i^U$.

En regardant les descriptions de $R$ et de $\overline{M^\Endo_\epsilon}$, il s'ensuit que l'on peut écrire
$$ \bar{s}_0 = \left( (\bar{s}_0^{ij})_{\substack{i \in \bar{I}\\ j \in J_i}}, \bar{s}_0^U, \bar{s_0}^+, \bar{s_0}^- \right) \in \hat{R} $$
où
\begin{align*}
  \bar{s}_0^U \leftrightarrow & \text{le groupe endoscopique } \U_{\bar{K}'/\bar{K}'^\#}(V'_{\bar{K}}, h'_{\bar{K}}) \times \U_{\bar{K}''/\bar{K}''^\#}(V''_{\bar{K}}, h''_{\bar{K}}) \\
  & \text{pour } \U_{\bar{K}/\bar{K}^\#}(\bar{W}_K, \bar{h}_K),\\
  \bar{s}_0^+ \leftrightarrow & \text{le groupe endoscopique } \Sp(\overline{V'_+}) \times \SO(\bar{V}''_-,\bar{q}''_-) \\
  & \text{pour } \Sp(\bar{W}_+), \\
  \bar{s}_0^- \leftrightarrow & \text{le groupe endoscopique } \Sp(\overline{V''_+}) \times \SO(\bar{V}'_-,\bar{q}'_-) \\
  & \text{pour } \Sp(\bar{W}_-).
\end{align*}

Il y a aussi des restrictions sur $\bar{s}_0^{ij}$ où $i \in \bar{I} \setminus I$: joints avec $(\bar{s}_0^U, \bar{s}_0^+, \bar{s}_0^-)$, ils sont tels que
\begin{itemize}
  \item $\U_{K'/K'^\#}(V'_K, h'_K) \times \U_{K''/K''^\#}(V''_K, h''_K)$ est un groupe endoscopique pour $\U_{K/K^\#}(W_K, h_K)$,
  \item $\Sp(\overline{V'_+}) \times \SO(V''_-,q''_-)$ est un groupe endoscopique pour $\Sp(W_+)$,
  \item $\Sp(\overline{V''_+}) \times \SO(V'_-,q'_-)$ est un groupe endoscopique pour $\Sp(W_-)$.
\end{itemize}

Il reste à traiter les composantes $\bar{s}_0^{ij}$ où $i \in I$ et $j \in J_i$. La seule condition est que $\bar{s}_0^{ij}$ appartient au centre du dual de $\GL_{F_{ij}}(n_{ij})$. Faisons désormais le choix suivant
\begin{gather}\label{eqn:s_0^ij}
  \bar{s}_0^{ij} = \begin{cases}
    1, & \text{si } j \in J_i^+, \\
    -1, & \text{si } j \in J_i^-, \\
    1, & \text{si } j \in J_i^U.
\end{cases}
\end{gather}
Le symbole ``$-1$'' a un sens car $F_{ij}=F$ lorsque $j \in J_i^\pm$, auquel cas le dual de $\GL_{F_{ij}}(n_{ij})$ est $\GL(n_{ij},\C)$. Ce choix des $\bar{s}_0^{ij}$ sera justifié dans \ref{prop:compatibilite-bar}

\subsection{Rapport avec $\mathcal{E}_{M^\Endo}(\tilde{G})$}
Dans cette sous-section, nous nous plaçons sous l'hypothèse (B). Les groupes $M^\Endo_\epsilon$ et $M_\eta$ sont tous non ramifiés, donc les données endoscopiques décrites dans le  \S\ref{sec:nouvelles-donnees} sont non ramifiées.

Rappelons $s_0$ est un élément dans $\dualmeta{M}$ qui détermine l'endoscopie elliptique pour $\tilde{M}$, dont $M^\Endo$ est le groupe endoscopique associé. On peut supposer que
$$ s_0 = ((1)_{i \in I}, s_0^\flat) \in \prod_{i \in I} \GL(n_i,\C) \times \Sp(2m,\C). $$

Fixons maintenant $s = s_0 t \in \mathcal{E}_{M^\Endo}(\tilde{G})$, où $t \in Z_{\dualmeta{M}}^\Gred/Z_{\dualmeta{G}}^\Gred$ est de la forme
$$ t = ((t_i)_{i \in I}, 1) \in \prod_{i \in I} \{\pm 1\} \times \Sp(2m,\C), $$
comme dans \ref{prop:finitude-s}, auquel est associée une décomposition $I = I' \sqcup I''$ telle que $i \in I'$ (resp. $i \in I'')$ si $t_i=+1$ (resp. $t_i=-1$).

La recette du \S\ref{sec:parametrage} marche aussi pour $\eta \in G(F)$. Nous nous contentons de décrire $G_\eta$:
$$ G_\eta = \U_{\mathcal{K}/\mathcal{K}^\#}(W_\mathcal{K}, h_\mathcal{K}) \times \Sp(\mathcal{W}_+) \times \Sp(\mathcal{W}_-) $$
avec des notations compréhensibles. Il contient $M_\eta$ comme un sous-groupe de Lévi.

D'autre part, introduisons $\epsilon[s] \in M^\Endo(F)$, c'est encore d'ordre fini premier à $p$ et on a $M^\Endo_\epsilon = M^\Endo_{\epsilon[s]}$. D'après \ref{prop:torsion-correspondance}, $\epsilon[s]$ et $\eta$ se correspondent pour la donnée endoscopique déterminée par $s$. Reprenons la recette du \S\ref{sec:parametrage} pour écrire
\begin{align*}
  G[s]_{\epsilon[s]} & = \U_{\mathcal{K}'/\mathcal{K}'^\#}(V'_\mathcal{K}, h'_\mathcal{K}) \times \U_{\mathcal{K}''/\mathcal{K}''^\#}(V''_\mathcal{K},h''_\mathcal{K}) \\
  & \times \SO(\mathcal{V}'_+, \mathcal{Q}'_+) \times \SO(\mathcal{V}''_-, \mathcal{Q}''_-) \\
  & \times \SO(\mathcal{V}''_+, \mathcal{Q}''_+) \times \SO(\mathcal{V}'_-, \mathcal{Q}'_-).
\end{align*}
Et aussi
\begin{align*}
  \overline{G[s]_{\epsilon[s]}} & = \U_{\mathcal{K}'/\mathcal{K}'^\#}(V'_\mathcal{K}, h'_\mathcal{K}) \times \U_{\mathcal{K}''/\mathcal{K}''^\#}(V''_\mathcal{K},h''_\mathcal{K}) \\
  & \times \Sp(\overline{\mathcal{V}'_+}) \times \SO(\mathcal{V}''_-, \mathcal{Q}''_-) \\
  & \times \Sp(\overline{\mathcal{V}''_+}) \times \SO(\mathcal{V}'_-, \mathcal{Q}'_-).
\end{align*}
C'est un groupe endoscopique non ramifié de $G_\eta$ et il contient $\overline{M^\Endo_\epsilon}$ comme un sous-groupe de Lévi, l'inclusion étant bien déterminée à conjugaison près. Il fait partie d'une donnée endoscopique non ramifiée $(\overline{G[s]_{\epsilon[s]}}, \overline{\mathcal{G}[s]_{\epsilon[s]}}, \bar{s}, \hat{\xi})$ qui est unique à isomorphisme près. Indiquons que la condition \eqref{eqn:s_0^ij} sur $\bar{s}$ est vide dans ce cas.

On a déjà décrit $\bar{s}$ et $\bar{s}_0$ dans le \S\ref{sec:nouvelles-donnees}. Nous nous proposons d'exprimer $\bar{s}$ en termes de $t$ et $\bar{s_0}$. Rappelons que les données de L-groupes pour $M$ et $M_\eta$ sont choisies de sorte que $\hat{M}$ et $\widehat{M_\eta}$ partagent le même tore maximal $\hat{T}$, qui fait partie de la paire de Borel $\Gamma_F$-stable dans les données de L-groupe. L'inclusion $Z_{\hat{M}} \hookrightarrow Z_{\widehat{M_\eta}}$ est bien définie et $\Gamma_F$-équivariante. Faisons la même construction pour $G$ et $G_\eta$ par rapport au même tore maximal $\hat{T}$ de façon compatible. On a
\begin{gather}
  \bar{s}_0, \bar{s} \in \hat{T}.
\end{gather}

On obtient ainsi l'homomorphisme
\begin{gather}
  \tau: Z_{\dualmeta{M}}^\Gred/Z_{\dualmeta{G}}^\Gred \to Z_{\hat{M}}/Z_{\hat{G}} \to Z_{\widehat{M_\eta}}^{\Gamma_F}/Z_{\widehat{G_\eta}}^{\Gamma_F}.
\end{gather}

\begin{lemma}\label{prop:compatibilite-bar}
  Quitte à changer $\bar{s}$ par un isomorphisme des données endoscopiques, on a $\bar{s} = \bar{s}_0 \tau(t)$.
\end{lemma}
\begin{proof}
  La torsion $\epsilon \mapsto \epsilon[s]$ n'affecte que les composantes $\GL(n_i)$ de $M^\Endo$. En comparant la paramétrisation de $\epsilon[s]$ dans $M^\Endo$ et dans $G[s]$, on voit que l'inclusion $\overline{M_\epsilon^\Endo} \hookrightarrow \overline{G[s]_{\epsilon[s]}}$ est de la forme
  \begin{align*}
    \U_{K'/K'^\#}(V'_K, h'_K) & \hookrightarrow \U_{\mathcal{K}'/\mathcal{K}'^\#}(V'_\mathcal{K}, h'_\mathcal{K}), \\
    \U_{K''/K''^\#}(V''_K, h''_K) & \hookrightarrow \U_{\mathcal{K}''/\mathcal{K}''^\#}(V''_\mathcal{K}, h''_\mathcal{K}), \\
    \Sp(\overline{V'_+}) & \hookrightarrow \Sp(\overline{\mathcal{V}'_+}), \\
    \SO(V''_-, q''_-) & \hookrightarrow \SO(\mathcal{V}''_-, \mathcal{Q}''_-), \\
    \Sp(\overline{V''_+}) & \hookrightarrow \Sp(\overline{\mathcal{V}''_+}),\\
    \SO(V'_-, q'_-) & \hookrightarrow \SO(\mathcal{V}'_-, \mathcal{Q}'_-).
  \end{align*}
  Pour tout $i \in \bar{I} \setminus I$, ces flèches déterminent aussi les inclusions des composantes $\GL_{F_{ij}}(n_{ij})$ dans $\overline{G[s]_{\epsilon[s]}}$ . Décomposons $\hat{T}$ et $\bar{s}$ selon la décomposition de $R$
  \begin{align*}
    \hat{T} &= \prod_{\substack{i \in \bar{I},\\ j \in J_i}} \hat{T}^{ij} \times \hat{T}^U \times \hat{T}^+ \times \hat{T}^- ; \\
    \bar{s} &= ((\bar{s}^{ij})_{i,j}, \bar{s}^U, \bar{s}^+, \bar{s}^-).
  \end{align*}

  Les inclusions précédentes affirment que l'on peut supposer que, quitte à changer la donnée endoscopique de $G_\eta$ par un isomorphisme, on a
  \begin{gather*}
    \bar{s}^U = \bar{s}_0^U, \; \bar{s}^+ = \bar{s}_0^+,\; \bar{s}^- = \bar{s}_0^-, \\
    \bar{s}^{ij} = \bar{s}_0^{ij}, \quad \text{si } i \in \bar{I} \setminus I.
  \end{gather*}

  Ce sont exactement les composantes correspondantes de $\bar{s}_0\tau(t)$. Il reste donc à considérer les composantes dans $\hat{T}^{ij}$ où $i \in I$. Écrivons $\tau(t) = ((\tau(t)^{ij})_{i,j},1,1,1)$ suivant la décomposition ci-dessus de $\hat{T}$.

  Pour tout $(i,j)$, $\GL_{F_{ij}}(n_{ij})$ se plonge dans un unique facteur dans la décomposition de $\overline{G[s]_{\epsilon[s]}}$ selon $i$ et la restriction de $\epsilon[s]$ sur $\GL_{F_{ij}}(n_{ij})$. Soit $i \in I'$ , alors $\tau(t)^{ij}=1$ pour tout $j \in J_i$. Rappelons aussi la définition \eqref{eqn:s_0^ij}. Sur le facteur $\GL_{F_{ij}}(n_{ij})$, on a:
  \begin{center}\begin{tabular}{c|c|c|c|c|c}
    $j$ & $\epsilon$ & $\epsilon[s]$ & $\bar{s}_0^{ij}$ & inclusion & $\bar{s}^{ij}$ \\ \hline
    $J_i^+$ & $1$ & $1$ & $1$ & $\GL_{F_{ij}}(n_{ij}) \hookrightarrow \Sp(\overline{\mathcal{V}'_+})$ & $1$ \\
    $J_i^-$ & $-1$ & $-1$ & $-1$ & $\GL_{F_{ij}}(n_{ij}) \hookrightarrow \SO(\mathcal{V}'_-, \mathcal{Q}'_-)$ & $-1$ \\
    $J_i^U$ & $\neq \pm 1$ & $\neq \pm 1$ & 1 & $\GL_{F_{ij}}(n_{ij}) \hookrightarrow \U_{\mathcal{K}'/\mathcal{K}'^\#}(V'_\mathcal{K}, h'_\mathcal{K})$ & $1$
  \end{tabular}\end{center}

  Soit $i \in I''$, alors $\tau(t)^{ij}=-1$ pour tout $j \in J_i$. Sur le facteur $\GL_{F_{ij}}(n_{ij})$, on a:
  \begin{center}\begin{tabular}{c|c|c|c|c|c}
    $j$ & $\epsilon$ & $\epsilon[s]$ & $\bar{s}_0^{ij}$ & inclusion & $\bar{s}^{ij}$ \\ \hline
    $J_i^+$ & $1$ & $-1$ & $1$ & $\GL_{F_{ij}}(n_{ij}) \hookrightarrow \SO(\mathcal{V}''_-,\mathcal{Q}''_-)$ & $-1$ \\
    $J_i^-$ & $-1$ & $1$ & $-1$ & $\GL_{F_{ij}}(n_{ij}) \hookrightarrow \Sp(\overline{\mathcal{V}''_+})$ & $1$ \\
    $J_i^U$ & $\neq \pm 1$ & $\neq \mp 1$ & 1 & $\GL_{F_{ij}}(n_{ij}) \hookrightarrow \U_{\mathcal{K}''/\mathcal{K}''^\#}(V''_\mathcal{K}, h''_\mathcal{K})$ & $-1$
  \end{tabular}\end{center}
  où la colonne $\bar{s}^{ij}$ signifie un choix loisible des $\bar{s}^{ij} \in \hat{T}^{ij}$, ce que l'on fixe.

  On voit ainsi qu'en tout cas, on a $\bar{s}_0^{ij}\tau(t)^{ij}=\bar{s}^{ij}$.
\end{proof}

On peut aussi regarder $\tau(t)$ comme un élément dans $Z_{\hat{R}}^{\Gamma_F}/Z_{\widehat{G_\eta}}^{\Gamma_F}$. Ledit lemme permet d'appliquer la construction d'Arthur dans la situation
$$\xymatrix{
  & & G_\eta \\
  \overline{M^\Endo_\epsilon} \ar@{--}[rr]^{\text{endo.ell}} & & R \ar@{^{(}->}[u]
}$$
avec $\bar{s} = \bar{s}_0\tau(t) \in \bar{s}_0 Z_{\hat{R}}^{\Gamma_F}/Z_{\widehat{G_\eta}}^{\Gamma_F}$. On construit ainsi une donnée endoscopique non ramifiée pour $G_\eta$ déterminée par le groupe endoscopique $G_\eta[\bar{s}]$, qui est éventuellement non elliptique.

\begin{lemma}\label{prop:compatibilite-pousser}
  La construction d'Arthur utilisant $\bar{s}$ et la descente en $(\epsilon,\eta)$ fournissent la même donnée endoscopique pour $G_\eta$: on a $G_\eta[\bar{s}] \simeq \overline{G[s]_{\epsilon[s]}}$.
\end{lemma}
\begin{proof}
  Les groupes en question sont des groupes endoscopiques non ramifiés pour $G_\eta$, un produit direct des groupes considérés dans le \S\ref{sec:ex-endo} à restriction des scalaires près.  De plus, \ref{prop:compatibilite-bar} entraîne qu'ils sont associés au même élément $\bar{s} \in \hat{T}$. Or tous les deux contiennent le même Lévi $\overline{M^\Endo_\epsilon}$, donc sont isomorphes d'après la remarque sur les noyaux anisotropes dans \ref{ex:symplectique}.
\end{proof}

\subsection{Une généralisation}
On aura aussi besoin de considérer le cas général $s \in s_0 Z_{\dualmeta{M}}^\Gred/Z_{\dualmeta{G}}^\Gred$. Montrons comment le ramener au cas $s \in \mathcal{E}_{M^\Endo}(\tilde{G})$.

Rappelons que l'application $\gamma \mapsto \gamma[s]$ sur $M^\Endo(F)$ est définie dans le cas général où $\tilde{G}$ est de type métaplectique et $s = s_0 t \in s_0 Z_{\dualmeta{M}}^\Gred/Z_{\dualmeta{G}}^\Gred$ (voir \ref{def:torsion}). On définit ainsi les éléments $\bar{s}, \bar{s}_0 \in \hat{T}$ en combinant les résultats dans le \S\ref{sec:nouvelles-donnees} et la descente des données endoscopiques pour $\GL(\cdot)$, qui est simple.

On dispose encore de l'homomorphisme
\begin{gather*}
  \tau: Z_{\dualmeta{M}}^\Gred/Z_{\dualmeta{G}}^\Gred \to Z_{\widehat{M_\eta}}^{\Gamma_F}/Z_{\widehat{G_\eta}}^{\Gamma_F}.
\end{gather*}

Ainsi, on formule les analogues de \ref{prop:compatibilite-bar} et \ref{prop:compatibilite-pousser} dans ce cadre, avec les mêmes notations. Le bilan est que les assertions de \ref{prop:compatibilite-bar} et \ref{prop:compatibilite-pousser} sont encore valables dans cette situation.

\begin{proposition}\label{prop:compatibilite-general}
  Quitte à changer $\bar{s}$ par un isomorphisme de données endoscopiques pour $G_\eta$, on a $\bar{s} = \bar{s}_0 \tau(t)$. De plus, la construction d'Arthur utilisant $\bar{s}$ et la descente fournissent la même donnée endoscopique pour $G_\eta$: on a $G_\eta[\bar{s}] \simeq \overline{G[s]_{\epsilon[s]}}$.
\end{proposition}
\begin{proof}
  Procédons pas à pas. Il convient de remarquer que l'élément $z[s]$ est défini dans \ref{def:torsion} suivant le même schéma.
  \begin{enumerate}
    \item Supposons $\tilde{G}$ de type métaplectique et $s \in \mathcal{E}_{M^\Endo}(\tilde{G})$. Alors on se ramène au cas $\tilde{G} = \Mp(W)$, qui est déjà traité, et au cas $\tilde{G} = \GL(n) \times \bmu_8$, qui est simple.

    \item Supposons $\tilde{G}=\Mp(W)$ mais $s \in s_0 Z_{\dualmeta{M}}^\Gred/Z_{\dualmeta{G}}^\Gred$ est quelconque. Alors \ref{prop:endo-Levi} affirme qu'il existe $G_1 \in \mathcal{L}^G(M)$ tel que $G[s]=G_1[s]$ et $s \in \mathcal{E}_{M^\Endo}(\tilde{G}_1)$. On peut aussi supposer que $\eta \in G_1(F)$. L'étape précédente est alors applicable à $\tilde{G}_1$.

    Posons $\bar{s} := \bar{s}_0 \tau(t)$. Après descente en $(\epsilon,\eta)$, on s'est ramené à la situation
    \begin{gather}\label{eqn:diag-generalisation}
    \xymatrix{
      \overline{G[s]_{\epsilon[s]}} \ar@{=}[d] & & & G_\eta \\
      \overline{G_1[s]_{\epsilon[s]}} \ar[r]^{\sim} & G_{1,\eta}[\bar{s}] \ar@{--}[rr]^{\text{endo.}}_{\bar{s}=\bar{s}_0 \tau(t)} & & G_{1,\eta} \ar@{^{(}->}[u]_{\text{Lévi}} \\
      &  \overline{M^\Endo_\epsilon} \ar@{^{(}->}[u]^{\text{Lévi}} \ar@{--}[rr]^{\text{endo.ell.}}_{\bar{s}_0} & & R \ar@{^{(}->}[u]_{\text{Lévi}}
    }.
    \end{gather}

    Montrons que
    \begin{gather}\label{eqn:regularite-descente}
      Z_{\widehat{G_\eta}}(\bar{s}) \subset \widehat{G_{1,\eta}}.
    \end{gather}

    Écrivons
    \begin{align*}
      G & = \Sp(W), \\
      G_1 & = \prod_{k \in I^\natural} \GL(n_k^\natural) \times \Sp(W^\natural), \\
      M & = \prod_{i \in I} \GL(n_i) \times \Sp(W^\flat), \\
      M^\Endo & = \prod_{i \in I} \GL(n_i) \times \SO(2m'+1) \times \SO(2m''+1).
    \end{align*}

    Posons $2n := \dim_F W$, $2m^\natural := \dim_F W^\natural$ et $2m := \dim_F W^\flat$. Soient $i \in I$ et $k \in I^\natural$, écrivons $i \mapsto k$ si $\GL(n_i) \hookrightarrow \GL(n_k^\natural)$. Exprimons $t$ selon la décomposition de $\dualmeta{G}_1$:
    \begin{align*}
      t = ((a_k)_{k \in I^\natural}, t_0^\natural) \in \prod_{k \in I^\natural} \GL(n_k^\natural, \C) \times \Sp(2m^\natural,\C)
    \end{align*}
    où $a_k \in \C^\times \subset \GL(n_k^\natural, \C)$, $a_k \neq \pm 1$ pour tout $k \in I^\natural$ et $a_k \neq a_{k'}^{\pm 1}$ si $k \neq k'$.

    Rappelons que $\bar{s}_0^{ij} \in \{\pm 1\}$ pour tous $i \in I$, $j \in J_i$. Vu la description des commutants des éléments semi-simples dans les groupes classiques, il suffit de montrer que, pour tous $i, i' \in I$, $j \in J_i$, $j' \in J_{i'}$ tels que
    \begin{itemize}
      \item dans la décomposition de $G_\eta$ selon les valeurs propres de $\eta$, $\GL_{F_{ij}}(n_{ij})$ et $\GL_{F_{i'j'}}(n_{i'j'})$ se plongent dans le même facteur direct,
      \item $i \mapsto k$, $i' \mapsto k'$,
    \end{itemize}
    on a $\bar{s}_0^{ij} a_k = \left( \bar{s}_0^{i'j'} a_{k'} \right)^{\pm 1}$ si et seulement si $k=k'$.

    En effet, $\GL_{F_{ij}}(n_{ij})$ et $\GL_{F_{i'j'}}(n_{i'j'})$ se plongent dans le même facteur seulement s'il existe $\bullet \in \{+,-,U\}$ tel que $j \in J_i^\bullet$ et $j' \in J_{i'}^\bullet$. Cela entraîne que $\bar{s}_0^{ij} = \bar{s}_0^{i'j'}$ d'après \eqref{eqn:s_0^ij}. Donc $\bar{s}_0^{ij} a_k = \left( \bar{s}_0^{i'j'} a_{k'} \right)^{\pm 1}$ si et seulement si $a_k = a_{k'}^{\pm 1}$, si et seulement si $k=k'$. On en déduit \eqref{eqn:regularite-descente}.

    Vu \eqref{eqn:diag-generalisation} et \eqref{eqn:regularite-descente}, le groupe endoscopique $\overline{G[s]_{\epsilon[s]}} = \overline{G_1[s]_{\epsilon[s]}} \simeq G_{1,\eta}[\bar{s}]$ pour $G_\eta$ est associé à $\bar{s} = \bar{s}_0 \tau(t)$, d'où l'analogue de \ref{prop:compatibilite-bar}. Montrons que $G_\eta[\bar{s}] \simeq \overline{G[s]_{\epsilon[s]}}$. En effet, ces deux groupes endoscopiques de $G_\eta$ sont tous associés à $\bar{s}$ et ils partagent le Lévi $\overline{M^\Endo_\epsilon}$, donc sont isomorphes (cf. la démonstration de \ref{prop:compatibilite-pousser}). D'où l'analogue de \ref{prop:compatibilite-pousser}.

    \item Le cas général: supposons $\tilde{G}$ de type métaplectique et $s \in s_0 Z_{\dualmeta{M}}^\Gred/Z_{\dualmeta{G}}^\Gred$ quelconque. Écrivons $\tilde{G} = \prod_{k \in K} \GL(n_k) \times \Mp(W)$ où $(n_k)_{k \in K} \in \Z_{\geq 1}^K$. Pour achever la preuve, il suffit de combiner l'étape précédente pour $\Mp(W)$ avec la descente des données endoscopiques pour $\GL(n_k)$ pour chaque $k \in K$.
  \end{enumerate}
\end{proof}

\section{Descente des intégrales orbitales}\label{sec:descente-int}
\subsection{Les fonctions combinatoires}
Pour l'instant, soit $G$ un $F$-groupe réductif connexe quelconque. Soit $M$ un sous-groupe de Lévi. Fixons une forme quadratique définie positive $W^G(M)$-invariante sur $\mathfrak{a}_M$ telle que $\mathfrak{a}_M = \mathfrak{a}^G_M \oplus \mathfrak{a}_G$ est une décomposition orthogonale.

Soit $\eta \in M(F)_\text{ss}$. Posons $\underline{G} := G_\eta$ et $\underline{M} := M_\eta$. Fixons un sous-groupe de Lévi $\underline{R}$ de $\underline{G}$ qui est inclus dans $\underline{M}$. Il existe une inclusion $\mathfrak{a}_M \hookrightarrow \mathfrak{a}_{\underline{R}}$ induite par $A_M \hookrightarrow A_{\underline{R}}$. On peut prolonger la forme quadratique choisie sur $\mathfrak{a}_M$ en une forme quadratique définie positive sur $\mathfrak{a}_{\underline{R}}$, invariante par $W^{\underline{G}}(\underline{R})$, etc.

Notons $\mathfrak{a}^M_{\underline{R}}$ le complément orthogonal de $\mathfrak{a}_M$ dans $\mathfrak{a}_{\underline{R}}$. Idem, on définit l'espace $\mathfrak{a}^G_{\underline{R}}$. Tous ces espaces héritent des formes quadratiques définies positives invariantes.

Soit $\underline{L} \in \mathcal{L}^{\underline{G}}(\underline{R})$. On a une application linéaire canonique $\Sigma: \mathfrak{a}^M_{\underline{R}} \oplus \mathfrak{a}^{\underline{L}}_{\underline{R}} \to \mathfrak{a}^G_{\underline{R}}$; ces espaces sont munis de mesures de Haar grâce aux formes quadratiques choisies précédemment. Suivant Arthur, on définit
$$
  d^G_{\underline{R}}(M, \underline{L}) := \begin{cases}
    \dfrac{\text{la mesure sur } \mathfrak{a}^G_{\underline{R}}}{\Sigma_* \left( \text{la mesure sur } \mathfrak{a}^M_{\underline{R}} \oplus \mathfrak{a}^{\underline{L}}_{\underline{R}}\right)}, & \text{si } \mathfrak{a}^M_{\underline{R}} \oplus \mathfrak{a}^{\underline{L}}_{\underline{R}} = \mathfrak{a}^G_{\underline{R}}, \\
    0, & \text{sinon}.
  \end{cases}
$$

Remarquons que $\mathfrak{a}^M_{\underline{R}} \oplus \mathfrak{a}^{\underline{L}}_{\underline{R}} = \mathfrak{a}^G_{\underline{R}}$ si et seulement si $\mathfrak{a}^G_M \oplus \mathfrak{a}^G_{\underline{L}} = \mathfrak{a}^G_{\underline{R}}$. En effet, il suffit de prendre les compléments orthogonaux dans $\mathfrak{a}^G_{\underline{R}}$.

\subsection{Descente de l'intégrale orbitale pondérée endoscopique}
Conservons les notations introduites dans le \S\ref{sec:descente-donnees}.

\begin{proposition}\label{prop:r-cpt}
  Soit $\gamma \in M^\Endo_\text{G-\text{reg}}(F)$. Si $\gamma$ n'est pas compact, alors $r^{\tilde{G}}_{M^\Endo, K}(\gamma)=0$.
\end{proposition}
\begin{proof}
  La correspondance des classes de conjugaison préserve la compacité, et une classe de conjugaison dans $G(F)$ ne coupe pas $K$ si elle n'est pas compacte. Donc l'assertion découle de la définition de $r^{\tilde{G}}_{M^\Endo, K}(\gamma)$.
\end{proof}

Supposons maintenant $\gamma \in M^\Endo_\text{G-\text{reg}}(F)$ compact avec la décomposition de Jordan topologique
$$ \gamma=\exp(Y)\epsilon, \qquad Y \in \mathfrak{m}^\Endo_{\epsilon,\text{reg}}(F).$$

Comme $\gamma \mapsto r^{\tilde{G}}_{M^\Endo, K}(\gamma)$ est invariante par conjugaison géométrique, on peut supposer que $M^\Endo_\epsilon$ est quasi-déployé d'après \cite{Ko82}.

D'après la correspondance des classes de conjugaison géométriques régulières dans l'endoscopie non standard \S\ref{sec:non-standard} et la définition de $\overline{M^\Endo_\epsilon}$ dans \ref{def:bar}, on déduit une correspondance de classes de conjugaison géométriques entre $\mathfrak{m}^\Endo_{\epsilon, \text{reg}}(F)$ et $\overline{\mathfrak{m}^\Endo_{\epsilon, \text{reg}}}(F)$.

\begin{proposition}[cf. \cite{Wa09} 5.4]\label{prop:descente-endo-1}
  Soit $\epsilon$ comme ci-dessus. 
  \begin{enumerate}
    \item Si $M^\Endo_\epsilon$ n'est pas non ramifié, alors $r^{\tilde{G}}_{M^\Endo, K}(\gamma)=0$.
    \item Si $M^\Endo_\epsilon$ est non ramifié, alors il existe $\eta \in K^M$ qui lui correspond tel que l'on peut choisir $R \subset M_\eta$ comme dans le \S\ref{sec:nouvelles-donnees} qui est non ramifié et
    $$ r^{\tilde{G}}_{M^\Endo, K}(\gamma) = \sum_{L \in \mathcal{L}^{G_\eta}(R)} d^G_R(M,L) r^L_{\overline{M^\Endo_\epsilon}}(\bar{Y}), $$
    où $\bar{Y} \in \overline{\mathfrak{m}^\Endo_{\epsilon, \mathrm{reg}}}(F)$ correspond à $Y$.
  \end{enumerate}
\end{proposition}
\begin{proof}
  Supposons d'abord que $M^\Endo_\epsilon$ n'est pas non ramifié. Si pour chaque $\delta \in M(F)$ correspondant à $\gamma$, la classe $\mathcal{O}^G(\delta)$ ne coupe pas $K$, alors $r^{\tilde{G}}_{M^\Endo, K}(\gamma)=0$ et l'assertion est vérifiée dans ce cas.

  Supposons maintenant qu'il existe $\delta_0 \in M(F)$ correspondant à $\gamma$ tel que $\mathcal{O}^G(\delta_0) \cap K \neq \emptyset$. Montrons que $\mathcal{O}^M(\delta_0) \cap K^M \neq \emptyset$. Il existe $g \in G(F)$ tel que $g \delta_0 g^{-1} \in K$. Prenons $P=MU \in \mathcal{P}^G(M)$, la décomposition d'Iwasawa permet d'écrire $g=kum$ avec $k \in K$, $u \in U(F)$ et $m \in M(F)$. Alors $um\delta_0 m^{-1}u^{-1} \in K$. Il existe $u' \in U(F)$ tel que $um\delta_0 m^{-1} u ^{-1} = u' m\delta_0 m^{-1}$. Comme $M$ est en bonne position relativement à $K$, on a $u' \in K$ et $m\delta_0 m^{-1} \in K^M$.

  On peut donc supposer que $\delta_0 \in K^M$. Notons $\delta_0 = \exp(X_0) \eta$ la décomposition de Jordan topologique, alors $\eta \in K$ et $K_\eta := G_\eta(F) \cap K$ est un sous-groupe hyperspécial de $G_\eta$ d'après \cite{Wa09} 5.3. En particulier, $G_\eta$ est non ramifié. Désignons $\mathfrak{k}_\eta$ le sous-réseau hyperspécial dans $\mathfrak{g}_\eta(F)$ associé à $K_\eta$. Idem, on obtient $K^M_\eta \subset M_\eta(F)$, $\mathfrak{k}^M_\eta \subset \mathfrak{m}_\eta(F)$ jouissant des mêmes propriétés.

  Posons $\Xi^{M_\eta}[\bar{Y}]$ l'ensemble des classes de conjugaison dans $\mathfrak{m}_{\eta,\text{reg}}(F)$ qui correspondent à $\bar{Y}$; posons $\Xi^{M_\eta,K_\eta^M}[\bar{Y}]$ son sous-ensemble des classes coupant $K^M_\eta$. Prenons un ensemble de représentants $\dot{\Xi}^{M_\eta,K_\eta^M}[\bar{Y}]$ dans $\mathfrak{m}_\eta(F)$. Comme dans la démonstration de \cite{Li09} 5.23, l'ensemble des classes de conjugaison dans $M(F)$ qui correspondent à $\gamma$ et coupent $K^M$ a pour ensemble de représentants
  $$ \left\{ \exp(X)\eta : X \in \dot{\Xi}^{M_\eta,K_\eta^M}[\bar{Y}] \right\}.$$

  Rappelons que le Lévi $R \subset M_\eta$ choisi dans le \S\ref{sec:nouvelles-donnees} n'est unique qu'à conjugaison par $M_\eta(F)$ près. Comme $K_\eta \subset G_\eta(F)$ est encore en bonne position relativement à $M_\eta$, on peut prendre un $F$-tore déployé maximal $T_0$ de $G_\eta$ tel que $T_0 \subset M_\eta$ et $K_\eta$ correspond à un sommet hyperspécial dans l'appartement associé à $T_0$ dans l'immeuble de Bruhat-Tits de $G_\eta$. Quitte à conjuguer $R$ par $M_\eta(F)$, on peut supposer que $T_0 \subset R$. Avec ce choix, $K^R := K \cap R(F)$ est un sous-groupe hyperspécial de $R(F)$.

  On définit l'ensemble $\Xi^{R}[\bar{Y}]$ en remplaçant $M_\eta$ par $R$ dans la définition ci-dessus. L'application naturelle $\Xi^{R}[\bar{Y}] \to \Xi^{M_\eta}[\bar{Y}]$ est bijective d'après \cite{Wa09} (4). Prenons un ensemble de représentants $\dot{\Xi}^R[\bar{Y}]$ dans $\mathfrak{r}(F)$. Le bilan est
  $$ r^{\tilde{G}}_{M^\Endo,K}(\exp(Y)\epsilon) = \sum_{X \in \dot{\Xi}^R[\bar{Y}]} \Delta(\exp(Y)\epsilon, \exp(X)\eta)\; r^{\tilde{G}}_{\tilde{M},K}(\exp(X)\eta), $$
  où on regarde $\eta$ comme un élément dans $\tilde{M}$ à l'aide du scindage de $\rev$ au-dessus de $K^M$.

  La descente de facteur de transfert \cite{Li09} 7.23 affirme que
  $$ \Delta(\exp(Y)\epsilon, \exp(X)\eta) = \Delta_{\overline{M^\Endo_\epsilon}, R}(\bar{Y},X) $$
  où le terme à gauche signifie un facteur de transfert pour la donnée endoscopique elliptique pour $R$ construite dans le \S\ref{sec:nouvelles-donnees}, qui n'est pas forcément non ramifiée.

  D'autre part, la descente d'intégrales orbitales pondérées non ramifiées \cite{Wa09} 4.4 s'adapte au cas métaplectique (cf. la preuve de \cite{Li09} 8.10). Cela affirme que
  $$ r^{\tilde{G}}_{\tilde{M},K}(\exp(X)\eta) = \sum_{L \in \mathcal{L}^{G_\eta}(R)} d^G_R(M,L) r^L_{R,K^L} (X) $$
  où $K^L := K_\eta \cap L(F)$, qui est hyperspécial dans $L(F)$. Donc
  \begin{gather}\label{eqn:r-descendu}
    r^{\tilde{G}}_{M^\Endo, K}(\gamma) = \sum_{L \in \mathcal{L}^{G_\eta}(R)} d^G_R(M,L) \left( \sum_{X \in \dot{\Xi}^R[\bar{Y}]} \Delta_{\overline{M^\Endo_\epsilon}, R}(\bar{Y},X) r^L_{R,K^L}(\bar{Y})\right).
  \end{gather}

  Si $M^\Endo_\epsilon$ n'est pas non ramifié, $\overline{M^\Endo_\epsilon}$ ne l'est pas non plus, et le terme dans la parenthèse est nul d'après une variante d'un résultat de Kottwitz (voir \cite{Wa09} pp.154-155). Alors l'assertion est vérifiée dans ce cas-là.

  Il reste à traiter le cas $M^\Endo_\epsilon$ non ramifié. Fixons un sous-groupe hyperspécial $K^\Endo$ de $M^\Endo$. Dans ce cas-là, quitte à remplacer $\gamma$ par un conjugué géométrique, on peut supposer de plus que $\epsilon \in K^\Endo$ d'après \cite{Wa08} 5.3 et \cite{Ko82}. Vu \cite{Li09} 8.9, il existe toujours $\delta \in M(F)$ qui correspond à $\gamma$ avec la décomposition de Jordan $\delta=\exp(X)\eta$ tel que $\eta \in K^M$. Donc la formule \eqref{eqn:r-descendu} est valable. L'hypothèse (B) du \S\ref{sec:descente-donnees} est valable et la donnée endoscopique elliptique $(\overline{M^\Endo_\epsilon},\ldots)$ pour $R$ est donc non ramifiée. Le facteur de transfert descendu $\Delta_{\overline{M^\Endo_\epsilon}, R}$ est normalisé par rapport à $K^R$, toujours d'après \cite{Li09} 7.23. D'où
  $$ \sum_{X \in \dot{\Xi}^R[\bar{Y}]} \Delta_{\overline{M^\Endo_\epsilon}, R}(\bar{Y},X) r^L_{R,K^L}(\bar{Y}) = r^L_{\overline{M^\Endo_\epsilon},K^L}(\bar{Y}), \quad L \in \mathcal{L}^{G_\eta}(R). $$

  L'assertion en résulte en le mettant dans \eqref{eqn:r-descendu}.
\end{proof}

Supposons que $M^\Endo_\epsilon$ est non ramifié et $\eta \in K^M$ est l'élément fourni par \ref{prop:descente-endo-1} tel que $\eta$ correspond à $\epsilon$ et $M_\eta$ est non ramifié. Avec ce choix de $(\epsilon,\eta)$, posons
\begin{gather}\label{eqn:Einst}
  E^\text{inst} := \{ (L,\bar{s}) : L \in \mathcal{L}^{G_\eta}(R), \; d^G_R(M,L) \neq 0, \; \bar{s} \in \mathcal{E}_{\overline{M^\Endo_\epsilon}}(L) \}.
\end{gather}

\begin{corollary}
  Avec les hypothèses précédentes, on a
  $$ r^{\tilde{G}}_{M^\Endo, K}(\gamma) = \sum_{(L,\bar{s}) \in E^\mathrm{inst}} d^G_R(M,L) i_{\overline{M^\Endo_\epsilon}}(L,L[\bar{s}]) s^{L[\bar{s}]}_{\overline{M^\Endo_\epsilon}}(\bar{Y}) $$
  où $\bar{Y} \in \overline{\mathfrak{m}^\Endo_{\epsilon, \mathrm{reg}}}(F)$ correspond à $Y$.
\end{corollary}
\begin{proof}
  Cela résulte aussitôt de \ref{prop:descente-endo-1} et du lemme fondamental pondéré sur les algèbres de Lie \ref{prop:LF-pondere-alg}.
\end{proof}

\subsection{Descente des fonctions stabilisées}
Cette sous-section est parallèle à la précédente. Nous ne répétons plus les notations et hypothèses.

\begin{lemma}\label{prop:s-cpt}
  Si $\gamma \in M^\Endo_{G-\text{reg}}(F)$ n'est pas compact, alors
  $$ \sum_{s \in \mathcal{E}_{M^\Endo}(\tilde{G})} i_{M^\Endo}(\tilde{G},G[s]) \cdot s^{G[s]}_{M^\Endo}(\gamma[s]) = 0. $$
\end{lemma}
\begin{proof}
  En fait, $s^{G[s]}_{M^\Endo}(\gamma[s]) = 0$ pour tout $s$ car $\gamma[s]$ n'est pas compact: cela est prouvé dans \cite{Wa09} 6.1.
\end{proof}

Supposons désormais que $\gamma=\exp(Y)\epsilon$ est compact. Comme remarqué plus haut, on peut supposer $M^\Endo_\epsilon$ quasi-déployé. Soit $s \in \mathcal{E}_{M^\Endo}(\tilde{G})$, alors $\gamma[s] = \exp(Y)\epsilon[s]$ est encore une décomposition de Jordan topologique.

Pour tout $L^\epsilon \in \mathcal{L}^{G[s]_{\epsilon[s]}}(M^\Endo_\epsilon)$, supposé muni d'une donnée de L-groupe , posons
$$ e^{G[s]}_{M^\Endo_\epsilon}(M^\Endo, L^\epsilon) := \begin{cases}
  [Z_{\widehat{M^\Endo}}^{\Gamma_F} \cap Z_{\widehat{L^\epsilon}}^{\Gamma_F} : Z_{\widehat{G[s]}}^{\Gamma_F}]^{-1} d^{G[s]}_{M^\Endo_\epsilon}(M^\Endo, L^\epsilon), & \text{si } d^{G[s]}_{M^\Endo_\epsilon}(M^\Endo, L^\epsilon) \neq 0, \\
  0, & \text{sinon}.
\end{cases}$$

Posons
\begin{gather}\label{eqn:Est}
  E^\text{st} := \{(L^\epsilon,s): s \in \mathcal{E}_{M^\Endo}(\tilde{G}), \; L^\epsilon \in \mathcal{L}^{G[s]_{\epsilon[s]}}(M^\Endo_\epsilon), \; d^{G[s]}_{M^\Endo_\epsilon}(M^\Endo, L^\epsilon) \neq 0 \}.
\end{gather}

\begin{proposition}\label{prop:descente-st-1}
  Conservons les hypothèses précédentes.
  \begin{enumerate}
    \item Si $M^\Endo_\epsilon$ n'est pas non ramifié, alors
      $$ \sum_{s \in \mathcal{E}_{M^\Endo}(\tilde{G})} i_{M^\Endo}(\tilde{G},G[s]) \cdot s^{G[s]}_{M^\Endo}(\gamma[s]) = 0. $$
    \item Si $M^\Endo_\epsilon$ est non ramifié, alors
      $$
        \sum_{s \in \mathcal{E}_{M^\Endo}(\tilde{G})} i_{M^\Endo}(\tilde{G},G[s]) \cdot s^{G[s]}_{M^\Endo}(\gamma[s]) =
        \sum_{(L^\epsilon, s) \in E^\mathrm{st}} i_{M^\Endo}(\tilde{G},G[s]) e^{G[s]}_{M^\Endo_\epsilon}(M^\Endo, L^\epsilon) \cdot s^{L^\epsilon}_{M^\Endo_\epsilon}(Y).
      $$
  \end{enumerate}
\end{proposition}
\begin{proof}
  Il suffit d'appliquer \cite{Wa09} 6.4 aux groupes $M^\Endo \hookrightarrow G[s]$, à l'élément $\exp(Y)\epsilon[s]$, pour tout $s$, et noter que $M^\Endo_\epsilon = M^\Endo_{\epsilon[s]}$.
\end{proof}

\subsection{Un ensemble d'indices}\label{sec:E}
Plaçons-nous sous l'hypothèse (B) du \S\ref{sec:descente-donnees} pour $\eta$ et $\epsilon$. Soit $s = s_0 t \in s_0 Z_{\dualmeta{M}}^\Gred/Z_{\dualmeta{G}}^\Gred$. Prenons $\bar{s} = \bar{s}_0 \tau(t)$ comme dans \ref{prop:compatibilite-general}. Soit $L \in \mathcal{L}^{G_\eta}(R)$. Par abus de notation, on désigne l'image de $\bar{s}$ dans $\bar{s}_0 Z_{\hat{R}}^{\Gamma_F}/Z_{\bar{L}}^{\Gamma_F}$ par le même symbole $\bar{s}$.

Soit $L \in \mathcal{L}^{G_\eta}(R)$, la construction d'Arthur donne un groupe endoscopique $L[\bar{s}]$ de $L$; si $L=R$ alors $L[\bar{s}]=\overline{M^\Endo_\epsilon}$. Supposons fixé un plongement $\overline{M^\Endo_\epsilon} \hookrightarrow G_\eta[\bar{s}]$, alors la construction d'Arthur est compatible aux sous-groupes de Lévi au sens que l'on peut regarder $L \mapsto L[\bar{s}]$ comme une application $\mathcal{L}^{G_\eta}(R) \to \mathcal{L}^{G_\eta[\bar{s}]}(\overline{M^\Endo_\epsilon})$: le dual de $L[\bar{s}]$ s'identifie au commutant dans $\widehat{G_\eta[\bar{s}]}$ de l'image de $Z_{\hat{L}}^{\Gamma_F, 0}$ dans $Z_{\widehat{\overline{M^\Endo_\epsilon}}}^{\Gamma_F, 0}$. D'après \ref{prop:compatibilite-general}, on obtient
$$ \mathcal{L}^{G_\eta}(R) \to \mathcal{L}^{\overline{G[s]_{\epsilon[s]}}}(\overline{M^\Endo_\epsilon}). $$

Cette application admet une section canonique, donc est surjective: soit $\overline{L^\epsilon} \in \mathcal{L}^{\overline{G[s]_{\epsilon[s]}}}(\overline{M^\Endo_\epsilon})$, on définit $L \in \mathcal{L}^{G_\eta}(R)$ tel que $\hat{L}$ est le commutant dans $\widehat{G_\eta}$ de l'image de  $Z_{\widehat{\overline{L^\epsilon}}}^{\Gamma_F,0}$ dans $Z_{\hat{R}}^{\Gamma_F, 0}$ (cf. \cite{Wa09} (7)). Par construction, $\overline{L^\epsilon}$ est un groupe endoscopique elliptique de $L$, d'où $\mathfrak{a}_L = \mathfrak{a}_{\overline{L^\epsilon}}$ pour cette section $\overline{L^\epsilon} \mapsto L$.

Signalons que $\overline{G[s]_{\epsilon[s]}}$ est un groupe du type obtenu par la construction \ref{def:bar}. Idem pour ses sous-groupes de Lévi. En inversant cette construction-là, on a une bijection canonique
$$ \mathcal{L}^{\overline{G[s]_{\epsilon[s]}}}(\overline{M^\Endo_\epsilon}) \rightiso \mathcal{L}^{G[s]_{\epsilon[s]}}(M^\Endo_\epsilon). $$

Le composé est noté
\begin{gather}\label{eqn:pousser-L}
\xymatrix{
  \mathcal{L}^{G_\eta}(R) \ar[r] & \mathcal{L}^{\overline{G[s]_{\epsilon[s]}}}(\overline{M^\Endo_\epsilon}) \ar[r]^{\sim} & \mathcal{L}^{G[s]_{\epsilon[s]}}(M^\Endo_\epsilon) \\
  L \ar@{|->}[r] & L[\bar{s}] \ar@{|->}[r] & L^\epsilon .
}\end{gather}
qui admet la section $L^\epsilon \mapsto \overline{L^\epsilon} \mapsto L$ décrite précédemment.

On a le diagramme suivant, dont toute flèche est injective et $\Gamma_F$-équivariante.
\begin{gather}\label{eqn:diagramme-Z}
\xymatrix@1@=1pc{
  & Z_{\dualmeta{G}}^\Gred \ar[ld] \ar'[d] [dd] \ar[rrrr] & & & & Z_{\widehat{G[s]}} \ar[dd] \ar[ld] \\
  Z_{\dualmeta{M}}^\Gred \ar[dd] \ar[rrrr] & & & & Z_{\widehat{M^\Endo}} \ar[dd] & \\
  & Z_{\widehat{G_\eta}} \ar[rr] \ar[ld] & & Z_{\widehat{\overline{G[s]_{\epsilon[s]}}}} \ar'[r] [rr] \ar[ld] & & Z_{\widehat{G[s]_{\epsilon[s]}}} \ar[ld] \\
  Z_{\widehat{M_\eta}} \ar[rr] \ar[dd] & & Z_{\widehat{\overline{M^\Endo_\epsilon}}} \ar[rr] \ar@{=}[dd] & & Z_{\widehat{M^\Endo_\epsilon}} & & & \\
  & & & & & & \\
  Z_{\hat{R}} \ar[rr] & & Z_{\widehat{\overline{M^\Endo_\epsilon}}} & & & & 
}\end{gather}

Montrons qu'il est commutatif. En effet, la commutativité est claire sauf pour les deux grandes faces. Pour prouver que $Z_{\dualmeta{M}}^\Gred \to Z_{\widehat{M_\eta}} \to Z_{\widehat{\overline{M^\Endo_\epsilon}}} \to Z_{\widehat{M^\Endo_\epsilon}}$ coïncide avec $Z_{\dualmeta{M}}^\Gred \to Z_{\widehat{M^\Endo}} \to Z_{\widehat{M^\Endo_\epsilon}}$, il suffit de noter que le facteur $\Sp(2m)$ de $M$ ne contribue pas à $Z_{\dualmeta{M}}^\Gred$; donc on se ramène au cas $M = \prod_{i \in I} \GL(n_i)$, pour lequel l'assertion devient évidente. Idem pour la grande face contenant $Z_{\dualmeta{G}}^\Gred$. Désormais, regardons tous ces centres comme des sous-groupes de $Z_{\widehat{\overline{M^\Endo_\epsilon}}}$.

On a un diagramme commutatif similaire
\begin{gather}\label{eqn:diagramme-a}
\xymatrix@1@=1pc{
  & \mathfrak{a}_G \ar[ld] \ar'[d] [dd] \ar[rrrr] & & & & \mathfrak{a}_{G[s]} \ar[dd] \ar[ld] \\
  \mathfrak{a}_M \ar[dd] \ar[rrrr]^{\sim} & & & & \mathfrak{a}_{M^\Endo} \ar[dd] & \\
  & \mathfrak{a}_{G_\eta} \ar[rr] \ar[ld] & & \mathfrak{a}_{\overline{G[s]_{\epsilon[s]}}} \ar'[r]^{\quad\sim} [rr] \ar[ld] & & \mathfrak{a}_{G[s]_{\epsilon[s]}} \ar[ld] \\
  \mathfrak{a}_{M_\eta} \ar[rr] \ar[dd] & & \mathfrak{a}_{\overline{M^\Endo_\epsilon}} \ar[rr]^{\sim} \ar@{=}[dd] & & \mathfrak{a}_{M^\Endo_\epsilon} & & & \\
  & & & & & & \\
  \mathfrak{a}_R \ar[rr]^{\sim} & & \mathfrak{a}_{\overline{M^\Endo_\epsilon}} & & & & 
}\end{gather}

Rappelons qu'une forme quadratique définie positive $W^G(M)$-invariante sur $\mathfrak{a}_M$ est fixée au début; on la transfère vers $\mathfrak{a}_{M^\Endo}$, puis on le prolonge en une forme sur $\mathfrak{a}_{M^\Endo_\epsilon}$ vérifiant des propriétés similaires. Ainsi, on peut supposer que les formes quadratiques sur chaque espace ci-dessus s'obtiennent par ces flèches. Désormais, regardons tous ces espaces comme sous-espaces de $\mathfrak{a}_{M^\Endo_\epsilon}$.

Pour tout $L \in \mathcal{L}^{G_\eta}(R)$, désignons par $\tau_L$ le composé
\begin{gather}\label{eqn:tau_L}
  \tau_L : \; Z_{\dualmeta{M}}^\Gred/Z_{\dualmeta{G}}^\Gred \stackrel{\tau}{\longrightarrow} Z_{\widehat{M_\eta}}^{\Gamma_F}/Z_{\widehat{G_\eta}}^{\Gamma_F} \longrightarrow Z_{\hat{R}}^{\Gamma_F}/Z_{\hat{L}}^{\Gamma_F}.
\end{gather}

Posons $E^{\natural}$ l'ensemble des paires $(t,L) \in 
\left( Z_{\dualmeta{M}}^\Gred/Z_{\dualmeta{G}}^\Gred \right) \times \mathcal{L}^{G_\eta}(R)$ vérifiant les conditions suivantes:
\begin{description}
  \item[(E1)] $s := s_0 t \in \mathcal{E}_{M^\Endo}(\tilde{G})$;
  \item[(E2)] $\bar{s} := \bar{s}_0 \tau_L(t) \in \mathcal{E}_{\overline{M^\Endo_\epsilon}}(L)$;
  \item[(E3)] $d^G_R(M,L) \neq 0$;
  \item[(E4)] $d^{G[s]}_{M^\Endo_\epsilon}(M^\Endo, L^\epsilon) \neq 0$.
\end{description}

Avec ces notations, on définit deux applications
\begin{align*}
  e^{\text{st}}: & E^{\natural} \longrightarrow E^\text{st}, \\
  & (t,L) \longmapsto (L^\epsilon, s); \\
  e^{\text{inst}}: & E^{\natural} \longrightarrow E^\text{inst}, \\
  & (t,L) \longmapsto (L, \bar{s}). \\
\end{align*}

\begin{lemma}\label{prop:est}
  L'application $e^\mathrm{st}$ est bijective.
\end{lemma}
\begin{proof}
  Soit $(L^\epsilon,s)$ dans l'image de $e^\text{st}$. Par (E2), on a $\mathfrak{a}_L=\mathfrak{a}_{L[\bar{s}]}=\mathfrak{a}_{L^\epsilon}$, ce qui détermine le sous-espace $\mathfrak{a}_L$, donc détermine $L \in \mathcal{L}^{G_\eta}(R)$. D'autre part, l'égalité $s=s_0 t$ détermine $t$. D'où l'injectivité.

  Soit $(L^\epsilon,s) \in E^\text{st}$. Notons $t$ l'élément tel que $s=s_0 t$ et $L$ l'image de $L^\epsilon$ par la section de la suite \eqref{eqn:pousser-L} associée à $s$. Montrons que $(t,L) \in E^{\natural}$. En effet, la définition de $E^\text{st}$ entraîne que
  \begin{gather*}
    s \in \mathcal{E}_{M^\Endo}(\tilde{G}), \\
    \mathfrak{a}^{M^\Endo}_{M^\Endo_\epsilon} \oplus \mathfrak{a}^{L^\epsilon}_{M^\Endo_\epsilon} = \mathfrak{a}^{G[s]}_{M^\Endo_\epsilon}.
  \end{gather*}
  qui ne sont que (E1) et (E4), respectivement.

  Comme $\mathfrak{a}_L = \mathfrak{a}_{L^\epsilon}$ pour la section $L^\epsilon \mapsto L$ de \eqref{eqn:pousser-L}, on voit que (E2) est vérifié. Enfin, \eqref{eqn:diagramme-a} donne un diagramme commutatif
  $$\xymatrix@1@C=1.5pc{
    \mathfrak{a}^G_M \ar@[=][d] & + & \mathfrak{a}^G_L \ar[d] \ar[r] & \mathfrak{a}^G_R \ar[d] \\
    \mathfrak{a}^{G[s]}_{M^\Endo} & + & \mathfrak{a}^{G[s]}_{L^\epsilon} \ar[r] & \mathfrak{a}^{G[s]}_{M^\Endo_\epsilon}.
  }$$
  Les flèches verticales sont des isomorphismes par ellipticité. (E4) équivaut à ce que la somme de la ligne inférieure est directe et la flèche horizontale inférieure est un isomorphisme, donc il en est de même pour la ligne supérieure, ce qui équivaut à (E3). D'où la surjectivité de $e^\text{st}$.
\end{proof}

\begin{lemma}\label{prop:einst}
  L'application $e^\mathrm{inst}$ est surjective, chaque fibre a $[Z_{\hat{L}}^{\Gamma_F} \cap Z_{\dualmeta{M}}^\Gred : Z_{\dualmeta{G}}^\Gred]$ éléments.
\end{lemma}
\begin{proof}
  Soit $(L,\bar{s}) \in E^\text{inst}$. On a $\bar{s} \in \bar{s}_0 Z_{\hat{R}}^{\Gamma_F}/Z_{\hat{L}}^{\Gamma_F}$ car $Z_{\widehat{M_\eta}}^{\Gamma_F} \subset Z_{\hat{R}}^{\Gamma_F}$. On a un homomorphisme canonique
  $$ \tau_L^0 : Z_{\dualmeta{M}}^\Gred/Z_{\dualmeta{G}}^\Gred \to Z_{\hat{R}}^{\Gamma_F,0}/Z_{\hat{L}}^{\Gamma_F,0}. $$

  Comme $\mathfrak{a}^M_R \oplus \mathfrak{a}^L_R = \mathfrak{a}^G_R$, ou ce qui revient au même, $\mathfrak{a}^G_M \oplus \mathfrak{a}^G_L = \mathfrak{a}^G_R$, on déduit par dualité que
  $Z_{\hat{R}}^{\Gamma_F,0} = Z_{\dualmeta{M}}^\Gred Z_{\hat{L}}^{\Gamma_F,0}$, d'où la surjectivité de $\tau_L^0$. Signalons aussi que $\tau_L$ est le composé de $\tau_L^0$ et $Z_{\hat{R}}^{\Gamma_F,0}/Z_{\hat{L}}^{\Gamma_F,0} \to Z_{\hat{R}}^{\Gamma_F}/Z_{\hat{L}}^{\Gamma_F}$.

  D'après \cite{Ar99} Lemma 1.1, l'homomorphisme $Z_{\hat{R}}^{\Gamma_F,0}/Z_{\hat{L}}^{\Gamma_F,0} \to Z_{\hat{R}}^{\Gamma_F}/Z_{\hat{L}}^{\Gamma_F}$ est surjectif. Écrivons $\bar{s}=\bar{s}_0 \bar{t}$ où $\bar{t} \in Z_{\hat{R}}^{\Gamma_F}/Z_{\hat{L}}^{\Gamma_F}$, alors il existe $t \in Z_{\dualmeta{M}}^\Gred/Z_{\dualmeta{G}}^\Gred$ tel que $\bar{t} = \tau_L(t)$. Le nombre de choix de $t$ est égal à $|\Ker(\tau_L)|=[Z_{\hat{L}}^{\Gamma_F} \cap Z_{\dualmeta{M}}^\Gred : Z_{\dualmeta{G}}^\Gred ]$. Il reste à montrer que $(t,L) \in E^{\natural}$.

  Comme $(L,\bar{s}) \in E^\text{inst}$, les conditions (E2) et (E3) sont automatiquement satisfaites. Posons $s=s_0 t$. Cet élément définit une donnée endoscopique pour $\tilde{G}$, éventuellement non elliptique. Contemplons le diagramme \eqref{eqn:diagramme-a}. Montrons que
  \begin{gather}\label{eqn:3-cap}
    \mathfrak{a}^G_{G[s]} \subset \mathfrak{a}^G_M \cap \mathfrak{a}^G_L .
  \end{gather}

  Pour prouver $\mathfrak{a}^G_{G[s]} \subset \mathfrak{a}^G_M$, on utilise $\mathfrak{a}_{G[s]} \subset \mathfrak{a}_{M^\Endo} = \mathfrak{a}_M$. Pour prouver $\mathfrak{a}^G_{G[s]} \subset \mathfrak{a}^G_L$, rappelons que $\mathfrak{a}_L = \mathfrak{a}_{L[\bar{s}]}$ par (E2); or $L[\bar{s}]$ est un Lévi de $\overline{G[s]_{\epsilon[s]}}$ et $\mathfrak{a}_{\overline{G[s]_{\epsilon[s]}}} = \mathfrak{a}_{G[s]_{\epsilon[s]}} \supset \mathfrak{a}_{G[s]}$, d'où l'inclusion cherchée.

  On a aussi $\mathfrak{a}^G_{G[s]} \subset \mathfrak{a}^G_R$. Vu (E3), on déduit $\mathfrak{a}^G_{G[s]}=\{0\}$ de \eqref{eqn:3-cap}, donc (E1) est vérifié. En utilisant l'argument dans la dernière étape de la preuve de \ref{prop:est}, on en déduit (E4).
\end{proof}

\section{Comparaison des coefficients}\label{sec:comparaison}
\subsection{Réduction}
Soit $\gamma \in M^\Endo_{G-\text{reg}}(F)$.

\begin{lemma}\label{prop:LF-pondere-trivial}
  On a $r^{\tilde{G}}_{M^\Endo,K}(\gamma) = \sum_{s \in \mathcal{E}_{M^\Endo}(\tilde{G})} i_{M^\Endo}(\tilde{G},G[s]) s^{G[s]}_{M^\Endo}(\gamma[s]) = 0$ sauf si, à conjugaison géométrique près, $\gamma$ vérifie les conditions suivantes:
  \begin{itemize}
    \item $\gamma$ est compact avec la décomposition de Jordan topologique $\gamma=\exp(Y)\epsilon$,
    \item $M^\Endo_\epsilon$ est non ramifié.
  \end{itemize}

  En particulier, \ref{prop:LF-pondere} est vérifié si $\gamma$ ne vérifie pas les conditions ci-dessus.
\end{lemma}
\begin{proof}
  Cela résulte de \ref{prop:r-cpt}, \ref{prop:s-cpt}, \ref{prop:descente-endo-1} et \ref{prop:descente-st-1}.
\end{proof}

Désormais, supposons que $\gamma$ est compact avec la décomposition de Jordan topologique $\gamma=\exp(Y)\epsilon$ telle que $M^\Endo_\epsilon$ est non ramifié. D'après \ref{prop:descente-endo-1}, on peut choisir $\eta \in K^M$ correspondant à $\epsilon$ tel que $M_\eta$ est non ramifié. On s'est ainsi ramené à la situation dans le \S\ref{sec:E}. Conservons le formalisme-là et posons, pour $(t,L) \in E^{\natural}$,
\begin{align*}
  (L,\bar{s}) & := e^\text{inst}(t,L), \\
  (L^\epsilon,s) & := e^\text{st}(t,L), \\
  c^\text{inst}(t,L) & := d^G_R(M,L) i_{\overline{M^\Endo_\epsilon}}(L,L[\bar{s}]) [Z_{\hat{L}}^{\Gamma_F} \cap Z_{\dualmeta{M}}^\Gred : Z_{\dualmeta{G}}^\Gred]^{-1}, \\
  c^\text{st}(t,L) & := e^{G[s]}_{M^\Endo_\epsilon}(M^\Endo, L^\epsilon) i_{M^\Endo}(\tilde{G},G[s]) \\
  & = d^{G[s]}_{M^\Endo_\epsilon}(M^\Endo,L^\epsilon) [Z_{\widehat{M^\Endo}}^{\Gamma_F} \cap Z_{\widehat{L^\epsilon}}^{\Gamma_F} : Z_{\widehat{G[s]}}^{\Gamma_F}]^{-1} i_{M^\Endo}(\tilde{G},G[s]).
\end{align*}
Indiquons que les intersections des groupes complexes sont prises à l'aide du diagramme \eqref{eqn:diagramme-Z}.

\begin{lemma}\label{prop:developpement-E}
  Supposons vérifié le lemme fondamental pondéré non standard \ref{prop:LF-nonstandard}. Soit $\gamma=\exp(Y)\epsilon$ et $\eta$ comme ci-dessus, où $Y \in \mathfrak{m}^\Endo_\epsilon(F)$. Alors
  \begin{gather*}
    r^{\tilde{G}}_{M^\Endo,K}(\gamma) = \sum_{(t,L) \in E^\natural} c^\mathrm{inst}(t,L) c^{L[\bar{s}],L^\epsilon}_{\overline{M^\Endo_\epsilon}, M^\Endo_\epsilon} \; s^{L^\epsilon}_{M^\Endo_\epsilon}(Y), \\
    \sum_{s \in \mathcal{E}_{M^\Endo}(\tilde{G})} i_{M^\Endo}(\tilde{G},G[s]) s^{G[s]}_{M^\Endo}(\gamma[s]) = \sum_{(t,L) \in E^\natural} c^\mathrm{st}(t,L) \; s^{L^\epsilon}_{M^\Endo_\epsilon}(Y).
  \end{gather*}
\end{lemma}

Ici nous utilisons la convention \eqref{eqn:c-SO} pour le coefficient $c^{L[\bar{s}],L^\epsilon}_{\overline{M^\Endo_\epsilon}, M^\Endo_\epsilon}$, qui sera justifiée dans la preuve.

\begin{proof}
  On combine \ref{prop:descente-endo-1}, \ref{prop:descente-st-1}, \ref{prop:est}, \ref{prop:einst} et le fait, noté dans \cite{Li09} \S 8, que $L[\bar{s}]_\text{SC}$ et $L^\epsilon_{\text{SC}}$ font partie d'un triplet non standard. Ce triplet est un produit direct de triplets tautologiques (avec $j_* = \identity$) ou des triplets de type $(\Sp(2a), \Spin(2a+1), j_*)$, où $j_*$ est choisi comme dans le \S\ref{sec:non-standard}, par lequel $(\overline{M^\Endo_\epsilon})_\text{sc}$ correspond à $(M^\Endo_\epsilon)_\text{sc}$.

  De plus, l'application $\mathfrak{a}^{L[\bar{s}]}_{\overline{M^\Endo_\epsilon}} \rightiso \mathfrak{a}^{L^\epsilon}_{M^\Endo_\epsilon}$ induit par $j_*$ coïncide avec celle obtenue par \eqref{eqn:diagramme-a} (ou plutôt \eqref{eqn:diagramme-La}). En effet, il suffit de comparer \ref{rem:ns-a} et \ref{def:bar}. Cela permet de conclure en appliquant le lemme fondamental pondéré non standard \ref{prop:LF-nonstandard-SO}.
\end{proof}

\begin{lemma}\label{prop:cinstcst}
  Pour tout $(t,L) \in E^\natural$, on a l'égalité
  $$ c^\mathrm{inst}(t,L) c^\mathrm{st}(t,L)^{-1} = [Z_{\widehat{L^\epsilon}}^{\Gamma_F} \cap Z_{\widehat{M^\Endo_\epsilon}}^{\Gamma_F,0} : Z_{\widehat{L[\bar{s}]}}^{\Gamma_F} \cap Z_{\widehat{M^\Endo_\epsilon}}^{\Gamma_F,0}]. $$
\end{lemma}
On démontrera ce résultat combinatoire dans le \S\ref{sec:yoga-centres}.  En admettant \ref{prop:cinstcst} pour l'instant, montrons notre théorème principal.
\begin{proof}[Démonstration de \ref{prop:LF-pondere}]
  D'après \ref{prop:LF-pondere-trivial} et \ref{prop:developpement-E}, il suffit de fixer $(\epsilon,\eta)$ comme ce que l'on a fait dans cette section, et prouver que
  \begin{gather}
    c^\text{inst}(t,L)^{-1} c^\text{st}(t,L) = c^{L[\bar{s}],L^\epsilon}_{\overline{M^\Endo_\epsilon}, M^\Endo_\epsilon}
  \end{gather}
  pour tout $(t,L) \in E^\natural$. D'après \ref{prop:cinstcst}, 
  $$ c^\mathrm{inst}(t,L)^{-1} c^\mathrm{st}(t,L) = [Z_{\widehat{L^\epsilon}}^{\Gamma_F} \cap Z_{\widehat{M^\Endo_\epsilon}}^{\Gamma_F,0} : Z_{\widehat{L[\bar{s}]}}^{\Gamma_F} \cap Z_{\widehat{M^\Endo_\epsilon}}^{\Gamma_F,0}]^{-1}.  $$

  Rappelons que $L[\bar{s}] \in \mathcal{L}^{\overline{G[s]_{\epsilon[s]}}}(\overline{M^\Endo_\epsilon})$ s'écrit sous la forme
  $$\xymatrix{
    L[\bar{s}] \ar@{=}[r] & L_0 & \times & \Sp(2a) & \times & \Sp(2b) \\
    \overline{M^\Endo_\epsilon} \ar@{^{(}->}[u]^{\text{Lévi}} \ar@{=}[r] & M_0 \ar@{^{(}->}[u] & \times & \prod \GL(\cdot) \times \Sp(2a^\flat) \ar@{^{(}->}[u] & \times & \prod \GL(\cdot) \times \Sp(2b^\flat) \ar@{^{(}->}[u]_{\text{Lévi}}
  }$$
  où $a,b,a^\flat,b^\flat \in \Z_{\geq 0}$, et $L_0$, $M_0$ n'ont aucun facteur direct de type $\SO$ impair déployé, cf. la description de $\overline{M^\Endo_\epsilon}$. Selon la définition \ref{def:bar}, on en déduit
  $$\xymatrix{
    L^\epsilon \ar@{=}[r] & L_0 & \times & \SO(2a+1) & \times & \SO(2b+1) \\
    M^\Endo_\epsilon \ar@{^{(}->}[u]^{\text{Lévi}} \ar@{=}[r] & M_0 \ar@{^{(}->}[u] & \times & \prod \GL(\cdot) \times \SO(2a^\flat + 1) \ar@{^{(}->}[u] & \times & \prod \GL(\cdot) \times \SO(2b^\flat + 1) \ar@{^{(}->}[u]_{\text{Lévi}} .
  }$$

  D'où
  \begin{multline*}
    [Z_{\widehat{L^\epsilon}}^{\Gamma_F} \cap Z_{\widehat{M^\Endo_\epsilon}}^{\Gamma_F,0} : Z_{\widehat{L[\bar{s}]}}^{\Gamma_F} \cap Z_{\widehat{M^\Endo_\epsilon}}^{\Gamma_F,0}]^{-1} = \\
    = \left| Z_{\Sp(2a,\C)} \cap (\prod Z_{\GL(\cdot,\C)} \times \{1\}) \right|^{-1} \cdot \left|Z_{\Sp(2b,\C)} \cap (\prod Z_{\GL(\cdot,\C)} \times \{1\}) \right|^{-1}.
  \end{multline*}

  Ces expressions s'évaluent sans difficulté:
  \begin{gather}\label{eqn:Z-calcul}
    \left| Z_{\Sp(2a,\C)} \cap (\prod Z_{\GL(\cdot,\C)} \times \{1\}) \right|^{-1}  = \begin{cases}
    \frac{1}{2}, & \text{si } a > 0, a^\flat=0, \\
    1, & \text{sinon}.
  \end{cases}\end{gather}

  Idem pour $b, b^\flat$ au lieu de $a, a^\flat$. D'autre part, $c^{L[\bar{s}],L^\epsilon}_{\overline{M^\Endo_\epsilon}, M^\Endo_\epsilon}$ est égal au produit $c^{\overline{G_a}, G_a}_{\overline{M_a}, M_a} \cdot c^{\overline{G_b}, G_b}_{\overline{M_b}, M_b}$, où $G_a = \SO(2a+1)$, $M_a = \prod\GL(\cdot) \times \SO(2a^\flat+1) \subset G_a$, et $\overline{G_a}$, $\overline{M_a}$ sont définies selon \ref{def:bar}. Idem pour $b,b^\flat$ au lieu de $a,a^\flat$. Vu \ref{prop:coef-nonstandard}, $c^{\overline{G_a}, G_a}_{\overline{M_a}, M_a}$ admet la même description ci-dessus que \eqref{eqn:Z-calcul}, et idem pour $c^{\overline{G_b}, G_b}_{\overline{M_b}, M_b}$. On conclut que
  $$ [Z_{\widehat{L^\epsilon}}^{\Gamma_F} \cap Z_{\widehat{M^\Endo_\epsilon}}^{\Gamma_F,0} : Z_{\widehat{L[\bar{s}]}}^{\Gamma_F} \cap Z_{\widehat{M^\Endo_\epsilon}}^{\Gamma_F,0}]^{-1} = c^{L[\bar{s}],L^\epsilon}_{\overline{M^\Endo_\epsilon}, M^\Endo_\epsilon}, $$
  ce qui fallait démontrer.
\end{proof}

\subsection{Yoga de centres}\label{sec:yoga-centres}
Nous nous proposons d'établir \ref{prop:cinstcst}. Fixons $(t,L) \in E^\natural$ et conservons les notations précédentes. On a les variantes suivantes des diagrammes \eqref{eqn:diagramme-Z}, \eqref{eqn:diagramme-a}:

\begin{gather}\label{eqn:diagramme-LZ}
\xymatrix@1@=1pc{
  & Z_{\dualmeta{G}}^\Gred \ar[ld] \ar'[d] [dd] \ar[rrrr] & & & & Z_{\widehat{G[s]}} \ar[dd] \ar[ld] \\
  Z_{\dualmeta{M}}^\Gred \ar[dd] \ar[rrrr] & & & & Z_{\widehat{M^\Endo}} \ar[dd] & \\
  & Z_{\widehat{L}} \ar[rr] \ar[ld] & & Z_{\widehat{L[\bar{s}]}} \ar'[r] [rr] \ar[ld] & & Z_{\widehat{L^\epsilon}} \ar[ld] \\
  Z_{\widehat{R}} \ar[rr] & & Z_{\widehat{\overline{M^\Endo_\epsilon}}} \ar[rr] & & Z_{\widehat{M^\Endo_\epsilon}} & & &
}\end{gather}
dont toute flèche est injective et $\Gamma_F$-équivariante, et

\begin{gather}\label{eqn:diagramme-La}
\xymatrix@1@=1pc{
  & \mathfrak{a}_G \ar[ld] \ar'[d] [dd] \ar[rrrr] & & & & \mathfrak{a}_{G[s]} \ar[dd] \ar[ld] \\
  \mathfrak{a}_M \ar[dd] \ar[rrrr] & & & & \mathfrak{a}_{M^\Endo} \ar[dd] & \\
  & \mathfrak{a}_L \ar[rr] \ar[ld] & & \mathfrak{a}_{L[s]} \ar'[r] [rr] \ar[ld] & & \mathfrak{a}_{L^\epsilon} \ar[ld] \\
  \mathfrak{a}_R \ar[rr] & & \mathfrak{a}_{\overline{M^\Endo_\epsilon}} \ar[rr] & & \mathfrak{a}_{M^\Endo_\epsilon} & & &
}\end{gather}
dont toute flèche est injective et toute flèche horizontale est un isomorphisme, car les données endoscopiques en vue sont toutes elliptiques.

\begin{lemma}\label{prop:dd}
  On a l'égalité
  $$ d^G_R(M,L) = d^{G[s]}_{M^\Endo_\epsilon}(M^\Endo, L^\epsilon). $$
\end{lemma}
\begin{proof}
  Contemplons \eqref{eqn:diagramme-La}. La définition de $E^\natural$ fournit le diagramme commutatif
  $$\xymatrix@1@C=1.5pc{
    \mathfrak{a}^G_M \ar@{=}[d] & \oplus & \mathfrak{a}^G_L \ar[r]^{\sim} \ar@{=}[d] & \mathfrak{a}^G_R \ar@{=}[d] \\
    \mathfrak{a}^{G[s]}_{M^\Endo} & \oplus & \mathfrak{a}^{G[s]}_{L^\epsilon} \ar[r]^{\sim} & \mathfrak{a}^{G[s]}_{M^\Endo_\epsilon}.
  }$$
  Les égalités verticales préservent les formes quadratiques choisies (rappelons \eqref{eqn:diagramme-a} et la remarque qui le suit). Les facteurs $d^*_*(\cdot,\cdot)$ sont les rapports des mesures relativement aux flèches horizontales, donc sont égaux.
\end{proof}

Les preuves suivantes reposeront sur \eqref{eqn:diagramme-LZ} et deux faits.
\begin{enumerate}
  \item Soient $H$ un $F$-groupe réductif connexe et $S$ un sous-groupe de Lévi de $H$. Supposons qu'ils sont munis de données de L-groupes compatibles. Alors
  $$ Z_{\hat{S}}^{\Gamma_F} = Z_{\hat{H}}^{\Gamma_F} Z_{\hat{S}}^{\Gamma_F,0}. $$
  \item Soient $a, A, B$ des sous-groupes dans un groupe commutatif tels que $a \subset A$. Alors
  $$ A \cap (Ba) = (A \cap B)a. $$
\end{enumerate}
Le premier est \cite{Ar99} Lemma 1.1 et le deuxième est élémentaire.

\begin{lemma}\label{prop:cinst}
  On a l'égalité
  \begin{gather*}
    \frac{c^\mathrm{inst}(t,L)}{d^G_R(M,L)} = [Z_{\widehat{L[\bar{s}]}}^{\Gamma_F} \cap Z_{\widehat{M^\Endo}}^{\Gamma_F,0} : Z_{\dualmeta{G}}^\Gred]^{-1} .
  \end{gather*}
\end{lemma}
\begin{proof}
  Notons $c_1$ le terme à gauche dans l'assertion. En déroulant les définitions, on voit que
  $$ c_1 = [Z_{\widehat{\overline{M^\Endo_\epsilon}}}^{\Gamma_F} : Z_{\hat{R}}^{\Gamma_F}] [Z_{\widehat{L[\bar{s}]}}^{\Gamma_F} : Z_{\hat{L}}^{\Gamma_F}]^{-1} [Z_{\hat{L}}^{\Gamma_F} \cap Z_{\dualmeta{M}}^\Gred : Z_{\dualmeta{G}}^\Gred]^{-1}. $$

  Comme $Z_{\widehat{\overline{M^\Endo_\epsilon}}}^{\Gamma_F} = Z_{\widehat{L[\bar{s}]}}^{\Gamma_F} Z_{\widehat{\overline{M^\Endo_\epsilon}}}^{\Gamma_F,0}$ et $Z_{\widehat{\overline{M^\Endo_\epsilon}}}^{\Gamma_F,0}=Z_{\hat{R}}^{\Gamma_F,0}$, on a la suite exacte
  $$ 1 \to \frac{Z_{\widehat{L[\bar{s}]}}^{\Gamma_F} \cap Z_{\hat{R}}^{\Gamma_F}}{Z_{\hat{L}}^{\Gamma_F}} \to \frac{Z_{\widehat{L[\bar{s}]}}^{\Gamma_F}}{Z_{\hat{L}}^{\Gamma_F}} \to \frac{Z_{\widehat{\overline{M^\Endo_\epsilon}}}^{\Gamma_F}}{Z_{\hat{R}}^{\Gamma_F}} \to 1 . $$

  Donc
  $$ c_1 = [Z_{\widehat{L[\bar{s}]}}^{\Gamma_F} \cap Z_{\hat{R}}^{\Gamma_F} : Z_{\hat{L}}^{\Gamma_F}]^{-1} [Z_{\hat{L}}^{\Gamma_F} \cap Z_{\dualmeta{M}}^\Gred : Z_{\dualmeta{G}}^\Gred]^{-1}. $$

  On a aussi
  \begin{align*}
    Z_{\hat{R}}^{\Gamma_F} & = Z_{\hat{L}}^{\Gamma_F} Z_{\hat{R}}^{\Gamma_F,0}, \\
    Z_{\hat{R}}^{\Gamma_F,0} & = Z_{\hat{L}}^{\Gamma_F,0} Z_{\dualmeta{M}}^\Gred
  \end{align*}
  où la dernière égalité découle de l'hypothèse $d^G_R(M,L) \neq 0$ et dualité. D'où $Z_{\hat{R}}^{\Gamma_F} = Z_{\hat{L}}^{\Gamma_F} Z_{\dualmeta{M}}^\Gred$, donc
  $$ Z_{\widehat{L[\bar{s}]}}^{\Gamma_F} \cap Z_{\hat{R}}^{\Gamma_F} = Z_{\widehat{L[\bar{s}]}}^{\Gamma_F} \cap \left( Z_{\dualmeta{M}}^\Gred Z_{\hat{L}}^{\Gamma_F}\right) = (Z_{\widehat{L[\bar{s}]}}^{\Gamma_F} \cap Z_{\dualmeta{M}}^\Gred) Z_{\hat{L}}^{\Gamma_F} $$
  car $Z_{\hat{L}}^{\Gamma_F} \subset Z_{\widehat{L[\bar{s}]}}^{\Gamma_F}$. On en déduit
  $$ \frac{Z_{\widehat{L[\bar{s}]}}^{\Gamma_F} \cap Z_{\dualmeta{M}}^\Gred}{Z_{\hat{L}}^{\Gamma_F} \cap Z_{\dualmeta{M}}^\Gred} \rightiso \frac{Z_{\widehat{L[\bar{s}]}}^{\Gamma_F} \cap Z_{\hat{R}}^{\Gamma_F}}{Z_{\hat{L}}^{\Gamma_F}}. $$

  Donc
  \begin{align*}
    c_1 & = [Z_{\widehat{L[\bar{s}]}}^{\Gamma_F} \cap Z_{\dualmeta{M}}^\Gred : Z_{\hat{L}}^{\Gamma_F} \cap Z_{\dualmeta{M}}^\Gred]^{-1} [Z_{\hat{L}}^{\Gamma_F} \cap Z_{\dualmeta{M}}^\Gred : Z_{\dualmeta{G}}^\Gred]^{-1} \\
    & = [Z_{\widehat{L[\bar{s}]}}^{\Gamma_F} \cap Z_{\dualmeta{M}}^\Gred : Z_{\dualmeta{G}}^\Gred]^{-1}.
  \end{align*}

  Il reste à observer que $Z_{\dualmeta{M}}^\Gred = Z_{\widehat{M^\Endo}}^{\Gamma_F,0}$.
\end{proof}

\begin{lemma}\label{prop:cst}
  On a l'égalité
  \begin{gather*}
    \frac{c^\mathrm{st}(t,L)}{d^{G[s]}_{M^\Endo_\epsilon}(M^\Endo, L^\epsilon)} = [Z_{\widehat{L^\epsilon}}^{\Gamma_F} \cap Z_{\widehat{M^\Endo}}^{\Gamma_F,0} : Z_{\dualmeta{G}}^\Gred]^{-1}
  \end{gather*}
\end{lemma}
\begin{proof}
  Notons $c_2$ le terme à gauche dans l'assertion. Déroulons les définitions: $c_2$ est égal à
  $$ [Z_{\widehat{M^\Endo}}^{\Gamma_F} \cap Z_{\widehat{L^\epsilon}}^{\Gamma_F} : Z_{\widehat{G[s]}}^{\Gamma_F}]^{-1} [Z_{\widehat{M^\Endo}}^{\Gamma_F} : Z_{\dualmeta{M}}^\Gred] [Z_{\widehat{G[s]}}^{\Gamma_F} : Z_{\dualmeta{G}}^\Gred]^{-1} = [Z_{\widehat{M^\Endo}}^{\Gamma_F} : Z_{\dualmeta{M}}^\Gred] [Z_{\widehat{M^\Endo}}^{\Gamma_F} \cap Z_{\widehat{L^\epsilon}}^{\Gamma_F} : Z_{\dualmeta{G}}^\Gred]^{-1} . $$

  On a $Z_{\widehat{M^\Endo}}^{\Gamma_F,0}= Z_{\dualmeta{M}}^\Gred$, donc
  \begin{align*}
    c_2 & = [Z_{\widehat{M^\Endo}}^{\Gamma_F} : Z_{\widehat{M^\Endo}}^{\Gamma_F,0}] [Z_{\widehat{M^\Endo}}^{\Gamma_F} \cap Z_{\widehat{L^\epsilon}}^{\Gamma_F} : Z_{\dualmeta{G}}^\Gred]^{-1} \\
    & = [Z_{\widehat{M^\Endo}}^{\Gamma_F} : Z_{\widehat{M^\Endo}}^{\Gamma_F,0}] [Z_{\widehat{M^\Endo}}^{\Gamma_F} \cap Z_{\widehat{L^\epsilon}}^{\Gamma_F} : Z_{\widehat{M^\Endo}}^{\Gamma_F,0} \cap Z_{\widehat{L^\epsilon}}^{\Gamma_F,0}]^{-1} \\
    & \cdot [Z_{\widehat{M^\Endo}}^{\Gamma_F,0} \cap Z_{\widehat{L^\epsilon}}^{\Gamma_F,0} : Z_{\dualmeta{G}}^\Gred]^{-1}.
  \end{align*}

  Montrons que
  \begin{gather}\label{eqn:cst-foo}
    [Z_{\widehat{M^\Endo}}^{\Gamma_F} : Z_{\widehat{M^\Endo}}^{\Gamma_F,0}] [Z_{\widehat{M^\Endo}}^{\Gamma_F} \cap Z_{\widehat{L^\epsilon}}^{\Gamma_F} : Z_{\widehat{M^\Endo}}^{\Gamma_F,0} \cap Z_{\widehat{L^\epsilon}}^{\Gamma_F,0}]^{-1} = [Z_{\widehat{M^\Endo}}^{\Gamma_F,0} \cap Z_{\widehat{L^\epsilon}}^{\Gamma_F} :  Z_{\widehat{M^\Endo}}^{\Gamma_F,0} \cap Z_{\widehat{L^\epsilon}}^{\Gamma_F,0}]^{-1} .
  \end{gather}

  On a
  \begin{align*}
    Z_{\widehat{M^\Endo_\epsilon}}^{\Gamma_F} & = Z_{\widehat{L^\epsilon}}^{\Gamma_F} Z_{\widehat{M^\Endo_\epsilon}}^{\Gamma_F,0}, \\
    Z_{\widehat{M^\Endo_\epsilon}}^{\Gamma_F,0} & = Z_{\widehat{M^\Endo}}^{\Gamma_F,0} Z_{\widehat{L^\epsilon}}^{\Gamma_F,0}
  \end{align*}
  où la dernière égalité résulte de l'hypothèse $d^{G[s]}_{M^\Endo_\epsilon}(M,L^\epsilon) \neq 0$. Par conséquent $Z_{\widehat{M^\Endo_\epsilon}}^{\Gamma_F} = Z_{\widehat{L^\epsilon}}^{\Gamma_F} Z_{\widehat{M^\Endo}}^{\Gamma_F,0}$, donc
  $$ Z_{\widehat{M^\Endo}}^{\Gamma_F} = Z_{\widehat{M^\Endo}}^{\Gamma_F} \cap Z_{\widehat{M^\Endo_\epsilon}}^{\Gamma_F} = Z_{\widehat{M^\Endo}}^{\Gamma_F} \cap \left( Z_{\widehat{L^\epsilon}}^{\Gamma_F} Z_{\widehat{M^\Endo}}^{\Gamma_F,0} \right) = \left( Z_{\widehat{M^\Endo}}^{\Gamma_F} \cap Z_{\widehat{L^\epsilon}}^{\Gamma_F} \right) Z_{\widehat{M^\Endo}}^{\Gamma_F,0}. $$

  D'où la suite exacte
  $$ 1 \to \frac{Z_{\widehat{L^\epsilon}}^{\Gamma_F} \cap Z_{\widehat{M^\Endo}}^{\Gamma_F,0}}{Z_{\widehat{L^\epsilon}}^{\Gamma_F,0} \cap Z_{\widehat{M^\Endo}}^{\Gamma_F,0}} \to \frac{Z_{\widehat{L^\epsilon}}^{\Gamma_F} \cap Z_{\widehat{M^\Endo}}^{\Gamma_F}}{Z_{\widehat{L^\epsilon}}^{\Gamma_F,0} \cap Z_{\widehat{M^\Endo}}^{\Gamma_F,0}} \to \frac{Z_{\widehat{M^\Endo}}^{\Gamma_F}}{Z_{\widehat{M^\Endo}}^{\Gamma_F,0}} \to 1. $$

  On déduit \eqref{eqn:cst-foo} de cette suite. En mettant \eqref{eqn:cst-foo} dans la dernière expression de $c_2$, on obtient l'égalité cherchée.
\end{proof}

\begin{proof}[Démonstration de \ref{prop:cinstcst}]
  Vu \ref{prop:dd}, \ref{prop:cinst} et \ref{prop:cst}, on a
  $$ c^\text{inst}(t,L) c^\text{st}(t,L)^{-1} = [Z_{\widehat{L^\epsilon}}^{\Gamma_F} \cap Z_{\widehat{M^\Endo}}^{\Gamma_F,0} : Z_{\widehat{L[\bar{s}]}}^{\Gamma_F} \cap Z_{\widehat{M^\Endo}}^{\Gamma_F,0}]. $$

  Posons
  \begin{align*}
    A & := Z_{\widehat{L^\epsilon}}^{\Gamma_F} \cap Z_{\widehat{M^\Endo}}^{\Gamma_F,0}, \\
    B & := Z_{\widehat{L[\bar{s}]}}^{\Gamma_F} \cap Z_{\widehat{M^\Endo}}^{\Gamma_F,0}, \\
    C & := Z_{\widehat{L[\bar{s}]}}^{\Gamma_F,0} = Z_{\widehat{L^\epsilon}}^{\Gamma_F,0}.
  \end{align*}

  Rappelons que $d^{G[s]}_{M^\Endo_\epsilon}(M^\Endo,L^\epsilon)$ entraîne que $Z_{\widetilde{M^\Endo_\epsilon}}^{\Gamma_F,0} = Z_{\widehat{M^\Endo}}^{\Gamma_F,0} C$. On a donc
  \begin{align*}
    Z_{\widehat{L^\epsilon}}^{\Gamma_F} \cap Z_{\widehat{M^\Endo_\epsilon}}^{\Gamma_F,0} &= Z_{\widehat{L^\epsilon}}^{\Gamma_F} \cap \left( Z_{\widehat{M^\Endo}}^{\Gamma_F,0} C \right) = \left( Z_{\widehat{L^\epsilon}}^{\Gamma_F} \cap Z_{\widehat{M^\Endo}}^{\Gamma_F,0}  \right) C = AC, \\
    Z_{\widehat{L[\bar{s}]}}^{\Gamma_F} \cap Z_{\widehat{M^\Endo_\epsilon}}^{\Gamma_F,0} &= Z_{\widehat{L[\bar{s}]}}^{\Gamma_F} \cap \left( Z_{\widehat{M^\Endo}}^{\Gamma_F,0} C \right) = \left( Z_{\widehat{L[\bar{s}]}}^{\Gamma_F} \cap Z_{\widehat{M^\Endo}}^{\Gamma_F,0}  \right) C = BC.
  \end{align*}

  D'après ce qui précède et l'inclusion $Z_{\widehat{L[\bar{s}]}}^{\Gamma_F} \hookrightarrow Z_{\widehat{L^\epsilon}}^{\Gamma_F}$,
  $$ A \cap BC = \left( Z_{\widehat{L^\epsilon}}^{\Gamma_F} \cap Z_{\widehat{M^\Endo}}^{\Gamma_F,0} \right) \cap \left( Z_{\widehat{L[\bar{s}]}}^{\Gamma_F} \cap Z_{\widehat{M^\Endo_\epsilon}}^{\Gamma_F,0} \right) = Z_{\widehat{L[\bar{s}]}}^{\Gamma_F} \cap Z_{\widehat{M^\Endo}}^{\Gamma_F,0} = B. $$

  Donc $A/B \simeq AC/BC$. Il en résulte que
  $$ c^\text{inst}(t,L) c^\text{st}(t,L)^{-1} = [A:B] = [AC:BC] = [Z_{\widehat{L^\epsilon}}^{\Gamma_F} \cap Z_{\widehat{M^\Endo_\epsilon}}^{\Gamma_F,0} : Z_{\widehat{L[\bar{s}]}}^{\Gamma_F} \cap Z_{\widehat{M^\Endo_\epsilon}}^{\Gamma_F,0}], $$
  ce qu'il fallait démontrer.
\end{proof}

\bibliographystyle{abbrv-fr}
\bibliography{endoscopie,tf}

\bigskip
\begin{flushleft}
Wen-Wei Li \\
Institut de Mathématiques de Jussieu \\
175 rue du Chevaleret, 75013 Paris  \\
France \\
Adresse électronique: \texttt{wenweili@math.jussieu.fr}
\end{flushleft}

\end{document}